\documentclass{amsart}
\usepackage{graphicx}
\usepackage[a4paper, hmargin={1.5in, 1.5in}, vmargin={1.3in, 1.3in}]{geometry}
\usepackage{hyperref}
\usepackage[T2A]{fontenc}
\usepackage[utf8]{inputenc}	
\usepackage[english]{babel}
\usepackage{mathrsfs}
\usepackage{tikz-cd}
\usepackage{amsmath}
\usepackage{amssymb}

\usepackage{amsthm}
\usepackage{stmaryrd}
\usepackage{soul}
\usepackage[leftmargin=0.5in]{quoting}
\usepackage{tikz-cd}
\usetikzlibrary {arrows.meta,bending} 
\usetikzlibrary {decorations.shapes}
\usepackage{subcaption}
\usepackage{xcolor}
\usepackage{aligned-overset}
\DeclareMathOperator{\Env}{\mathrm{Env}}

\DeclareMathOperator{\Obj}{\mathrm{Obj}}

\DeclareMathOperator{\Double}{\mathrm{D}}

\DeclareMathOperator{\Path}{\mathrm{Path}}
\DeclareMathOperator{\Free}{\mathrm{Free}}
\DeclareMathOperator{\Rel}{\mathrm{Rel}}

\newtheoremstyle{mytheorem}%
{10.0pt plus 2.0pt minus 2.0pt} 
{10.0pt plus 2.0pt minus 2.0pt} 
{\itshape} 
{} 
{\scshape} 
{.} 
{ } 
{} 

\newtheoremstyle{mydefinition}%
{10.0pt plus 2.0pt minus 2.0pt} 
{10.0pt plus 2.0pt minus 2.0pt} 
{} 
{} 
{\scshape} 
{.} 
{ } 
{} 

\newtheoremstyle{myremark}%
{10.0pt plus 2.0pt minus 2.0pt} 
{10.0pt plus 2.0pt minus 2.0pt} 
{} 
{} 
{\scshape} 
{.} 
{ } 
{} 

\theoremstyle{mytheorem}
\newtheorem{thm}{Theorem}[section]
\newtheorem{lem}[thm]{Lemma}

\newtheorem{cor}[thm]{Corollary}
\newtheorem{prop}[thm]{Proposition} 

\theoremstyle{myremark}
\newtheorem{rem}[thm]{Remark}
\newtheorem{conv}[thm]{Convention}
\newtheorem{notat}[thm]{Notation}

\theoremstyle{mydefinition}
\newtheorem{defin}[thm]{Definition}
\newtheorem{ex}[thm]{Example}

\definecolor{Maroon}{rgb}{0.6 0 0}
\definecolor{Prussian}{rgb}{0.05 0 0.6}
\definecolor{Emerald}{rgb}{0 0.5 0.1}

\newcommand{\set}[1]{\lbrace #1 \rbrace}

\newcommand{\source}{\mathfrak{s}}
\newcommand{\target}{\mathfrak{t}}

\newcommand{\N}{\mathbb{N}}
\newcommand{\Z}{\mathbb{Z}}

\newcommand{\Cat}[1]{\mathsf{#1}}
\newcommand{\id}[1][]{\mathrm{id}_{#1}}
\newcommand{\blank}{\raisebox{-2pt}{\text{---}}}

\newcommand{\opp}[1]{{#1}^{\text{op}}}

\newcommand{\Aa}{{\mathscr{A}}}
\newcommand{\Bb}{{\mathscr{B}}}
\newcommand{\Cc}{{\mathscr{C}}}

\newcommand{\Gg}{{\mathscr{G}}}
\newcommand{\Hh}{{\mathscr{H}}}

\newcommand{\Pp}{{\mathscr{P}}}

\newcommand{\Ss}{{\mathscr{S}}}

\title[Structure groupoids of quiver-theoretic 
Yang--Baxter maps]{Structure groupoids of\\
	quiver-theoretic 
	Yang--Baxter maps}
\author{Davide Ferri and Youichi Shibukawa}
\address{%
	\parbox[b]{0.9\linewidth}{(\textit{D.\@ Ferri}) Department of Mathematics ``G.\@ Peano'', University of Turin, via
		Carlo Alberto 10, 10123 Turin, Italy.\\
		Department of Mathematics and Data Science, Vrije Universiteit Brussel, Pleinlaan 2, 1050, Brussels, Belgium.}}
\email{d.ferri@unito.it, Davide.Ferri@vub.be}
\address{%
	\parbox[b]{0.9\linewidth}{(\textit{Y.\@ Shibukawa}) Department of Mathematics, Faculty of Science, Hokkaido University, Sapporo 060-0810, Japan.}}
\email{shibu@math.sci.hokudai.ac.jp}
\begin{document}
	\begin{abstract}
		\noindent Solutions to the quiver-theoretic quantum Yang--Baxter equation are associated with structure categories and structure groupoids. We prove that the structure groupoids of involutive non-degenerate solutions are Garside. This ge\-ne\-ra\-li\-ses a well-known result about the structure groups of set-theoretic solutions, due to Chouraqui. We also construct involutive non-degenerate solutions from sui\-ta\-ble presented categories. We then investigate the case of solutions of principal homogeneous type. Finally, we present some examples of this new class of Garside groupoids.
	\end{abstract}
	\maketitle
	\section{Introduction}
	\noindent The Yang--Baxter Equation (YBE), germinated from the works of Yang \cite{yang} and Baxter \cite{baxter1972partition,baxter1985exactly}, has been studied for a long time in mathematical physics and representation theory, whereas its set-theoretic variant (proposed by Drinfeld \cite{drinfeld2006some}) has recently grown into a major field of research in algebra. A solution to the set-theoretic (quantum) Yang--Baxter equation, called a \textit{set-theoretic Yang--Baxter map}, is a \textit{braided set}---i.e., the datum of a set $X$ and of a map $\sigma\colon X\times X\to X\times X$ satisfying the \textit{braid relation}
	\begin{equation*}
		(\sigma\times\id)(\id\times\sigma)(\sigma\times\id) = (\id\times\sigma)(\sigma\times\id)(\id\times\sigma).
	\end{equation*}
	
	A generalisation is provided by the 
	quiver-theoretic quantum Yang--Baxter Equation (quiver-theoretic YBE). A solution to the quiver-theoretic YBE, called a \textit{quiver-theoretic Yang--Baxter map} (\textit{quiver-theoretic YBM}, or simply \textit{YBM}), is a \textit{braided quiver}---i.e., the datum of a quiver $\Aa$ over a set of vertices $\Lambda$, and of a morphism $\sigma\colon \Aa\otimes\Aa\to \Aa\otimes\Aa$ of quivers over $\Lambda$ satisfying the braid relation
	\begin{equation*}
		(\sigma\otimes\id)(\id\otimes\sigma)(\sigma\otimes\id) = (\id\otimes\sigma)(\sigma\otimes\id)(\id\otimes\sigma)
	\end{equation*}
	in the monoidal category $\Cat{Quiv}_\Lambda$ of quivers over $\Lambda$. Here $\otimes$ denotes the tensor product of quivers over $\Lambda$ (see Definition \ref{def:tensor-prod}), and $\Aa\otimes\Aa$ is understood as the quiver $\Path_2(\Aa)$ consisting of paths of length $2$ on $\Aa$. A thorough study of the quiver-theoretic YBE was initiated by Andruskiewitsch \cite{andruskiewitsch2005quiver}.
	
	A set-theoretic YBM may be regarded as a quiver-theoretic YBM on a quiver $\Aa$ with a single vertex $\lambda$, and with loops on $\lambda$ being in a 1:1 correspondence with the elements of $X$. Much of the ``one-vertex'' theory of the set-theoretic YBE generalises almost \textit{verbatim} to the ``multiple-vertices'' situation of the quiver-theoretic YBE, as we shall see.
	
	A weaker version of the quiver-theoretic YBE, called the \textit{set-theoretic dynamical Yang--Baxter equation}\footnote{The original DYBE, in the context of Lie algebras, was introduced in mathematical physics by Gervais and Neveu \cite{gervais1984novel}, then developed by Felder \cite{felder1995conformal}, Etingof and Schiffmann \cite{etingof2001lectures}.} (\textit{set-theoretic DYBE}), was first introduced in the framework of \textit{dynamical sets} \cite{shibu}. Later, a connection between dynamical sets and quivers was established by Matsumoto and Shimizu \cite{matsumotoshimizu}, although the core idea was already sketched by Matsumoto in \cite[\S5]{matsu}. This provides further motivation for the study of the quiver-theoretic YBE.
	
	Another motivation for studying the quiver-theoretic YBMs resides in the fact that they are \textit{partial solutions} to the set-theoretic YBE. These have already raised the interest of other researchers, such as Chouraqui \cite{chouraquiMbraces, chouraquiPartial}.
	
	Solutions to the quiver-theoretic YBE are associated with \textit{structure categories} and \textit{structure groupoids}, which are the main subjects of this paper.
	
	The structure monoid
	(resp.\@ structure group) of a set-theoretic YBM $\sigma$ on a set $X$ is defined as the monoid
	(resp.\@ the group) generated by $X$ modulo the relations
	$$xy\sim x'y'\text{ for all }x,y, x', y'\in X
	\text{ satisfying } (x',y') = \sigma(x,y).$$
	Analogously, the structure category (resp.\@ structure groupoid) of a YBM $\sigma$ on a quiver $\Aa$ is defined as the category (resp.\@ the groupoid) generated by $\Aa$ modulo the relations
	$$x|y\sim \sigma(x|y)\text{ for all }x|y\in \Path_2(\Aa).$$
	More details and notations about presented categories and groupoids will be recalled later in \S\ref{section:Preliminaries}.
	
	Our viewpoint on the structure groupoid is preeminently Garside-theoretic. Garside theory is an approach to normal forms and the word problem in algebraic structures; it emerged from the work of Garside on braid groups \cite{garside1969braid}, and has been successfully applied to many other algebraic objects. Its interplay with the quiver-theoretic YBE is investigated here for the first time, although the main ideas were anticipated in the work of Dehornoy \textit{et al.\@} \cite{dehornoy2015foundations}.
	Our results allow us to construct a class of examples of Garside groupoids. These objects commonly arise both in algebra and in geometry (see for instance \cite{paris2000fundamental}, for a discussion of Garside groupoids arising from hyperplane arrangements).
	
	In the final section of this paper, we shall investigate the special case of \emph{solutions of principal homogeneous type}. We shall prove in Corollary \ref{ternary_operation_and_left_quasigroup} that a \emph{braiding} on a groupoid of pairs $\Gg$ with a distinguished vertex is, in fact, tantamount to a group structure on the set $\Obj(\Gg)$ of vertices of $\Gg$. This builds on an interpretation of \textit{heaps}, defined by Pr\"ufer \cite{prufer1924theorie} and Baer \cite{baer1929einfuhrung}, as an ``affine notion'' of groups; see also \cite{breaz2024heaps,BRZEZINSKITrussesParagons,brzezinski2024affgebras,wagner1953}. We exploit this class of YBMs to construct examples of Garside groupoids.
	
	We hope that our viewpoint helps further advance Garside theory, by providing a class of concrete examples to work with.
	
	\subsection{Scope and structure of the paper}
	In this paper, we prove that the structure groupoid of an involutive non-degenerate quiver-theoretic YBM is Garside. The approach we adopt is derived from \cite{dehornoy2015foundations}: we merge Chouraqui's method with Dehornoy \textit{et al.}'s investigation of weak RC-systems \cite[\S XIV.2]{dehornoy2015foundations}. Furthermore, we exploit our result to present some new examples of Garside groupoids.
	
	The paper is structured as follows:
	\begin{itemize}
		\item[\textbf{\S\ref{section:Preliminaries}}]\textbf{Preliminaries.}\quad Here we survey some foundational concepts in the theory of quivers, presented categories, and Garside theory, and we set up our notations. We do this for making this paper self-contained. However, the entire content of this section can be found in the monograph \cite{dehornoy2015foundations}, which is going to be one of our main references.
		\item[\textbf{\S\ref{section:cyclic}}]\textbf{Weak RC-systems and other cyclic systems.}\quad Here we describe the notion of a weak RC-system, a weak LC-system and a weak RLC-system, and recall some useful results. We give the definition of the \textit{structure category} of a weak RC-system, and prove that it is Garside under some assumptions. This section is mainly drawn from \cite[\S XIV.2]{dehornoy2015foundations}, although we fix some details.
		\item[\textbf{\S\ref{section:structure}}]\textbf{Quiver-theoretic YBMs and their structure categories.}\quad We recall the definition of quiver-theoretic YBMs, and the \textit{involutivity} and \textit{non-degeneracy} conditions. Then, we define the structure category of a YBM.
		\item[\textbf{\S\ref{section:Garside}}]{\textbf{The interplay between weak RC-systems and YBMs.}}\quad In this section, from every left-non-degenerate involutive YBM, we shall construct a left-non-degenerate weak RC-system. As a converse connection, we shall construct YBMs from a suitable class of presented categories, whose relations contain the RC-law. 
		\item[\textbf{\S\ref{sec:Garside}}] {\textbf{On the Garsideness of the structure category of YBMs.}}\quad We prove that, if $\sigma$ is a non-degenerate involutive YBM, its structure category $\Cc(\sigma)$ is isomorphic to the structure category of a suitable weak RC-system. Under these assumptions, we prove that $\Cc(\sigma)$ is perfect Garside, and we give a description of the Garside family. Moreover, the structure groupoid $\Gg(\sigma)$ of $\sigma$ is the same as the enveloping groupoid of $\Cc(\sigma)$, and $\Cc(\sigma)$ is embedded in $\Gg(\sigma)$, thus making $\Gg(\sigma)$ into a Garside groupoid.
		\item[\textbf{\S\ref{sec:examples}}]{\textbf{Examples of solutions and their structure categories.}}\quad Making use of the construction from \S\ref{sec:converse}, we present several examples of YBMs and Garside categories.
		\item[\textbf{\S\ref{sec:PH-type}}]{\textbf{Solutions of principal homogeneous type.}}\quad We recall the definition of solutions of principal homogeneous (PH) type \cite{matsumotoshimizu,shibukawa2007invariance}. We moreover recall the definition of braided groupoids, of principal homogeneous groupoids, and prove that the category of braided principal homogeneous groupoids with a distinguished vertex is equivalent to the category of groups. We finally describe some examples of structure groupoids of involutive non-degenerate YBMs of principal homogeneous type, which turn out, by the previous discourse, to be examples of Garside groupoids.
	\end{itemize}
	\medskip
	
	\noindent\textbf{Acknowledgements.} The authors are grateful to Tomasz Brzezi\'nski and Marino Gran for precious conversations, which inspired \S\ref{subsec:PH-groupoids}; and to Alessandro Ardizzoni, Andrea Maffei and Francesco Sala for their valuable remarks and comments. Furthermore, the authors wish to thank an anonymous referee for many helpful suggestions.
	
	This project was started by the first author (DF) under the supervision of the second author (YS), as part of his master's thesis research. During the early phase of research, the first author was affiliated to Hokkaido University and to the University of Pisa, and supported by the University of Pisa Scholarship for the Achievement of Credits Abroad and Co-Promoted Thesis, AY.\@ 2022/2023. The first author was later supported by the University of Turin through a PNRR DM 118 scholarship, and by the Vrije Universiteit Brussel Bench Fee for a Joint Doctoral Project, grant number OZR4257, and through the FWO Senior Research Project G004124N.
	
	The second author was supported by 
	JSPS KAKENHI Grant Numbers JP17K05187, JP23K03062.
	
	\section{Preliminaries}\label{section:Preliminaries}
	\noindent In this section, we recall some definitions and set some notations that we are going to use extensively in the rest of this paper. All the definitions and results in \S\ref{subsection:presentations}--\ref{subsection:Garside} are extracted from the monograph \cite{dehornoy2015foundations}. The definitions in \S\ref{subsection:quivers} are presented following quite closely \cite{andruskiewitsch2005quiver}.
	\subsection{Quivers}\label{subsection:quivers}
	A quiver is a directed multigraph with loops. More formally:
	\begin{defin}
		A \textit{quiver $Q$ over $\Lambda$} is the datum of a set $\Lambda\neq\emptyset$, a set $Q$, and two set-theoretic functions $\source,\target\colon Q\to \Lambda$, respectively called the \textit{source} and \textit{target} maps. The elements of $\Lambda$ are called \textit{vertices}, and the elements of $Q$ are called \textit{edges} or \textit{arrows}. 
		
		We shall henceforth say that ``$Q$ is a quiver'', implying that the data of $\Lambda$ and of the maps $\source,\target$ are understood. When we want to highlight that $\source,\target$ are the source and target maps of a certain quiver $Q$, we write $\source_Q, \target_Q$ instead. We denote by $\Obj(Q) = \Lambda$ the set of vertices of $Q$.
		
		A \textit{morphism} $f\colon Q\to Q'$ \textit{of quivers over $\Lambda$} is a set-theoretic map from $Q$ to $Q'$ that preserves the sources and the targets: namely, such that $\source_{Q'}(f(x)) = \source_Q(x)$ and $\target_{Q'}(f(x))=\target_Q(x)$ for all $x\in Q$.\footnote{The reader may have encountered a different definition, which is strictly milder, of morphisms between quivers that are allowed to have different sets of vertices. These are the \textit{weak morphisms of quivers}; see Definition \ref{def:weak-morph}.} The category of quivers over $\Lambda$ is denoted by $\Cat{Quiv}_\Lambda$.
	\end{defin}
	\begin{notat}\label{notat:sources-targets}
		Let $Q$ be a quiver as above. If $x\in Q$ has $\source(x)=\lambda$ and $\target(x)=\mu$, we say that $x$ is an arrow from $\lambda$ to $\mu$, and we write $x\colon\lambda\to\mu$. The set of arrows with source $\lambda$ and target $\mu$ is denoted by $Q(\lambda,\mu)$. The set of arrows with source $\lambda$, and any possible target, is denoted by $Q(\lambda,\Lambda)$. Analogously, $Q(\Lambda,\mu)$ denotes the set of arrows with target $\mu$, and any possible source.
	\end{notat}
	Given a quiver $Q$, it is natural to consider \textit{paths} of arrows in $Q$. These have, in turn, a quiver structure.
	\begin{defin}\label{defin:paths}
		Let $Q$ be a quiver over $\Lambda$. A \textit{path of length $n$} in $Q$ ($n\in \N_{>0}$) is a sequence $x_1|\dots|x_n$ of elements $x_i\in Q$, where the target of each $x_i$ equals the source of $x_{i+1}$ (we say that such a sequence is \textit{composable}). A \textit{path of length $0$} is tantamount to specifying a vertex $\lambda$, and it is also called the \textit{empty path on $\lambda$}. We denote it by $\varepsilon_\lambda$.
		
		The source of $x_1|\dots|x_n$ is, by definition, $\source(x_1)$, and the target is by definition $\target(x_n)$. With these source and target maps, the set $\Path_n(Q)$ of paths of length $n$ in $Q$ is, in turn, a quiver over $\Lambda$.
		
		We denote by $\Path(Q) = \bigcup_{n\ge 0}\Path_n(Q)$ the set of all paths in $Q$, of any possible length. This is also a quiver over $\Lambda$,
		with $\source (x_1|\dots |x_n) = \source(x_1)$, $\target(x_1|\dots| x_n) = \target(x_n)$, and $\source(\varepsilon_\lambda)=\target(\varepsilon_\lambda)=\lambda$. Notice that, if we define the composition of two paths $x_1|\dots|x_n$ and $y_1|\dots|y_m$ as the usual concatenation $x_1|\dots|x_n|y_1|\dots|y_m$ whenever $\target(x_n)=\source(y_1)$, then $\Path(Q)$ becomes a category with set of objects $\Lambda$. For all $\lambda\in \Lambda$, the empty path $\varepsilon_\lambda$ plays the role of the identity on $\lambda$.
	\end{defin} 
	Notice that $\Path(Q)$ has a \textit{graded structure}: the composition of a path of length $n$ with a path of length $m$ lies in $\Path_{n+m}(Q)$. 
	
	The category $\Cat{Quiv}_\Lambda$ of quivers over $\Lambda$ is monoidal, with a monoidal product defined as follows.
	\begin{defin}\label{def:tensor-prod}
		Let $Q, Q'$ be quivers over $\Lambda$. The tensor product $Q\otimes Q'$ is the quiver over $\Lambda$ defined as follows:
		\begin{enumerate}
			\item[\textit{i.}] As a set, $Q\otimes Q'$ is the subset of $Q\times Q'$ consisting of the pairs $(x,y)$ with $\target_Q(x)=\source_{Q'}(y)$.
			\item[\textit{ii.}] The source of $(x,y)$ is defined to be $\source_Q(x)$, and the target is defined to be $\target_{Q'}(y)$.
		\end{enumerate}
	\end{defin}
	It is easy to verify (see \cite{matsumotoshimizu}) that $Q\otimes (Q'\otimes Q'')$ and $(Q\otimes Q')\otimes Q''$ are isomorphic. Therefore, it makes sense to define the tensor power $Q^{\otimes n}$.
	\begin{rem}
		Notice that $Q\otimes Q$ is naturally identified with $\Path_2(Q)$. Analogously, $Q^{\otimes n}$ is naturally identified with $\Path_n(Q)$.
	\end{rem}
	\begin{defin}
		The \textit{opposite} of a quiver $Q$ is the quiver $\bar{Q}$, with same vertices and reverted arrows. The \textit{double} $\Double(Q)$ of a quiver $Q$ over $\Lambda$ is the quiver having $\Lambda$ as a set of vertices and, as a set of arrows, the disjoint union $Q\sqcup \bar{Q}$.
	\end{defin}
	\subsection{Generalities about categories}\label{subsection:categories}
	\noindent Here we report, for clarity, some standard definitions and results about categories. The reader may find the fundamentals in category theory covered in any textbook, such as \cite{borceux, categories-for-the-work}.
	
	In the rest of this paper, we shall consistently consider categories as \textit{algebraic objects}, endowed with an associative \textit{binary composition}. Under this viewpoint, we lose interest in the objects of the category, since the ``elements'' that we want to compose are \textit{morphisms}, or \textit{arrows}, of the category.
	
	\begin{conv}
		When we say that $x$ is an \textit{element} of the category $\Cc$, and we write $x\in \Cc$, we shall always mean that $x$ is an arrow in $\Cc$---and \textit{never} mean that it is an object.
	\end{conv}
	
	\begin{notat}
		For a category $\Cc$, we adopt the same conventions as in Notation \ref{notat:sources-targets}. Let $\Lambda = \Obj(\Cc)$ be the class of objects of $\Cc$ and let $\lambda,\mu\in\Lambda$; then we write $f\colon \lambda\to \mu$ for an arrow $f$ from $\lambda$ to $\mu$, and we say that $\lambda$ is the \textit{source} and $\mu$ is the \textit{target} of $f$. The notations $\Cc(\lambda,\mu), \Cc(\lambda,\Lambda), \Cc(\Lambda,\mu)$ also have analogous meaning to those in Notation \ref{notat:sources-targets}. We denote by $\mathbf{1}_\lambda$ the identity of the object $\lambda$, and the subscript will be omitted whenever $\lambda$ is clear from the context.
	\end{notat}
	\begin{conv}
		Unless otherwise specified, here $f\circ g$ denotes the composition of maps from \textit{right to left}, while, from now on, the composition without the symbol `$\circ$' means that we compose from \textit{left to right}: $gf = f\circ g$. This convention will be particularly useful when treating categories as algebraic structures, with an inner operation given by the composition.
	\end{conv}
	
	For the following definitions and lemmas, we refer to 
	\cite[\S II.2]{dehornoy2015foundations}.
	\begin{defin}
		Let $\Cc$ be a category, and let $x,y\in \Cc$. We say that \textit{$x$ left-divides $y$} (resp.\@ \textit{right-divides $y$}) if there exists $z\in\Cc$ such that $xz=y$ (resp.\@ $zx=y$). We denote this left-divisibility (resp.\@ \textit{right-divisibility}) relation by $x\preccurlyeq_L y$ (resp.\@ $x\preccurlyeq_R y $).
		
		We say that $x$ is a \textit{factor} of $y$ if $y = uxv$ for some suitable $u,v\in\Cc$. We denote the factoriality relation by $x\subseteq y$.
		
		A \textit{common right-multiple} (resp.\@ \textit{left-multiple}) of $x$ and $y$ is an arrow $xu = yv$ (resp.\@ $ux = vy$) for some suitable $u,v\in\Cc$. Such $u,v$ need not exist, nor be unique.
		
		Given $x,y\in\Cc$, an element $z\in\Cc$ is said to be a \textit{least common multiple on the right}, \textit{right-lcm} for short (resp.\@ \textit{least common multiple on the left}, \textit{left-lcm} for short) of $x$ and $y$, if $z$ is a common right-multiple $z = xu=yv$ of $x$ and $y$ (resp.\@ a common left-multiple $z = ux=vy$), and moreover $z$ left-divides (resp.\@ right-divides) every common right-multiple (resp.\@ left-multiple) of $x$ and $y$.
	\end{defin}
	\begin{defin}
		We say that a category $\Cc$ admits \textit{conditional right-lcms} (resp.\@ \textit{conditional left-lcms}) if any two elements $x,y\in\Cc$ admitting a common right-multiple (resp.\@ left-multiple) also admit a right-lcm (resp.\@ left-lcm).
	\end{defin}
	\begin{defin}
		A category $\Cc$ is said to be \textit{left-cancellative} (resp.\@ \textit{right-cancellative}) if $fx = fy$ implies $x=y$ (resp.\@ $xf=yf$ implies $x=y$) for all $f,x,y\in\Cc$ such that the compositions make sense.
	\end{defin}
	\begin{defin}\label{complementationop}
		Suppose $\Cc$ is a left-cancellative (resp.\@ right-cancellative) category that admits \textit{unique} conditional right-lcms (resp.\@ left-lcms). Then, given any two elements $x,y\in\Cc$, we define their \textit{complementation on the right} (resp.\@ \textit{on the left}) as the element $x\backslash_R y$ such that $x (x\backslash_R y) = y(y\backslash_R x)$ is the right-lcm of $x$ and $y$, if this right-lcm exists (resp.\@ the element $x\backslash_L y$ such that $ (x\backslash_L y) x= (y\backslash_L x) y$ is the left-lcm of $x$ and $y$, if this left-lcm exists). This complementation is unique by left (resp.\@ right) cancellativity.
	\end{defin}
	\begin{defin}
		An element $a$ of a category $\Cc$ is an \textit{atom} if, for all decompositions of $a$ into a product $a = x_1 x_2\dots x_n$ of elements of $\Cc$, exactly one of the $x_i$'s is non-invertible. This implies in particular that $a$ is non-invertible.
	\end{defin}
	\begin{defin}
		In a category $\Cc$, we say that two elements $x$ and $y$ are $=^\times$-equivalent (resp.\@ ${}^\times\!\!\!=$-equivalent) if $x = y e$ (resp.\@ $x = ey$) holds for some invertible element $e\in\Cc$. We write $x=^\times y$ (resp.\@ $x\; {}^\times\!\!\!=y$). Notice that these are equivalence relations on $\Cc$.
	\end{defin}
	
	For these relations, see \cite[Notation II.1.17]{dehornoy2015foundations}.
	\begin{lem}\label{lem:lcms-unique-up-to...}
		Let $\Cc$ be a left-cancellative (resp.\@ right-cancellative) category. If $x,y\in\Cc$ admit a right-lcm
		(resp.\@ left-lcm) $z$, then every other right-lcm (resp.\@ left-lcm) of $x$ and $y$ is $=^\times$-equivalent (resp.\@ ${}^\times\!\!\!=$-equivalent) to $z$.
	\end{lem}
	\begin{proof} The result was proven, for right-lcms, in \cite[Proposition II.2.10]{dehornoy2015foundations}. The proof for left-lcms is analogous.
	\end{proof}
	\begin{rem}
		Right-lcms are unique up to deformation by invertible elements: thus the existence of \textit{unique} conditional right-lcms implies that all invertible elements are identities; see \cite[p.\@ 204]{dehornoy2015foundations}.
	\end{rem}
	\subsection{Presented categories and presented groupoids}\label{subsection:presentations}
	\noindent We shall use the notion of \textit{presented categories} extensively, in the rest of the paper. This is a way of describing categories with generators and relations, analogously to group presentations, monoid presentations, etc. However, since (unlike the case of groups) we do not have the notion of ``a quotient of a category by a subcategory'', some more care is needed in this context. Indeed, quotients of categories are defined with respect to \textit{congruence relations}: this is analogous to quotients of monoids or semigroups. Our discourse closely follows \cite[Chapter II]{dehornoy2015foundations}.
	
	\begin{defin}
		A \textit{precategory} $\Pp$ is just the datum of objects and arrows. We do not require the existence of identities, nor the existence of a binary composition.  A \textit{small} precategory, endowed with its source and target maps, is a quiver.
	\end{defin}
	\begin{defin}
		Let $\Cc$ be a category. A \textit{subfamily} $\Ss$ of $\Cc$ is a precategory with same objects as $\Cc$, and arrows whose class is a subclass of the arrows of $\Cc$.	Given a subfamily $\Ss$ of $\Cc$, it makes sense to consider the \textit{generated subcategory}---that is, the smallest subcategory of $\Cc$ including $\Ss$.
	\end{defin}
	\begin{defin}
		A \textit{relation} on a category $\Cc$ is a class of pairs $(x,y)$ of elements $x,y\in\Cc$ with $\source(x) = \source(y)$ and $\target(x)= \target(y)$.
	\end{defin}
	Typically, we will be interested in relations on a path category $\Path(\Pp)$ on a given precategory $\Pp$. These will be the relations of our presentation.
	
	We do not care so much about working with classes, as the rest of the paper will only involve \textit{small categories}. For this reason, we prefer to assume directly, from now on, that our categories and precategories are small, although most of definitions and results may be generalised.
	\begin{defin}
		An equivalence relation $\equiv$ on a (small) category $\Cc$ is called a \textit{congruence} if it is compatible with compositions: i.e., if $x_1\equiv x_2$ and $y_1\equiv y_2$ imply $x_1y_1\equiv x_2y_2$ whenever the composition is defined.
	\end{defin}
	\begin{lem}[{\protect\cite[Lemma II.1.37]{dehornoy2015foundations}}]
		Let $R$ be a  relation on $\Cc$. There exists a unique minimal congruence $\equiv^+_R$ on $\Cc$ such that $x\equiv^+_R y$ for all $(x, y)$ in $R$. This is said to be the congruence generated by $R$.
	\end{lem}
	\begin{rem}\label{rem:congruence}
		For the path category $\Path(\Pp)$ on a precategory $\Pp$, the generated congruence relation $\equiv^+_R$ can be described in an alternative way which is often practically useful \cite[Lemma II.1.37]{dehornoy2015foundations}. Two paths $p,q\in\Path(\Pp)$ are $\equiv^+_R$-equivalent if and only if we can bring $p$ into $q$ in a finite number of steps, where each step is the application of a relation of the form $p_1| \alpha| p_2 \sim p_1|\beta |p_2$, where $(\alpha,\beta)$ or $(\beta, \alpha)$ lies in $R$.
	\end{rem}
	A \textit{category presentation} is a pair $(\Pp, R)$, where $\Pp$ is a precategory, and $R$ is a relation on $\Path(\Pp)$. The elements of $\Pp$ are called the \textit{generators} of this presentation.
	\begin{defin}
		Let $\Pp$ be a precategory, and let $R$ be a relation on $\Path(\Pp)$. The corresponding \textit{presented category}, denoted by $\langle \Pp \mid R\rangle^+$, is defined as follows:
		\begin{enumerate}
			\item[\textit{i.}] the objects are the objects of $\Pp$;
			\item[\textit{ii.}] the arrows are the equivalence classes in the quotient $\Path(\Pp)/\equiv^+_R$;
			\item[\textit{iii.}] the identities are the equivalence classes of the identities of $\Path(\Pp)$;
			\item[\textit{iv.}] the composition is the operation induced by the composition of $\Path(\Pp)$ modulo $\equiv^+_R$: this is well defined because $\equiv^+_R$ respects the composition.
		\end{enumerate}
		The fact that $\langle \Pp \mid R\rangle^+$ is a category is an easy verification
		(See also \cite[\S II.8]{categories-for-the-work}). We say that the pair $(\Pp,R)$ is a \textit{positive presentation} of $\langle \Pp \mid R\rangle^+$.
	\end{defin}
	\begin{conv}
		When a relation $R$ is clear from the context, we write $x\sim y$ for the pair $(x,y)\in R$. Abusing terminology, we sometimes say that $x\sim y$ is ``a relation''.
	\end{conv}
	\begin{rem}Every category $\Cc$ admits the trivial presentation $\Cc\cong \langle\Cc\mid \Rel(\Cc)\rangle^+$, where $\Rel(\Cc)$ is generated by the relations\footnote{If we include all relations of the form $x|y\sim xy$, then the generated congruence clearly includes all the relations of the form $x_1|\dots|x_r\sim x_1\dots x_r$.} $x_1|x_2|\dots|x_r\sim  x_1 x_2\dots x_r$ 
		in $\Path(\Cc)$ for all $x_1|\dots |x_r\in\Path(\Cc)$, and the relations $\mathbf{1}_\lambda\sim \varepsilon_\lambda$
		for all $\mathbf{1}_\lambda\in\mathbf{1}_\Cc$; where
		$\mathbf{1}_\Cc$ is the family consisting of all the identity elements of the category $\Cc$, and $\varepsilon_\lambda$ denotes the empty path on an object $\lambda$.\end{rem}
	In a presentation $\Cc = \langle \Pp\mid R\rangle^+$, an \textit{$\varepsilon$-relation} is a relation of the form $p \sim \varepsilon_{\source(p)}$ or $\varepsilon_{\source(p)}\sim p$, where $p$ is a non-empty path.
	\begin{lem}[{\protect\cite[Lemma II.1.42]{dehornoy2015foundations}}]\label{lem:epsilon-rel}
		If $\Cc$ admits a presentation $\Cc = \langle \Pp\mid R\rangle^+$ where $R$ contains no $\varepsilon$-relations, then $\Cc$ has no nontrivial invertible elements.
	\end{lem}
	If the precategory $\Pp$ has only one object, the category $\langle \Pp\mid R\rangle^+$ is a monoid. 
	\begin{notat}
		For any category or precategory $\Pp$, we use the notation $\bar{\Pp}$ (instead of the usual $\opp{\Pp}$) to denote its \textit{opposite}, and for $x\in\Pp$ we denote by $\bar{x}\in \bar{\Pp}$ its reverse. We denote by $\Double(\Pp)=\Pp\sqcup \bar{\Pp}$ the \textit{double} of $\Pp$---thus adopting the same notation as for quivers.
	\end{notat}
	Notice that, in the double $\Double(\Pp)$, we do not introduce any relation of the form $x\bar{x}\sim\varepsilon_{\source(x)}$ or $\bar{x}x\sim \varepsilon_{\target(x)}$---indeed, for precategories, compositions and identities need not even be defined. Moreover, notice that, if $x$ is a loop in $\Pp$, we do \textit{not} consider $x$ to be its own reverse: thus $x\neq \bar{x}$ always holds, and every loop is counted \textit{twice} in $\Double(\Pp)$. 
	
	\begin{defin}A \textit{groupoid} is a category (which we assume to be small) in which all morphisms are isomorphisms. Equivalently, it is a quiver $\Gg$ over $\Lambda$, with a morphism of quivers $\Gg\otimes \Gg\to \Gg$, $a|b\mapsto ab$, called the \textit{multiplication}, and a family of loops $\mathbf{1}_\lambda$ on each $\lambda\in\Lambda$, called the \textit{units}, such that the multiplication is associative, for all $a\in \Gg$ one has $a\mathbf{1}_{\target(a)} = a$ and $\mathbf{1}_{\source(a)} a=a$, and for every $a\in\Gg$ there is an \textit{inverse} $a^{-1}\in\Gg$ satisfying $aa^{-1} = \mathbf{1}_{\source(a)}$, $a^{-1} a = \mathbf{1}_{\target(a)}$.\end{defin}
	
	The following definition generalises the notion of a \textit{presented group}.
	\begin{defin}
		Let $\Pp$ be a precategory, and let $R$ be a set of relations on $\Path(\Pp)$. The corresponding \textit{presented groupoid}, denoted by $\langle \Pp \mid R\rangle$, is the category
		\[\langle \Double(\Pp)\mid R\cup F\rangle^+,\]
		where $F$ denotes the set of relations $x|\bar{x}\sim \varepsilon_{\source(x)}$, $\bar{x}|x\sim \varepsilon_{\target(
			x)}$ for all $x\in \Pp$. This is a groupoid: indeed, the inverse of the class of a path $x_1|\dots|x_r$ is the class of the path $\bar{x}_r|\dots|\bar{x}_1$.
	\end{defin}
	\begin{defin}\label{defin:free-groupoid}
		Given a precategory $\Pp$, the \textit{free groupoid} on $\Pp$ is defined as 
		$$\Free(\Pp) = \langle \Pp\mid \emptyset\rangle= \langle \Double(\Pp)\mid F\rangle^+.$$
	\end{defin}
	\subsection{Noetherianity}\label{subsection:noetherianity}
	Let $\Cc$ be a category. A binary relation $\prec$ on $\Cc$ is \textit{well-founded} if every nonempty subfamily $\Ss$ of $\Cc$ has a $\prec$-minimal element---i.e., an element $x\in\Ss$ such that, if $y\prec x$, then $y\notin \Ss$.
	\begin{defin}
		A category is left-Noetherian (resp.\@ right-Noetherian) if the relation $\prec_L$ (resp.\@ $\prec_R$) is well-founded.
		Here, $x\prec_L y$ (resp.\@ $x\prec_R y$) means that $xy'=y$ (resp.\@ $y'x=y$) for some $y'$ that is not invertible.
		
		A category is Noetherian if the relation of proper factoriality $\subset$ is well-founded.
		Here, $x\subset y$ means that $y'xy''=y$ and at least one of $y', y''$ is non-invertible.
	\end{defin}
	\begin{rem}
		If a left-cancellative category is both left- and right-Noetherian, then it is also Noetherian  \cite[Proposition II.2.29]{dehornoy2015foundations}.  The proof relies on a characterisation of right-Noetherianity via increasing sequences of left-divisibility relations \cite[Proposition II.2.28]{dehornoy2015foundations}, and this does not work without assuming left-cancellativity \cite[Exercise II.12]{dehornoy2015foundations}.
	\end{rem}
	A relation $x_1|\dots| x_n\sim y_1|\dots| y_m$ is called \textit{homogeneous} if $n=m$, and a \textit{homogeneous presentation} is a presentation in which all relations are homogeneous. We refer to \cite[Propositions II.2.32 and II.2.33]{dehornoy2015foundations} for the proof of the following:
	\begin{prop}\label{prop:homogeneous-implies-noetherian}
		If a category $\Cc$ admits a homogeneous presentation $\Cc \cong \langle \Pp\mid R\rangle^+$, then it is Noetherian.
	\end{prop}
	\subsection{Complemented presentations}\label{subsection:complement}
	A \textit{complement} for a presentation is, informally speaking, a way to find paths that ``complete'' an element to the smallest path that appears in some relation. This turns out to be a way to find common multiples.
	\begin{defin}
		A category presentation $(\Pp, R)$ is \textit{right-complemented} if the following conditions hold.
		\begin{enumerate}
			\item[\textit{i.}] The set $R$ contains no $\varepsilon$-relations.
			\item[\textit{ii.}] The set $R$ contains no relations of the form $x|\ldots \sim x|\ldots$
			for $x\in\Pp$.
			\item[\textit{iii.}] For all $x,y\in \Pp$, $x\neq y$, the set $R$ contains at most one relation of the form $x|\ldots \sim y|\ldots$.
		\end{enumerate}
		In the relations considered above, the paths occurring can have arbitrary length. 
		
		There is a partially defined map $\vartheta\colon \Pp\times \Pp\to \Path(\Pp)$, sending $(x,y)$ to the unique $\vartheta(x,y)\in\Path(\Pp)$ such that $x|\vartheta(x,y)\sim y|\vartheta(y,x)$ lies in $R$. By definition, $\vartheta(x, x)=\varepsilon_{\target(x)}$ for all $x\in\Pp$. We say that $\vartheta$ is a \textit{syntactic right-complement} for the presentation.
		
		We say that the presentation is \textit{short right-complemented} if it is right-complemented, and the syntactic right complement $\vartheta$, on every pair $(x,y)$, either takes values in $\Pp$ (identified with $\Path_1(\Pp)$) or is undefined. In this case, we call such a $\vartheta\colon \Pp\times \Pp\to \Pp$ a \textit{short syntactic right-complement}.
	\end{defin}
	The following is a specialisation of \cite[Lemma II.4.6]{dehornoy2015foundations}.
	\begin{lem}\label{lem:extension-to-vartheta*}
		Given a short right-complemented presentation $(\Pp, R)$ with a short syntactic right-complement $\vartheta\colon \Pp\times \Pp\to\Pp\cup\{\varepsilon_\lambda\mid\lambda\in\Lambda\}$, there exists a unique extension $\vartheta^*\colon \Path(\Pp)\times \Path(\Pp)\to \Path(\Pp)$ of $\vartheta$, such that the following conditions are satisfied:
		\begin{enumerate}
			\item[\textit{i.}] $\vartheta^*(x,x) = \varepsilon_{\target(x)}$ for all $x\in\Pp$;
			\item[\textit{ii.}] $\vartheta^*(p|q, r) = \vartheta^*(q, \vartheta^*(p,r))$ for all suitable $p,q,r\in\Path(\Pp)$ (see Figure \ref{subfig:rel2});
			\item[\textit{iii.}] $\vartheta^*(p, q|r) = \vartheta^*(p,q)|\vartheta^*(\vartheta^*(q,p),r)$ for all suitable $p,q,r\in\Path(\Pp)$ (see Figure \ref{subfig:rel3});
			\item[\textit{iv.}] $\vartheta^*(\varepsilon_{\source(p)}, p) = p$ and $\vartheta^*(p,\varepsilon_{\source(p)}) =\varepsilon_{\target(p)}$ for all $p\in\Path(\Pp)$.
		\end{enumerate}
		Moreover, this map $\vartheta^*$ is such that $\vartheta^*(p,q)$ is defined if and only if $\vartheta^*(q,p)$ is defined.
	\end{lem}
	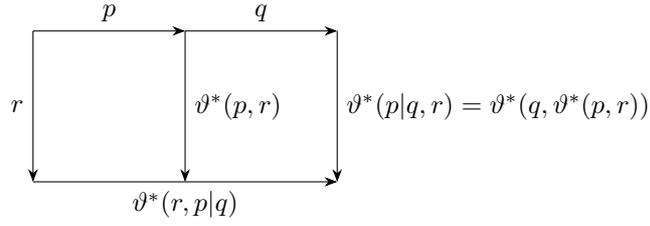
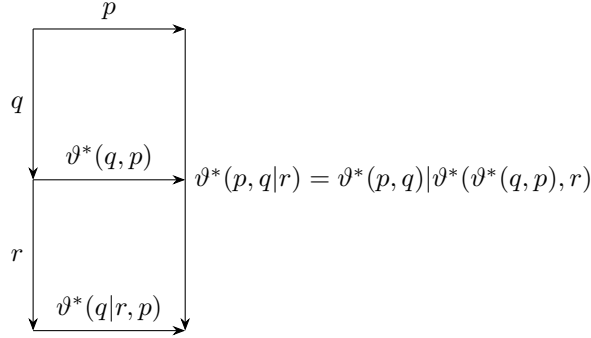
\begin{figure}
		\centering
		\begin{subfigure}{0.5\linewidth}
			\begin{tikzpicture}
				\draw[-Stealth] (-2,1) to node[above] {$p$} (0,1);
				\draw[-Stealth] (0,1) to node[above] {$q$} (2,1);
				\draw[-Stealth] (2,1) to node[right] {$\vartheta^*(p|q,r)=\vartheta^*(q,\vartheta^*(p,r))$} (2,-1);
				\draw[-Stealth] (-2,1) to node[left] {$r$} (-2,-1);
				\draw[-Stealth] (-2,-1) to node[below] {$\vartheta^*(r,p|q)$} (2,-1);
				\draw[-Stealth] (0,1) to node[right] {$\vartheta^*(p,r)$} (0,-1);
			\end{tikzpicture}
			\caption{The relation \textit{ii} of Lemma \ref{lem:extension-to-vartheta*} for $\vartheta^*$.}\label{subfig:rel2}
		\end{subfigure}
		\begin{subfigure}{0.5\linewidth}
			\begin{tikzpicture}
				\draw[-Stealth] (-1,2) to node[above] {$p$} (1,2);
				\draw[-Stealth] (1,2) to node[right] {$\vartheta^*(p,q|r) = \vartheta^*(p,q)|\vartheta^*(\vartheta^*(q,p),r)$} (1,-2);
				\draw[-Stealth] (-1,2) to node[left] {$q$} (-1,0);
				\draw[-Stealth] (-1,0) to node[left] {$r$} (-1,-2);
				\draw[-Stealth] (-1,0) to node[above] {$\vartheta^*(q,p)$} (1,0);
				\draw[-Stealth] (-1,-2) to node[above] {$\vartheta^*(q|r,p)$} (1,-2);
			\end{tikzpicture}
			\caption{The relation \textit{iii} of Lemma \ref{lem:extension-to-vartheta*} for $\vartheta^*$.}\label{subfig:rel3}
		\end{subfigure}
		\caption{A graphic interpretation of the relations \textit{ii} and \textit{iii} of Lemma \ref{lem:extension-to-vartheta*}, understood as consistency relations on a grid.}
	\end{figure}
	
	Given a short right-complement $\vartheta$, we define $\vartheta^*_3(p, q, r)=\vartheta^*(\vartheta^*(p, q), \vartheta^*(p, r))$
	for suitable $p, q, r\in\Path(\Pp)$.
	\begin{defin}We say that a short right-complemented presentation with a syntactic right-complement $\vartheta$ satisfies the \textit{sharp $\vartheta$-cube
			condition} if 
		$$\vartheta^*_3(p,q,r) = \vartheta^*_3(q,p,r)$$
		holds for all $p,q,r\in\Path(\Pp)$ such that both sides of the above equation are defined; and, if the left-hand side is not defined, neither is the right-hand side.
		
		We say that the sharp $\vartheta$-cube condition is true on a subfamily $\Ss$
		of $\Path(\Pp)$, if the above condition is true for all $(p,q,r)\in\Ss\times \Ss\times \Ss$ sharing the same source.\end{defin}
	\begin{prop}\label{prop:complemented-implies-lcms}
		Let $(\Pp, R)$ be a short right-complemented presentation, with a syntactic right-complement $\vartheta$. Suppose that the sharp $\vartheta$-cube condition is true for all triples of pairwise distinct elements of $\Pp$ with same source. Then, the presented category $\langle \Pp\mid R\rangle^+$ is left-cancellative, it admits conditional right-lcms, and the complementation operation $($on the right$)$ of the elements represented by $p, q\in\Path(\Pp)$
		is given by the element represented by $\vartheta^*(p,q)$
		$($For the complementation operation, see Definition \ref{complementationop}$)$.
		The right-lcm is thus the element represented by $p\,\vartheta^*(p, q)$.
	\end{prop}
	\begin{proof}
		See \cite[Proposition II.4.16]{dehornoy2015foundations}.
	\end{proof}

\subsection{Enveloping groupoids and the Ore criterion}\label{subsection:Ore}
A groupoid is a category in which every arrow is an isomorphism. Given a category $\Cc$, there is a ``smallest'' groupoid $\Env(\Cc)$ such that there exists a functor $\Cc\to\Env(\Cc)$:
\begin{defin}\label{def:envGpd}
	Let $\Cc$ be a category. The \textit{enveloping groupoid} $\Env(\Cc)$ is determined, up to isomorphism, by the following universal property: $\Env(\Cc)$ is a groupoid equipped with a functor $\iota\colon \Cc\to\Env(\Cc)$ such that, if $f$ is any functor $\Cc\to \Gg$ with $\Gg$ groupoid, then $f$ factors through $\iota$;
	that is, there uniquely exists a functor
	$\tilde{f}: \Env(\Cc)\to \Gg$ such that $f=\tilde{f}\circ\iota$:
	$$\begin{tikzcd}
		\Cc\ar[r,"f"]\ar[d,"\iota"] &\Gg\\ 
		\Env(\Cc)\ar[ru,dotted]& 
	\end{tikzcd}$$
\end{defin}
Such an object $\Env(\Cc)$ exists for every category $\Cc$. It can be explicitly constructed as follows; see \cite[Definition II.3.3]{dehornoy2015foundations}.
\begin{prop}\label{prop:envgroupoid}
	Given a category $\Cc$, the enveloping groupoid $\Env(\Cc)$ of $\Cc$ is defined (up to isomorphism) by
	$$\Env(\Cc) = \langle \Cc\mid \Rel(\Cc)\rangle = \langle \Double(\Cc)\mid \Rel(\Cc)\cup F\rangle^+,$$
	where $F$ denotes the set of relations $f|\bar{f}\sim \mathbf{1}_{\source(f)}$, $\bar{f}|f\sim \mathbf{1}_{\target(f)}$, and $\Rel(\Cc)$ denotes the set of relations $x_1|\dots|x_r\sim x_1\dots x_r$ and $\mathbf{1}_\lambda\sim \varepsilon_\lambda$. There is an obvious map $\iota\colon\Cc\to \Env(\Cc)$, given by sending each $f\in\Cc$ to the equivalence class of the corresponding path $f$ of length one.
\end{prop}
\begin{rem}
	The map $\iota\colon \Cc\to\Env(\Cc) $ need not be an embedding. Indeed, $\iota$ is certainly not an embedding if $\Cc$ is not cancellative (since an isomorphic copy of a subcategory of a groupoid is always cancellative). However, even the cancellativity is not a sufficient condition (see for instance \cite[Example II.3.9]{dehornoy2015foundations}).
\end{rem}
An important criterion to determine whether a monoid can be embedded into its enveloping group is due to Ore. We present, here, a generalisation of the original criterion, as reported in \cite[Proposition II.3.11]{dehornoy2015foundations}.
\begin{defin}\label{defin:Ore}
	A category is a \textit{left-Ore} (resp.\@ \textit{right-Ore})
	category if it is cancellative and any two elements with the same target (resp.\@ source) admit a common left-multiple (resp.\@ common right-multiple). A category is said to be \textit{Ore} if it is both left- and right-Ore.
\end{defin}
\begin{prop}\label{prop:Ore}
	Let $\Cc$ be a category. The following are equivalent:
	\begin{enumerate}
		\item[\textit{i.}] The category $\Cc$ is left-Ore.
		\item[\textit{ii.}] There exists a functor $\iota\colon \Cc\to\Env(\Cc)$, faithful and injective on the objects, and every element of $\Env(\Cc)$ is a left-fraction over $\iota(\Cc)$: i.e., it has the form $\iota(x)^{-1}\iota(y)$ for suitable $x,y\in\Cc$.
	\end{enumerate}
\end{prop}
A proof of Proposition \ref{prop:Ore} can be found in \cite[Appendix]{dehornoy2015foundations}.
\subsection{Garside families}\label{subsection:Garside}
We summarise here the main definitions in Garside theory. For a thorough discussion, we redirect the reader to the most extensive monograph on the topic \cite{dehornoy2015foundations}, and to the original papers \cite{chouraqui2010garside,dehornoy2002groupes,dehornoy2013garside,garside1969braid}, in which parts of this theory were introduced.

We begin with some definitions. In order to avoid logical issues, we shall always assume all categories to be \textit{small}. 
For a subfamily $\Ss$ of a category $\Cc$, we define $\Ss^\sharp = \Ss\Cc^\times\cup \Cc^\times$, where $\Cc^\times$ is the 
set of invertible elements in $\Cc$, and $\Ss\Cc^\times$ denotes the 
set of compositions $sc$, where $s\in \Ss$ and $c\in\Cc^\times$ are composable. Morally, $\Ss^\sharp$ is the \textit{deformation} of $\Ss$ (on the right) by the invertible elements.
\begin{defin}
	Let $\Cc$ be a left-cancellative category, and $\Ss$ a subfamily of $\Cc$. A path $x|y$ of length two in $\Cc$ is $\Ss$\textit{-greedy} if, for all $s\in\Ss$ and for all $c\in\Cc$ with $\target(c)=\source(x)$, whenever $s$ left-divides $cxy$, one also has that $s$ left-divides $cx$.
	
	A path $x_1|\dots|x_r$ in $\Cc$ is $\Ss$-greedy, by definition, if all subpaths $x_i|x_{i+1}$ are $\Ss$-greedy.
	
	A path is $\Ss$-\textit{normal} if it is $\Ss$-greedy, and all the entries lie in $\Ss^\sharp$.
\end{defin}
\begin{defin}\label{defin:garside-family}
	A subfamily $\Ss$ in a left-cancellative category $\Cc$ is a \textit{Garside family} if every element of $\Cc$ admits an $\Ss$-normal decomposition---i.e., can be written as $x_1\dots x_r$ where $x_1|\dots|x_r$ is $\Ss$-normal.
\end{defin}
In the case of the above definition, if a normal decomposition exists then it is ``essentially'' unique, meaning that it is unique up to a deformation by invertible elements \cite[Proposition III.1.25]{dehornoy2015foundations}.
\begin{prop}[{\protect\cite[Corollary IV.2.41]{dehornoy2015foundations} }]\label{prop:closure-is-Garside-family}
	Let $\Cc$ be a left-cancellative category, and $\Ss$ a subfamily of $\Cc$ that generates $\Cc$. Suppose $\Cc$ is right-Noetherian and admits unique conditional right-lcms. Then, the closure of $\Ss$ under right-lcms and $\backslash_R$ is the smallest Garside family which is $=^\times$-closed and includes 
	$\Ss\cup\mathbf{1}_{\Cc}$.
	Here, $\mathbf{1}_{\Cc}$ is the family consisting of all identity elements of the category $\Cc$.
\end{prop}
\section{Weak RC-systems and other cyclic systems}\label{section:cyclic}
\noindent First, we recall the definitions of weak RC-systems and their variants.
\begin{defin}[{\cite[Definition XIV.2.3]{dehornoy2015foundations}}]\label{def:weakRCsys}
	A \emph{weak RC-system} $(Q,\star)$ is the datum of a quiver $Q$ over $\Lambda$ and a partially defined binary operation $\star$ on $Q$, such that:
	\begin{enumerate}
		\item[\textit{i.}] The operation $x\star y$ is defined only if $\source(x)=\source(y)$.\footnote{ We may pose a stronger condition, and require that $x\star y$ is defined \textit{if and only if} $\source(x) = \source(y)$, because, in our context, there seems to be no reason to weaken the request. However, we respect Dehornoy \textit{et al.}'s definition, because it leads to more generality in the results.}
		\item[\textit{ii.}] Whenever $x\star y$ is defined, $y\star x$ is also defined, and $\source(x\star y) = \target(x)$, $\source(y\star x) = \target(y)$, $\target(x\star y) = \target(y\star x)$, as depicted in the diagram:
		
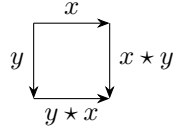
\begin{center}
			\begin{tikzpicture}
				\draw[-Stealth] (-0.5,0.5) to node[above] {$x$} (0.5,0.5);
				\draw[-Stealth] (-0.5,0.5) to node[left] {$y$} (-0.5,-0.5);
				\draw[-Stealth] (0.5,0.5) to node[right] {$x\star y$} (0.5,-0.5);
				\draw[-Stealth] (-0.5,-0.5) to node[below] {$y\star x$} (0.5,-0.5);
			\end{tikzpicture}
		\end{center}
		\item[\textit{iii.}] Whenever $x\star y$, $x\star z$ and $(x\star y)\star (x\star z)$ are defined, $y\star z$
		and $(y\star x)\star (y\star z)$ are also defined, and
		$$(x\star y)\star (x\star z) = (y\star x)\star (y\star z)$$
		holds (this is called the \textit{RC-law}, where ``RC'' stands for \textit{right-cyclic}, and it is fundamentally a ``cube rule''; see Figure \ref{fig:RC-law}).
		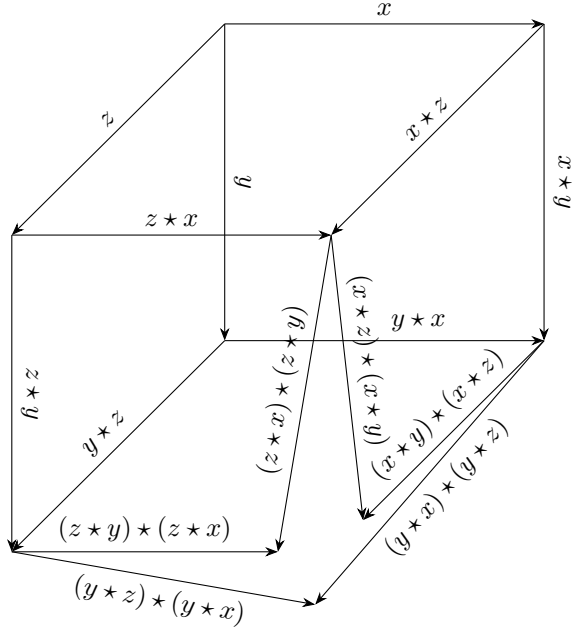
\begin{figure}\centering
			\begin{tikzpicture}[x=0.7cm, y=0.7cm]
				\draw[-Stealth] (0,6) to node[sloped,above] {\textcolor{black}{$x$}} (6,6);
				\draw[-Stealth] (0,6) to node[sloped,above] {$z$} (-4,2);
				\draw[-Stealth] (0,6) to node[sloped,above] {$y$} (0,0);
				\draw[-Stealth] (6,6) to node[sloped,above] {\textcolor{black}{$x\star y$}} (6,0);
				\draw[-Stealth] (6,6) to node[sloped,above] {$x\star z$} (2,2);
				\draw[-Stealth] (-4,2) to node[sloped,above] {$z\star x$} (2,2);
				\draw[-Stealth] (-4,2) to node[sloped,above] {$z\star y$} (-4,-4);
				\draw[-Stealth] (0,0) to node[sloped,above] {$y\star z$} (-4,-4);
				\draw[-Stealth] (0,0) to node[sloped,above] {$\qquad\;\; y\star x$} (6,0);
				\draw[-Stealth] (-4,-4) to node[sloped,above] {$(z\star y)\star(z\star x)$} (1,-4);
				\draw[-Stealth] (2,2) to node[sloped,above] {$(z\star x)\star(z\star y)$} (1,-4);
				\draw[-Stealth] (2,2) to node[sloped,above] {$(x\star z)\star(x\star y)\;\;\;\;$} (2.6,-3.4);
				\draw[-Stealth] (-4,-4) to node[sloped,below] {$(y\star z)\star(y\star x)$} (1.7,-5);
				\draw[-Stealth] (6,0) to node[sloped,below] {$(y\star x)\star(y\star z)$} (1.7,-5);
				\draw[-Stealth] (6,0) to node[sloped,above] {\textcolor{black}{$(x\star y)\star (x\star z)$}} (2.6,-3.4);
			\end{tikzpicture}
			\caption[RC-law]{The RC-law for $\star$ may be interpreted as a cube relation.}\label{fig:RC-law}
		\end{figure}
	\end{enumerate}
	We shall say that $(Q,\star) $ is \textit{left-non-degenerate}, when the operation $x\star y$ is defined if and only if $\source(x) = \source(y)$, and all the maps $x\star\blank\colon Q(\source(x),\Lambda)\to Q(\target(x),\Lambda)$, for $x\in Q$, are bijections.
\end{defin}
\begin{defin}
	A weak \emph{co-RC-system} $(Q,\bullet)$ is the datum of a quiver $Q$ and a partially defined binary operation $\bullet$ on $Q$, such that:
	\begin{enumerate}
		\item[\textit{i.}] The operation $x\bullet y$ is defined only if $\target(x)=\target(y)$.
		\item[\textit{ii.}] Whenever $x\bullet y$ is defined, $y\bullet x $ is also defined, and $\target(x\bullet y) = \source(x)$, $\target(y\bullet x) = \source(y)$, $\source(x\bullet  y)=\source(y\bullet  x)$, as depicted in the diagram:
		
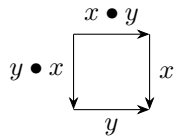
\begin{center}
			\begin{tikzpicture}
				\draw[-Stealth] (-0.5,0.5) to node[above] {$x\bullet  y$} (0.5,0.5);
				\draw[-Stealth] (-0.5,0.5) to node[left] {$y\bullet  x$} (-0.5,-0.5);
				\draw[-Stealth] (0.5,0.5) to node[right] {$x$} (0.5,-0.5);
				\draw[-Stealth] (-0.5,-0.5) to node[below] {$y$} (0.5,-0.5);
			\end{tikzpicture}
		\end{center}
		\item[\textit{iii.}] Whenever $x\bullet  y$, $x\bullet  z$ and $(x\bullet  y)\bullet  (x\bullet  z)$ are defined,
		$y\bullet z$ and $(y\bullet  x)\bullet  (y\bullet  z)$ are also defined, and $(x\bullet  y)\bullet  (x\bullet  z) = (y\bullet  x)\bullet (y\bullet z)$ holds (this is called the \textit{co-RC-law}).
	\end{enumerate}
	We shall say that $(Q,\bullet) $ is \textit{left-non-degenerate}, when $x\bullet y$ is defined if and only if $\target(x)=\target(y)$, and all the maps $x\bullet\blank\colon Q(\Lambda,\target(x))\to Q(\Lambda,\source(x))$, for $x\in Q$, are bijections.
\end{defin}
\begin{defin}
	A \textit{weak LC-system} $(Q,\tilde{\star})$ is the datum of a quiver $Q$ and a partially defined binary operation $\tilde{\star}$, such that $(Q,\bullet)$ is a co-RC-system with the binary operation $x\bullet y= y\,\tilde{\star}\, x$.
\end{defin}
\begin{defin}\label{definition:RLCsystem}
	A \textit{weak RLC-system} $(Q,\star,\tilde{\star})$ is the datum of a quiver $Q$ and two partially defined operations $\star,\tilde{\star}$ on $Q$, such that
	\begin{enumerate}
		\item[\textit{i.}] $(Q,\star)$ is a weak RC-system;
		\item[\textit{ii.}] $(Q,\tilde{\star})$ is a weak LC-system;
		\item[\textit{iii.}] the two operations satisfy the following compatibility condition:
		if $x\star y$ is defined, then 
		$(y\star x)\;\tilde{\star}\;(x\star y)$
		is also defined and
		$x = (y\star x)\;\tilde{\star}\; (x\star y)$;
		moreover, if $x\;\tilde{\star}\;y$ is defined, then 
		$(y\;\tilde{\star}\;x)\star(x\;\tilde{\star}\;y)$
		is also defined and 
		$x= (y\;\tilde{\star}\; x)\star (x\;\tilde{\star}\; y).$
	\end{enumerate}
\end{defin}
The above compatibility between $\star$ and $\tilde{\star}$ first appeared in the work of Rump \cite{RumpDecompositionTheorem}.
\subsection{Unit families and unital weak RC-systems}
In this section,
we recall the definition of unit families for weak RC-systems.
\begin{defin}[{cf.\@ \cite[Definition XIV.2.7]{dehornoy2015foundations}}]
	Let $(Q,\star) $ be a weak RC-system, and $\mathcal{E} = \set{\epsilon_\lambda}_{\lambda\in\Obj(Q)}$ a subfamily of $Q$. We say that $\mathcal{E}$ is a \emph{unit family} for $Q$ if
	\begin{enumerate}
		\item[\textit{i.}] $\epsilon_\lambda\in Q(\lambda,\lambda)$ for all $\lambda\in\Obj(Q)$;
		\item[\textit{ii.}] $x\star \epsilon_{\source(x)} $ is defined for all $x$, and $x\star \epsilon_{\source(x)}= \epsilon_{\target(x)}$;
		\item[\textit{iii.}] $\epsilon_{\source(x)}\star x$ is defined for all $x$, and $\epsilon_{\source(x)}\star x = x$;
		\item[\textit{iv.}] $x \star x $ is defined for all $x$, and $x \star x = \epsilon_{\target(x)}$.
	\end{enumerate}
\end{defin}
Therefore, $\mathcal{E}$ is a family of loops, one on each $\lambda\in\Obj(Q)$, such that applying $\star$ gives the following squares:
$$\begin{tikzcd}\lambda\ar[r,"x"]\ar[d,"\epsilon_\lambda"]& \mu\ar[d,"\epsilon_\mu"]\\
	\lambda\ar[r,"x"]& \mu \end{tikzcd}\quad \begin{tikzcd}\lambda\ar[r,"x"]\ar[d,"x"]& \mu\ar[d,"\epsilon_\mu"]\\
	\mu\ar[r,"\epsilon_\mu"]& \mu \end{tikzcd}$$
\begin{defin}[{cf.\@ \cite[Definition XIV.2.7]{dehornoy2015foundations}}]
	A weak RC-system $(Q,\star)$ is \emph{unital} if it has a unit family $\mathcal{E}$ and, moreover, the following property holds: for all $x,y\in Q(\Lambda,\mu)$, if $x\star y = y\star x = \epsilon_\mu$ then $x = y$.
\end{defin}
If $Q$ is unital with respect to unit families $\set{\epsilon_\lambda}_\lambda$ and $\set{\epsilon_\lambda'}_\lambda$, then $\epsilon_\lambda= \epsilon_\lambda\star\epsilon_\lambda = \epsilon_\lambda'$ for all $
\lambda$. Therefore, the unit family of a unital weak RC-system is unique \cite[\S XIV.2.1]{dehornoy2015foundations}.
\subsection{Completions}\label{subsection:completions}
If a weak RC-system $(Q, \star)$ does not have a unit family (or possibly even if it already has one), we may extend $Q$ with artificial ``units'' as follows; see \cite[Lemma XIV.2.12]{dehornoy2015foundations}.

Let $Q\subset Q'$, where $Q'$ is another precategory, whence we are going to pick up our additional units. For all objects $\lambda\in \Obj{Q}$ we choose $\epsilon_\lambda\in Q'(\lambda,\lambda)$, and we define $Q^\sharp(\lambda,\lambda) = Q(\lambda,\lambda)\cup \set{\epsilon_\lambda}$. Notice that $\epsilon_\lambda$ need not be in $Q'(\lambda,\lambda)\smallsetminus Q(\lambda,\lambda)$ and, in fact, it might \textit{either} be an additional element, \textit{or} be selected among the arrows already in $Q$.

We define $Q^\sharp(\lambda, \mu)=Q(\lambda, \mu)$ if $\lambda\ne\mu$, and 
then modify the operation $\star$ to an operation $\star^\sharp$ defined on $Q^\sharp$. We set
\begin{align*}
	&\epsilon_{\source(y)}\star^\sharp y =y,\\
	&	x\star^\sharp \epsilon_{\source(x)} = \epsilon_{\target(x)},\\
	&	x\star^\sharp x = \epsilon_{\target(x)},\\
	&	x\star^\sharp y =x\star y\; \text{in all the remaining cases.}
\end{align*}
An operation $(\blank)^\sharp$ as above will be called a \textit{unit insertion} on a weak RC-system. The loops $\epsilon_\lambda$ that we add to $Q$ will be called the \textit{inserted loops}.

There is a privileged way of choosing such an extension; namely, the case in which we pick $\epsilon_\lambda\notin Q(\lambda,\lambda)$ for all $\lambda$. This is called the \emph{completion} by Dehornoy \textit{et al.} \cite{dehornoy2015foundations}, and denoted $(\hat{Q},\hat{\star})$.
\begin{rem}
	At a first glance, this \textit{unit insertion} operation might seem to be an \textit{extension} of $(Q,\star)$. But, in fact, it is not. Indeed, when we set $x\star^\sharp x = \epsilon_{\source(x)}$, if $x\star x$ was already defined in $Q$ we are now forcing a different definition; see \cite[Example XIV.2.14]{dehornoy2015foundations}.
\end{rem}
\begin{lem}\label{lem:extension-is-weakRCs}
	Let $(Q,\star)$ be a left-non-degenerate\footnote{A similar lemma is reported in the monograph \cite[Lemma XIV.2.12]{dehornoy2015foundations}. However, that lemma lacks the left-non-degeneracy hypothesis---and, because of this missing hypothesis, the proof is flawed. To the reader's present day, this may have been amended in a revised version of the monograph.}
	weak RC-system. Then, $(\hat{Q},\hat{\star})$ is a unital weak RC-system. The unit family is given by the inserted loops $\set{\epsilon_\lambda}_{\lambda\in\Obj(Q)}$.
\end{lem}
\begin{proof}
	We only have to prove the RC-law: the other properties follow easily from the construction of the completion. 
	
	For the sake of simplicity, we write $\epsilon$ instead of the $\epsilon_\lambda$'s: the source of $\epsilon$ will be intended to be the only one that makes sense.
	
	Let $x,y,z$ be in $\hat{Q}$, such that $(x\;\hat{\star}\;y)\;\hat{\star}\; (x\;\hat{\star}\; z)$ is defined. Then $(y\;\hat{\star}\; x)\;\hat{\star}\; (y\;\hat{\star}\;z)$ is also defined. For $z=\epsilon$, we have\begin{align*}
		(x\;\hat{\star}\; y)\;\hat{\star}\; (x\;\hat{\star}\; \epsilon) = (x\;\hat{\star}\; y)\;\hat{\star}\;\epsilon = \epsilon_{\target(x\hat{\star}y)},
	\end{align*}
	meanwhile
	\begin{align*}
		(y\;\hat{\star}\;x)\;\hat{\star}\; (y\;\hat{\star}\;z) = (y\;\hat{\star}\; x)\;\hat{\star}\; \epsilon = \epsilon_{\target(y\,\hat{\star}\,x)}.
	\end{align*}
	This proves the RC-law in this case, because $\target(x\;\hat{\star}\;y)=\target(y\;\hat{\star}\;x)$. For $y =\epsilon$, we have 
	\begin{align*}
		(x\;\hat{\star}\; \epsilon)\;\hat{\star}\;(x\;\hat{\star}\;z) = \epsilon\;\hat{\star}\;(x\;\hat{\star}\;z) = x\;\hat{\star}\; z,
	\end{align*}
	and, on the other hand,
	\begin{align*}
		(\epsilon\;\hat{\star}\; x)\;\hat{\star}\;(\epsilon\;\hat{\star}\;z) = x\;\hat{\star}\; z,
	\end{align*}
	as desired.	Similarly, for the case $x=\epsilon$, 	$$(\epsilon\;\hat{\star}\; y)\;\hat{\star}\; (\epsilon\;\hat{\star}\;z) = y\;\hat{\star}\; z,$$
	and
	$$(y\;\hat{\star}\;\epsilon)\;\hat{\star}\; (y\;\hat{\star}\;z) = \epsilon\;\hat{\star}\; (y\;\hat{\star}\;z) = y\;\hat{\star}\; z,$$
	as desired. 
	
	In the case $x = y$, the RC-law becomes trivial. In the case $x=z$, 
	$$(x\;\hat{\star}\;y)\;\hat{\star}\; (x\;\hat{\star}\; x) = (x\;\hat{\star}\; y)\;\hat{\star}\;\epsilon = \epsilon_{\target(x\,\hat{\star}\,y)}$$
	on one side; while, on the other side, 
	$$(y\;\hat{\star}\; x)\;\hat{\star}\; (y\;\hat{\star}\; x) = \epsilon_{\target(y\,\hat{\star}\,x)}.$$
	They are same, because $\target(x\;\hat{\star}\;y)=\target(y\;\hat{\star}\;x)$.	
	In the case $y=z$, the RC-law again boils down to 
	the equation $\epsilon_{\target(x\,\hat{\star}\,y)} = \epsilon_{\target(y\,\hat{\star}\,x)}$.
	
	Thus, we are now left with the case $x,y,z\neq \epsilon$ distinct. Two subcases can occur: either $x\;\hat{\star}\; y = x\;\hat{\star}\; z$, or $x\;\hat{\star}\; y\neq x\;\hat{\star}\; z$.
	
	Suppose $x\;\hat{\star}\; y = x\;\hat{\star}\; z$. Since $x,y,z\neq\epsilon$ are distinct, $x\star y = x\star z$ holds: but $(Q,\star)$ is left-non-degenerate, thus $y = z$, which is a contradiction.	Therefore, $x\;\hat{\star}\;y\neq x\;\hat{\star}\; z$. Using the left-non-degeneracy again, we also observe that $y\;\hat{\star}\;x\neq y\;\hat{\star}\;z$ holds. Then, 
	\begin{align*}(x\;\hat{\star}\; y)\;\hat{\star}\;(x\;\hat{\star}\; z) =&\; (x\star y)\star (x\star z) \\=&\; (y\star x)\star (y\star z)\\
		=&\;(y\;\hat{\star}\; x)\;\hat{\star}\; (y\;\hat{\star}\;z),\end{align*}
	
	because $(Q, \star)$ is a weak RC-system. This proves the lemma.
\end{proof}
\subsection{The structure category of a unital weak RC-system}
\noindent For a given weak RC-system $(Q,\star)$, we may consider the \textit{path category} $\Path(Q)$. This is naturally endowed with an operation derived from $\star$, that we are going to denote by $\star$ again, for the sake of simplicity. If $\alpha = a_1^1|a_2^1|\dots|a_r^1$ and $\beta = b_1^1|b_1^2|\dots| b_1^s$ are two non-empty paths sharing the same source, the paths $\alpha\star\beta$ and $\beta\star \alpha$ are given by ``completing the grid'' as shown in the diagram:
$$\begin{tikzcd}
	{}\ar[r,"a_1^1"]\ar[d,"b_1^1"]& {}\ar[r,"a_2^1"]\ar[d,"b_2^1"] & \dots &{}\ar[r,"a_r^1"] &{}\ar[d,"b_{r+1}^1"]\\
	{}\ar[r,"a_1^2"]\ar[d,"b_1^2"]& {}\ar[r,"a_2^2"]\ar[d,"b_2^2"] & \dots & {}\ar[r,"a_r^2"]&{}\ar[d,"b_{r+1}^2"]\\
	{\vdots}& {\vdots}&{} &  {} & {\vdots}\\
	{}\ar[d,"b_1^s"]& {}\ar[d,"b_2^s"] & {} & {}&{}\ar[d,"b_{r+1}^s"]\\
	{}\ar[r,"\;\;a_1^{s+1}"]& {}\ar[r,"\;\;a_2^{s+1}"] & \dots &{}\ar[r,"a_{r}^{s+1}"] &{}\\
\end{tikzcd}$$
Each square represents the application of $\star$ in $Q$; that is, $a_i^j=b_i^{j-1}\star a_i^{j-1}$ $(i=1, 2, \ldots, r, j=2, 3, \ldots, s+1)$ and $b_j^i=a_{j-1}^i\star b_{j-1}^i$ $(i=1, 2, \ldots, s, j=2, 3, \ldots, r+1)$. Thus $\alpha\star \beta$ is defined if and only if all the squares in the grid are defined. By definition, $$\alpha\star \beta = b_{r+1}^1|b_{r+1}^2|\dots| b_{r+1}^s,\qquad \beta\star\alpha = a_1^{s+1}|a_2^{s+1}|\dots|a_r^{s+1}.$$
If $\varepsilon_\lambda$ is the empty path on $\lambda$, we define
\begin{align}
	\label{eq:star-on-empty-path-1}& \varepsilon_\lambda\star\varepsilon_\lambda=\varepsilon_\lambda,\\
	\label{eq:star-on-empty-path-2}&\varepsilon_{\source(\alpha)}\star \alpha= \alpha,\\
	\label{eq:star-on-empty-path-3}&\alpha\star \varepsilon_{\source(\alpha)}=\varepsilon_{\target(\alpha)}.
\end{align}
We note that the symbol $\varepsilon_\lambda$ means an empty path, while
$\epsilon_\lambda$ was used in \S\ref{subsection:completions} to signify a unit.
\begin{rem}\label{rem:star-on-empty-paths}
	If we consider non-empty paths $\alpha$ and $\beta$ as above, and we insert some occurrences of empty paths in the midst of them, the definition of $\alpha\star \beta$ does not change, because of the above \eqref{eq:star-on-empty-path-1}, \eqref{eq:star-on-empty-path-2},
	and \eqref{eq:star-on-empty-path-3}. In other words, our definition of $\star$ is consistent with the possibility of inserting empty subpaths in $\alpha,\beta$. On the other hand, the definition of $\star$ \textit{must} be consistent with this insertion of empty paths, in order for $\star$ to be well defined; and one can easily get convinced that conditions \eqref{eq:star-on-empty-path-1}, \eqref{eq:star-on-empty-path-2}, \eqref{eq:star-on-empty-path-3} are \textit{exactly} the conditions that we have to impose in order to get this consistency.\end{rem}
\begin{lem}
	If $(Q,\star)$ is a weak RC-system, then $(\Path(Q), \star)$ is a weak RC-system.
\end{lem}
\begin{proof}
	Observe that if $\alpha\star \beta$ is defined, for two paths $\alpha,\beta$, then all the squares in the grid are defined, and hence $\beta\star \alpha$ is also defined, as one of the sides of this grid, and it is immediate to see condition \textit{ii} from Definition \ref{def:weakRCsys}. 
	
	If $\alpha\star \beta$, $\alpha\star \gamma$, and $(\alpha\star\beta)\star (\alpha\star\gamma)$ are defined for paths $\alpha,\beta,\gamma$, then we can draw a three-dimensional grid in which all cubes are well-defined because of the RC-law in $(Q,\star)$: thus $(\beta\star\alpha)\star (\beta\star \gamma)$ is defined, as one of the edges of this three-dimensional grid, and $(\alpha\star\beta)\star (\alpha\star\gamma)= (\beta\star\alpha)\star (\beta\star \gamma)$.
\end{proof}

Let $(Q,\star)$ be a unital weak RC-system, with a unit family $\mathcal{E}$. Notice that $\Path(\mathcal{E})$ cannot be a unit family on $\Path(Q)$, because the definition of a unit family requires that for each vertex $\lambda$ we have a \textit{unique} loop $\epsilon_\lambda$---while $\Path(\mathcal{E})$ has infinitely many loops on each vertex.

However, $\Path(\mathcal{E})$ has many similarities with unit families. This leads us to the following definition: given two paths $\alpha,\beta\in\Path(Q)$, we say that $\alpha\equiv\beta $ holds if and only if $\alpha$ and $\beta $ fit in a grid as above, where the two other sides of the grid lie in $\Path(\mathcal{E})$; that is, $\alpha\equiv \beta$ if and only if $\alpha\star\beta, \beta\star\alpha\in\Path(\mathcal{E})$.
\begin{lem}
	The relation $\equiv$ is an equivalence relation and a congruence on $\Path(Q)$.
\end{lem}
\begin{proof}
	It is clear that $\equiv$ is reflexive and symmetric, as well as it is clear that it respects the composition of paths. We only need to prove the transitivity.
	
	Suppose $\alpha\equiv \beta $ and $\beta\equiv\gamma$ hold. Because of the cube rule, we have the following diagram:
	
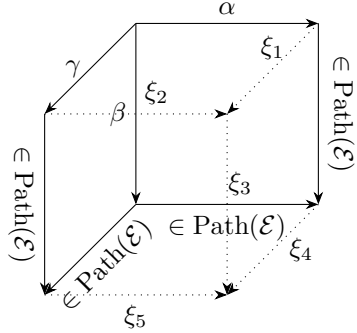
\begin{center}
		\begin{tikzpicture}[x=1.2cm, y=1.2cm]
			\draw[-Stealth] (0,2) to node[sloped, above]{$\alpha$} (2,2);
			\draw[-Stealth] (0,2) to node[left]{$\beta$} (0,0);
			\draw[-Stealth] (0,2) to node[left]{$\gamma$} (-1,1);
			\draw[-Stealth] (2,2) to node[sloped, above]{\small$\in\Path(\mathcal{E})$} (2,0);
			\draw[-Stealth] (0,0) to node[sloped,below]{\small$\in\Path(\mathcal{E})$} (2,0);
			\draw[-Stealth] (-1,1) to node[sloped, below]{\small$\in\Path(\mathcal{E})$} (-1,-1);
			\draw[-Stealth] (0,0) to node[sloped,below]{\small$\in\Path(\mathcal{E})$} (-1,-1);
			\draw[-Stealth,dotted] (2,2) to node[above]{$\xi_1$} (1,1);
			\draw[-Stealth,dotted] (-1,1) to node[above]{$\quad\;\;\xi_2$} (1,1);
			\draw[-Stealth,dotted] (1,1) to node[above]{$\quad\xi_3$} (1,-1);
			\draw[-Stealth,dotted] (2,0) to node[right]{$\;\xi_4$} (1,-1);
			\draw[-Stealth,dotted] (-1,-1) to node[below]{$\xi_5$} (1,-1);
		\end{tikzpicture}
	\end{center}
	Observe that, if $\zeta_1$ and $\zeta_2$ lie in $\Path(\mathcal{E})$, then $\zeta_1\star\zeta_2$ and $\zeta_2\star \zeta_1$ also lie in $\Path(\mathcal{E})$. Therefore, in the diagram, we obtain $\xi_4,\xi_5\in\Path(\mathcal{E})$.
	
	Since $\epsilon\star x=x$ 
	and $x\star\epsilon=\epsilon$ hold  in $Q$, it is easy to see, by induction on the grids, that $\xi\star \zeta = \zeta$ holds for all $\zeta\in\Path(Q)$ and for all $\xi\in\Path(\mathcal{E})$ with $\source(\xi) = \source(\zeta)$. Therefore, in the above diagram, $\xi_1$ and $\xi_2$ also lie in $\Path(\mathcal{E})$. This means $\alpha\equiv \gamma$, as desired.
\end{proof}
We can finally give the main definition\footnote{Our definition amends \cite[Definition XIV.2.25]{dehornoy2015foundations}.} of this section:
\begin{defin}\label{def:structure-category-of-wRCs}
	If $(Q,\star)$ is a unital weak RC-system with the unit family $\mathcal{E}$, we define the \textit{structure category} $\Cc(Q)$ of $(Q,\star)$ as the quotient
	$$\Cc(Q)=\Path(Q)/\equiv.$$
	The composition of the category is induced by the junction of paths: this is well defined, because $\equiv$ is a congruence, and thus the quotient modulo $\equiv$ respects the junction of paths.
\end{defin}
The ``grid calculus'' is a visualisation of the consistency relations from Lemma \ref{lem:extension-to-vartheta*}. Indeed, \textit{a posteriori} we want $\star$ to behave like a complementation, thus we needed to extend $\star$ on $\Path(Q)$ in a way that was consistent with concatenations and grids. 
\subsection{When the structure category is Garside} A presented category $\Cc = (\Pp,R)$ satisfies a \textit{quadratic isoperimetric inequality} with respect to this presentation, if for all $n$ (larger than a certain $n_0$) and for all equivalent paths $\alpha,\beta\in\Path(\Pp)$ with $|\alpha|+ |\beta|\le n$ (where $|\cdot|$ is the length function), the number of relations in $R$ to apply in order to get from $\alpha$ to $\beta$ is $\mathcal{O}(n^2)$; see \cite[Definition IV.3.9]{dehornoy2015foundations}.

The main result about the structure category of a unital weak RC-system is the following proposition (which is \cite[Proposition XIV.2.27]{dehornoy2015foundations}, with some modifications and addenda, proven by using our alternative definition of $\Cc(Q)$ rather than \cite[Definition XIV.2.25]{dehornoy2015foundations}).
\begin{prop}\label{prop:Garside-family}
	Let $(Q,\star)$ be a unital weak RC-system, with the unit family $\mathcal{E}$. Let $\Cc(Q)$ denote the associated structure category. Let $\iota\colon Q\to \Cc(Q)$ be the map sending $x\in Q$ first to the corresponding path $x\in\Path(Q)$ of length one and, then, mapping this path to its class $\iota(x)$ modulo $\equiv$. Then, the following properties hold.
	\begin{enumerate}
		\item[\textit{i.}] The map $\iota$ embeds $Q\smallsetminus\mathcal{E}$ into $\Cc(Q)$.
		\item[\textit{ii.}] The category $\Cc(Q)$ admits the presentation \begin{align*}\Big\langle Q\;\Big|\;\big\lbrace x|(x\star y)\sim y|(y\star x)\text{ for all }x\neq y\in Q\smallsetminus\mathcal{E}\text{ such that }x\star y\text{ is defined}\big\rbrace  \\  \cup \;\big\lbrace\epsilon_\lambda\sim \varepsilon_\lambda\text{ for all }\lambda\in\Obj(Q)\big\rbrace \Big\rangle^+,\end{align*}
		and $\Cc(Q)$ satisfies a quadratic isoperimetric inequality with respect to this presentation.
	\end{enumerate}
	If moreover $Q$ satisfies the condition\footnote{The condition reported in \cite[Proposition XIV.2.27]{dehornoy2015foundations} was the following: \[\text{«if }x\notin \mathcal{E} \text{ and } x\star y \in\mathcal{E}, \text{ then } y = x\text{».} \]However, we require the weaker condition \eqref{eqn:add_cond} for two reasons. First, because the weaker condition is sufficient. Second, because the stronger condition of \cite{dehornoy2015foundations} is hardly ever satisfied: if $\mathcal{E}$ is a unit family, then \( x\star \epsilon = \epsilon\) holds for all \( x\), hence the instance \( y = \epsilon\) makes the condition false in every case except the trivial one $Q = \mathcal{E}$.}
	\begin{equation}\label{eqn:add_cond}\text{«if }x,y\notin \mathcal{E} \text{ and } x\star y\in\mathcal{E}, \text{ then } y = x\text{»,}\end{equation}
	then $\Cc(Q)$ also admits the presentation
	\begin{align}\label{presentation}\Cc(Q) = \Big\langle Q\smallsetminus\mathcal{E}\;\Big|\; x|(x\star y) \sim y|(y\star x)\text{ for all }x\neq y \in Q\smallsetminus \mathcal{E}\Big.\\
		\nonumber\Big.\text{ such that }x\star y\text{ is defined}\Big\rangle^+,\end{align}and the following additional properties hold:
	\begin{enumerate}
		\item[\textit{iii.}] The category $\Cc(Q)$ has no nontrivial invertible elements.
		\item[\textit{iv.}] The category $\Cc(Q)$ is left-cancellative, it admits unique conditional right-lcms, and the complementation operation 
		$($on the right$)$ of the elements represented by $u$ and $v$ $($$u, v\in\Path(Q\smallsetminus\mathcal{E})$$)$ is given by 
		the element represented by $(u\star v)^\bullet$. Here, for $u\in\Path(Q)$, $u^\bullet$ is the element of $\Path(Q\smallsetminus\mathcal{E})$  defined by replacing with identities every element of $\mathcal{E}$ occurring in the entries of $u$.
		\item[\textit{v.}] The category $\Cc(Q)$ is Noetherian, and the atoms are given by the elements of $Q\smallsetminus\mathcal{E}$.
		\item[\textit{vi.}] The closure $E$ of the subfamily $Q\smallsetminus\mathcal{E}$ under right-lcms is a Garside family for $\Cc(Q)$, and it is the smallest Garside family containing $(Q\smallsetminus\mathcal{E})\cup \mathbf{1}_{\Cc(Q)}$.\end{enumerate}
\end{prop}
\begin{proof} Recall that, in the weak RC-system $\Path(Q)$, two paths $\alpha = a_1^1|\dots |a_r^1$ and $\beta = b_1^1|\dots|b_1^s$ are equivalent modulo $\equiv$ if and only if they fit in a grid like the following, where the other two sides lie in $\Path(\mathcal{E})$:
	$$\begin{tikzcd}
		{}\ar[r,"a_1^1"]\ar[d,"b_1^1"]& {}\ar[r,"a_2^1"]\ar[d] & \dots &{}\ar[r,"a_r^1"] &{}\ar[d,"\in \mathcal{E}"]\\
		{}\ar[r]\ar[d,"b_1^2"]& {}\ar[r]\ar[d]& \dots & {}\ar[r]&{}\ar[d,"\in\mathcal{E}"]\\
		{\vdots}& {\vdots}&{}\ar[r,"a_i^j"]\ar[d,"b_i^j"] &  {}\ar[d,"b_{i+1}^j"] & {\vdots}\\
		{}\ar[d,"b_1^s"]& {}\ar[d] & {}\ar[r,"a_i^{j+1}"] & {}&{}\ar[d,"\in\mathcal{E}"]\\
		{}\ar[r,"\in\mathcal{E}"]& {}\ar[r,"\in\mathcal{E}"] & \dots &{}\ar[r,"\in\mathcal{E}"] &{}\\
	\end{tikzcd}$$
	In each inner square with arrows labelled by $a_i^j, b_{i+1}^j, b_i^j, a_i^{j+1}$, four cases can occur:
	\begin{enumerate}
		\item[1)] $a_i^j, b_i^j\in Q\smallsetminus\mathcal{E}$,
		\item[2)] $a_i^j\in\mathcal{E}, b_i^j\in Q\smallsetminus\mathcal{E}$, 
		\item[3)] $a_i^j\in Q\smallsetminus\mathcal{E}, b_i^j\in\mathcal{E}$, or
		\item[4)] $a_i^j, b_i^j\in\mathcal{E}$.
	\end{enumerate}
	If we set $x=a_i^j$ and $\epsilon_{\source(x)}=b_i^j$ in the case 3),
	then $a_i^{j+1}=\epsilon_{\source(x)}\star x=x$
	and $b_{i+1}^j=x\star\epsilon_{\source(x)}=\epsilon_{t(x)}$. In the same manner, we can see that each inner square with arrows labelled by $a_i^j$, $b_{i+1}^j$, $b_i^j$, $a_i^{j+1}$ corresponds to one of the following types:
	\begin{itemize}
		\item[A)] $a_i^j=x$, $b_{i+1}^j=x\star y$, $b_i^j=y$, $a_i^{j+1}=y\star x$ $(x, y\in Q\smallsetminus\mathcal{E}, x\ne y, \mathrm{\ and\ } x\star y \mathrm{\ is\ defined})$,
		\item[B)] $a_i^j=x$, $b_{i+1}^j=\epsilon_{\target(x)}$, $ b_i^j=x$, $a_i^{j+1}=\epsilon_{\target(x)}$ $(x\in Q\smallsetminus\mathcal{E})$,
		\item[C1)] $a_i^j=x$, $b_{i+1}^j=\epsilon_{\target(x)}$, $b_i^j=\epsilon_{\source(x)}$, $a_i^{j+1}=x$ $(x\in Q\smallsetminus\mathcal{E})$,
		\item[C2)] $a_i^j=\epsilon_{\source(x)}$, $b_{i+1}^j=b_i^j=x$, $a_i^{j+1}=\epsilon_{\target(x)}$ $(x\in Q\smallsetminus\mathcal{E})$,
		\item[D)] $a_i^j=b_{i+1}^j=b_i^j=a_i^{j+1}=\epsilon_\lambda$.
	\end{itemize}
	
	For $\alpha = a_1^1|\dots |a_r^1$ and $\beta = b_1^1|\dots|b_1^s$ with $\target(a_r^1)=\target(b_1^s)$, 
	we set $\lambda=\target(a_r^1)(=\target(b_1^s))$.
	It is clear that $\alpha\equiv\alpha|\epsilon_\lambda|\cdots|\epsilon_\lambda$,
	in which $\epsilon_\lambda(\in\mathcal{E})$ appears $s$ times;
	and that $\beta\equiv\beta|\epsilon_\lambda|\cdots|\epsilon_\lambda$,
	in which $\epsilon_\lambda$ appears $r$ times.
	We note that these two paths $\alpha\equiv\alpha|\epsilon_\lambda|\cdots|\epsilon_\lambda$
	and $\beta\equiv\beta|\epsilon_\lambda|\cdots|\epsilon_\lambda$
	have the same length $r+s$.
	
	Because $\alpha|\epsilon_\lambda|\cdots|\epsilon_\lambda$ 
	(resp.\ $\beta|\epsilon_\lambda|\cdots|\epsilon_\lambda$)
	appears on the uppermost and the rightmost arrows 
	(resp.\ the leftmost and the bottom arrows) in the grid, we can change the path
	$\alpha|\epsilon_\lambda|\cdots|\epsilon_\lambda$ to the path
	$\beta|\epsilon_\lambda|\cdots|\epsilon_\lambda$
	by replacing $a_i^j|b_{i+1}^j$ with $b_i^j|a_i^{j+1}$,
	along every square with arrows labelled $a_i^j, b_{i+1}^j, b_i^j, a_i^{j+1}$. Consequently, $\alpha\equiv\beta$ means that 
	we can bring
	$\alpha|\epsilon_\lambda|\cdots|\epsilon_\lambda$ into
	$\beta|\epsilon_\lambda|\cdots|\epsilon_\lambda$
	by replacing:
	\begin{itemize}
		\item[A)]$x|(x\star y)$ with $y|(y\star x)$
		$(x, y\in Q\smallsetminus\mathcal{E}, x\ne y, \mathrm{\ and\ } x\star y \mathrm{\ is\ defined})$,
		\item[B)]$x|\epsilon$ with $x|\epsilon$ $(x\in Q\smallsetminus\mathcal{E})$,
		\item[C1)] $x|\epsilon$ with $\epsilon |x$ $(x\in Q\smallsetminus\mathcal{E})$,
		\item[C2)] $\epsilon|x$ with $x|\epsilon$ $(x\in Q\smallsetminus\mathcal{E})$,
		\item[D)] $\epsilon|\epsilon$ with $\epsilon|\epsilon$.
	\end{itemize}
	The following relations of types A and E can thereby generate the congruence $\equiv$:
	\begin{itemize}
		\item[A)]$x|(x\star y)\sim y|(y\star x)$
		$(x, y\in Q\smallsetminus\mathcal{E}, x\ne y, \mathrm{\ and\ } x\star y \mathrm{\ is\ defined})$,
		and
		\item[E)]$\epsilon_\lambda\sim \varepsilon_\lambda$ for all $\lambda\in\Obj(Q)$,
	\end{itemize}
	where $\varepsilon_\lambda$ denotes the empty path on $\lambda$. 
	Notice that the relations of type E are all $\varepsilon$-relations.
	
	We have thus proven a part of \textit{ii}: we denote by
	$\equiv'$ the smallest congruence relation
	on $\Path(Q)$ that contains all the relations of the forms A and E.
	Then $\alpha\equiv\beta$ holds $(\alpha, \beta\in\Path(Q))$, if $\alpha\equiv'\beta$.
	
	For completing the proof of \textit{ii},
	it suffices to show that
	\[
	a_1|\cdots|a_r|\epsilon_\lambda|b_1|\cdots|b_s\equiv a_1|\cdots|a_r|\varepsilon_\lambda|b_1|\cdots|b_s,
	\]
	where $\target(a_r)=\source(b_1)=\lambda$.
	The proof is obvious.
	In fact, by using the replacements of type C2 we get
	\begin{align*}
		a_1|\cdots|a_r|\epsilon_\lambda|b_1|\cdots|b_s
		&\equiv 
		a_1|\cdots|a_r|b_1|\cdots|b_s|\epsilon_{\target(b_s)},
	\end{align*}
	which is exactly $a_1|\cdots|a_r|b_1|\cdots|b_s=a_1|\cdots|a_r|\varepsilon_\lambda|b_1|\cdots|b_s$ by the definition of 
	the congruence $\equiv$.
	
	Now, we turn to the proof of \textit{i}; i.e., we prove that the map
	$\iota$ restricted to $Q\smallsetminus\mathcal{E}$ is an embedding. Let $x$ and $y$ be elements of $Q\smallsetminus\mathcal{E}$ such that $\iota(x) = \iota(y)$. The relation $\iota(x) = \iota(y)$ implies $x\equiv y$, and hence
	$$x\star y\in\mathcal{E},\;y\star x\in\mathcal{E}.$$
	Since $(Q,\star)$ is unital, this implies $x=y$. Therefore, $\iota$ is an embedding.
	
	Moreover, if $\alpha$ has length $r$ and $\beta$ has length $s$, and $\alpha$ and $\beta$ are equivalent in $\Cc(Q)$, then it takes $s$ relations to bring $\alpha$ into $\alpha | (\epsilon_{\target(\alpha)})^s$; it takes $(r+s)^2$ relations to bring $\alpha | (\epsilon_{\target(\alpha)})^s$ into $\beta | (\epsilon_{\target(\beta)})^r$; and it finally takes $r$ relations to bring $\beta | (\epsilon_{\target(\beta)})^r$ into $\beta$. This amounts to $r+s+r^2+s^2+2rs$ total relations, thus the category $\Cc(Q)$ satisfies a quadratic isoperimetric inequality.
	
	Now, if the additional condition \eqref{eqn:add_cond} holds then, for all $x\neq y\in Q\smallsetminus\mathcal{E}$ such that $x\star y$ is defined, we have that both $x\star y$ and $ y\star x$ lie in $ Q\smallsetminus\mathcal{E}$. As a consequence, no $\epsilon$'s appear in the set of relations
	\begin{equation*}
		\big\lbrace x|(x\star y)\sim y|(y\star x)\text{ for all }x\neq y\in Q\smallsetminus\mathcal{E}\text{ such that }x\star y\text{ is defined}\big\rbrace.
	\end{equation*}
	Therefore, since the $\epsilon$'s are modded out by the set of relations 
	\begin{equation*}
		\big\lbrace\epsilon_\lambda\sim \varepsilon_\lambda\text{ for all }\lambda\in\Obj(Q)\big\rbrace,
	\end{equation*}
	we can simply omit the $\epsilon$'s from the set of generators, thus obtaining 
	$$\Cc(Q) = \Big\langle Q\smallsetminus\mathcal{E}\;\Big|\; x|(x\star y) \sim y|(y\star x)\text{ for all }x\neq y \in Q\smallsetminus \mathcal{E}\text{ such that }x\star y\text{ is defined}\Big\rangle^+.$$
	
	For the proof of \textit{iii}, we notice that the presentation \eqref{presentation} contains no $\varepsilon$-relations. Then \textit{iii} follows directly from Lemma \ref{lem:epsilon-rel}.
	
	The presentation \eqref{presentation} of $\Cc(Q)$ is short right-complemented. The syntactic right-complement $\vartheta(x, y)$ $(x, y\in Q\smallsetminus\mathcal{E})$ coincides with $(x\star y)^\bullet$.
	Recall that the extension $\vartheta^*$ is constructed by ``filling the grid'', which is 
	almost the same way $\star$ is extended on $\Path(Q)$: therefore, $\vartheta^*(u, v)$ $(u, v\in\Path(Q\smallsetminus\mathcal{E}))$ coincides with $(u\star v)^\bullet$.
	Since $(\Path(Q), \star)$ is a weak RC-system, $\star$ satisfies the RC-law, which induces the sharp $\vartheta$-cube condition, because $(u^\bullet\star v^\bullet)^\bullet=(u\star v)^\bullet$ for $u, v\in\Path(Q)$.
	From Proposition \ref{prop:complemented-implies-lcms} we obtain that $\Cc(Q)$ is left-cancellative,
	admits conditional right-lcms, 
	and its complementation 
	(on the right) of the elements represented by $u, v\in\Path(Q\smallsetminus\mathcal{E})$ is given by 
	the element represented by $(u\star v)^\bullet$.
	
	If two elements $x$ and $y$ admit a right-lcm, then this is unique up to right-multiplication
	by invertible elements (Lemma \ref{lem:lcms-unique-up-to...}): since $\Cc(Q)$ has no nontrivial invertible elements, the right-lcms, when they exist, are unique. This concludes the proof of \textit{iv}.
	
	Noetherianity follows from the fact that the presentation \eqref{presentation} is homogeneous (Proposition \ref{prop:homogeneous-implies-noetherian}). The atoms are exactly the elements represented by paths of length 1, i.e., the elements of $Q\smallsetminus\mathcal{E}$. 
	
	We denote by $E'$ the closure of $Q\smallsetminus\mathcal{E}$ under right-lcms and the complementation $\backslash_R$ (on the right) and by $E$ the closure of $Q\smallsetminus\mathcal{E}$ under right-lcms. By Proposition \ref{prop:closure-is-Garside-family}, $E'$ is a Garside subfamily. Moreover,
	\begin{align*}
		&g\backslash_R\mathrm{lcm}(f_1, f_2, \ldots, f_l)=\mathrm{lcm}(g\backslash_R f_1, g\backslash_R f_2, \ldots,
		g\backslash_R f_l),
		\\
		&(g_1g_2)\backslash_R\mathrm{lcm}(f_1, f_2, \ldots, f_l)=g_2\backslash_R(g_1\backslash_R\mathrm{lcm}(f_1, f_2, \ldots, f_l)),
	\end{align*}
	and $E$ is thereby closed under the complementation $\backslash_R$ on the right.
	By Proposition \ref{prop:closure-is-Garside-family}, this is also the smallest Garside family of $\Cc(Q)$ containing $(Q\smallsetminus\mathcal{E})\cup\mathbf{1}_{\Cc(Q)}$.
\end{proof}
\section{Quiver-theoretic YBMs and their structure categories}\label{section:structure}

\noindent Now we have described the main result about weak RC-systems, it is time to 
apply them to the investigation of the structure category of quiver-theoretic YBMs. 
\subsection{Quiver-theoretic YBMs, non-degeneracy and involutivity}
We now recall from \cite{andruskiewitsch2005quiver, matsumotoshimizu} the notion of quiver-theoretic YBM. We define the non-degeneracy and involutivity properties, and express all conditions in components.
\begin{defin}
	Let $\Aa$ be a quiver over a non-empty set of vertices $\Lambda$. A morphism of quivers $\sigma\colon \Aa\otimes \Aa\to\Aa\otimes\Aa$ is a (quiver-theoretic) \textit{Yang--Baxter map} (YBM) on $\Aa$ if the Yang--Baxter equation (YBE)
	$$(\sigma\otimes\id)(\id\otimes\sigma)(\sigma\otimes\id) = (\id\otimes\sigma)(\sigma\otimes\id)(\id\otimes\sigma)$$
	holds. This is an equation of morphisms $\Aa\otimes\Aa\otimes\Aa\to \Aa\otimes\Aa\otimes\Aa$. We call the pair $(\Aa,\sigma)$ a \textit{braided quiver}. Notice that we do not assume that $\sigma$ is bijective.
\end{defin}
\begin{prop}\label{prop:DYBE-and-I-in-components}
	Let $\Aa$ be a quiver over $\Lambda$, and let $\sigma(x,y) = (x\rightharpoonup y, x\leftharpoonup y)$ define a morphism of quivers $\Aa\otimes\Aa\to\Aa\otimes\Aa $. Then, the YBE for $\sigma$ is rewritten as follows in terms of the components:
	\begin{align}
		\label{yb1}\tag{YB1}&(a\rightharpoonup b )\rightharpoonup \left((a\leftharpoonup b)\rightharpoonup c\right) = a\rightharpoonup (b\rightharpoonup c);\\
		\label{yb2}\tag{YB2}&(a\rightharpoonup b )\leftharpoonup \left( (a\leftharpoonup b)\rightharpoonup c\right) = \left( a\leftharpoonup (b\rightharpoonup c)\right)\rightharpoonup (b\leftharpoonup c);\\
		\label{yb3}\tag{YB3}&(a\leftharpoonup b)\leftharpoonup c = \left( a\leftharpoonup (b\rightharpoonup c)\right) \leftharpoonup (b\leftharpoonup c);
	\end{align}
	for all $a,b,c\in\Aa$ such that $a|b|c$ is a well-defined path.\end{prop}
\begin{proof}
	It is an easy computation.
\end{proof}
\begin{defin}
	A quiver-theoretic YBM $\sigma\colon \Aa\otimes\Aa\to\Aa\otimes\Aa$, described as before by $\sigma(x,y) = (x\rightharpoonup y, x\leftharpoonup y)$, is \textit{left-non-degenerate}
	if the maps $$x\rightharpoonup\blank\colon \Aa(\target(x),\Lambda)\to \Aa(\source(x),\Lambda)$$
	are 1:1 for all $x\in\Aa$. It is \textit{right-non-degenerate} if the maps
	$$\blank\leftharpoonup y\colon \Aa(\Lambda, \source(y))\to \Aa(\Lambda,\target(y))$$
	are 1:1 for all $y\in\Aa$. It is \textit{non-degenerate} if it is both left- and right-non-degenerate.
\end{defin}
\begin{defin}
	A morphism of quivers $\sigma\colon\Aa\otimes\Aa\to\Aa\otimes\Aa$ is \textit{involutive} if $\sigma^2=\id[\Aa\otimes\Aa]$.
\end{defin}

The proof of the following proposition is immediate.
\begin{prop}\label{prop:involutivity-in-components}
	Let $\Aa$ be a quiver over $\Lambda$, and let $\sigma(x,y) = (x\rightharpoonup y, x\leftharpoonup y)$ define a morphism of quivers $\Aa\otimes\Aa\to\Aa\otimes\Aa $. The involutive condition $\sigma^2 = \id[\Aa\otimes \Aa]$ for $\sigma$ is rewritten in components as
	\begin{align}
		\label{i1}\tag{I1}&(a\rightharpoonup b)\rightharpoonup (a\leftharpoonup b) = a;\\
		\label{i2}\tag{I2}&(a\rightharpoonup b)\leftharpoonup (a\leftharpoonup b) = b;
	\end{align}
	for all $a,b\in \Aa$ such that $a|b$ is a well-defined path.
\end{prop}

\subsection{Structure categories and structure groupoids of YBMs}
The notions of structure monoid and structure group, for a set-theoretic YBM, are well known. Here we introduce their straightforward quiver-theoretic analogues: the structure category $\Cc(\sigma)$ and the structure groupoid $\Gg(\sigma)$
of a YBM $\sigma$ on a quiver $\Aa$. Structure groupoids already appear in \cite{andruskiewitsch2005quiver}.
\begin{defin}\label{defin:struct-grp}
	Let $(\Aa, \sigma)$ be a braided quiver, where $\source,\target$ denote the source and target maps of $\Aa$ respectively, and the YBM is given by the morphism $\sigma\colon \Aa\otimes\Aa\to\Aa\otimes\Aa $, $\sigma(x,y) = (x\rightharpoonup y, x\leftharpoonup y)$. The \emph{structure category} $\Cc(\sigma)$ is defined as $\Cc(\sigma) = \langle \Aa\mid R\rangle^+$, where $R$ is the set of all relations $x|y\sim (x\rightharpoonup y)|(x\leftharpoonup y)$ for all $x,y\in\Aa$, $\source(y) = \target(x)$.	The \emph{structure groupoid} $\Gg(\sigma)$ is defined as	$\Gg(\sigma) = \langle \Aa\mid R\rangle$, where $R$ is the set of relations defined above.
\end{defin}

The proof of the following proposition is immediate, using Definitions \ref{def:envGpd} and \ref{defin:struct-grp}.

\begin{prop}\label{lemma:Gsigma_is_Env}
	For a YBM $\sigma$ on a quiver $\Aa$, one has $\Gg(\sigma) \cong \Env(\Cc(\sigma))$.
\end{prop}
\begin{rem}
The canonical isomorphism $\Phi_\sigma\colon \Env(\Cc(\sigma))\to \Gg(\sigma)$ is natural in $\sigma$, in the following sense: if $(\Aa,\sigma)$ and $ (\Bb,\tau)$ are braided quivers over $\Lambda$, and $f\colon \Aa\to \Bb$ is a morphism in $\Cat{Quiv}_\Lambda$ intertwining $\sigma$ and $\tau$, this induces a square
$$\begin{tikzcd}
	\Env(\Cc(\sigma))\ar[d]\ar[r, "\Phi_{\sigma}"]&\Gg(\sigma)\ar[d]\\
	\Env(\Cc(\tau))\ar[r,"\Phi_\tau"]&\Gg(\tau)
\end{tikzcd}$$
and the naturality of $\Phi$ is the commutativity of the above square.
\end{rem}

\section{The interplay between weak RC-systems and YBMs}\label{section:Garside}
\noindent In this section, we establish a connection between suitable YBMs and suitable cyclic systems. In one direction, we prove that left-non-degenerate involutive YBMs provide left-non-degenerate weak RC-systems, and we prove a similar result for weak co-RC-systems. As a converse connection, we prove that categories with a suitable presentation (whose relations incorporate the RC-law) are structure categories of YBMs.

\subsection{From YBMs to weak RC-systems}\label{subsection:YBMandRC}
We construct here weak RC-systems and weak co-RC-systems from suitable YBMs.

\begin{prop}\label{prop:braided_quivers_are_wRCs}
Let $\sigma$ be a YBM on a quiver $\Aa$, where we write $\sigma(x,y) = (x\rightharpoonup y, x\leftharpoonup y)$ as before. Suppose $\sigma$ is left-non-degenerate and involutive. Set $x\star y = (x\rightharpoonup\blank)^{-1}(y)$, where the inverse is well-defined because of the left-non-degeneracy condition. Then, $x\star y$ is defined whenever $\source(x) = \source(y)$, and $(\Aa,\star)$ is a left-non-degenerate weak RC-system.
\end{prop}
\begin{proof}
Before we plunge into the proof of the RC-law, there is one crucial remark to point out. We would like the two squares
\begin{center}
	\begin{tikzpicture}
		\draw[-Stealth] (-2.5,0.5) to node[above] {$x$} (-1.5,0.5);
		\draw[-Stealth] (-2.5,0.5) to node[left] {$y$} (-2.5,-0.5);
		\draw[-Stealth] (-1.5,0.5) to node[right] {$x\star y$} (-1.5,-0.5);
		\draw[-Stealth] (-2.5,-0.5) to node[below] {$y\star x$} (-1.5,-0.5);
		\draw[-Stealth] (1.5,0.5) to node[above] {$x$} (2.5,0.5);
		\draw[-Stealth] (1.5,0.5) to node[left] {$x\rightharpoonup z$} (1.5,-0.5);
		\draw[-Stealth] (2.5,0.5) to node[right] {$z$} (2.5,-0.5);
		\draw[-Stealth] (1.5,-0.5) to node[below] {$x\leftharpoonup z$} (2.5,-0.5);
	\end{tikzpicture}
\end{center}
to be the same. This forces the definition $x\star y = z = (x\rightharpoonup\blank)^{-1}(y)$. However, we now have \textit{two} different definitions of the lower edge $y\star x$: in the left-hand square, it is defined as $y\star x = (y\rightharpoonup\blank)^{-1}(x)$; while, in the right-hand square, it is defined as $x\leftharpoonup z = x\leftharpoonup(x\star y)$. We conclude that, for this definition to make sense, we must prove
\begin{equation}\label{cond:with_y}(y\rightharpoonup\blank)^{-1}(x) =  x\leftharpoonup\left( (x\rightharpoonup\blank)^{-1}(y) \right)\quad \text{for all }x,y\in \Aa, \source(x)=\source(y).\end{equation}
We manipulate the equation as follows:
\begin{align*}
	(y\rightharpoonup\blank)^{-1}(x) =  x\leftharpoonup\left( (x\rightharpoonup\blank)^{-1}(y) \right)\iff& x = y\rightharpoonup\left(x\leftharpoonup\left((x\rightharpoonup\blank)^{-1}(y)\right)\right)\\
	\iff &x = \left(x\rightharpoonup z\right)\rightharpoonup\left(x\leftharpoonup z\right).
\end{align*}
Now, the condition
\begin{equation}\label{cond:with_z} x = \left(x\rightharpoonup z\right)\rightharpoonup\left(x\leftharpoonup z\right) \quad \text{for all }x,y\in \Aa, \target(x)=\source(z)\end{equation}
is Condition \eqref{i1} from Proposition \ref{prop:involutivity-in-components}, which holds true because $\sigma$ is involutive. Notice that $x|z$ is a well-defined path, hence applying \eqref{i1} makes sense. Moreover, since $z=(x\rightharpoonup\blank)^{-1}(y)$ holds, $y$ and $z$ can be obtained from each other uniquely, and consequently \eqref{cond:with_z} holds for all $z$ if and only if \eqref{cond:with_y} holds for all $y$. We have obtained the relation
\begin{equation}\label{eqn:claim}y\star x = x\leftharpoonup z = x\leftharpoonup (x\star y),\end{equation}
which is going to become useful in a moment.

We now turn to the proof of the RC-law. Consider $x,y,z$ sharing the same source. Define
$$a= z,\quad b=z\star y,\quad c=(z\star y)\star (z\star x).$$
We are going to apply the YBE to the path $a|b|c$. 

It is easy to check that the path $r|(r\star s)$ is well-defined for all $r,s\in Q$, $\source(r) = \source(s)$. Therefore, $a|b|c$ is a well-defined path.

We now have to show that the cube in Figure \ref{fig:cube-for-star} closes; cf.~\cite[Proposition XIII.1.34]{dehornoy2015foundations}.

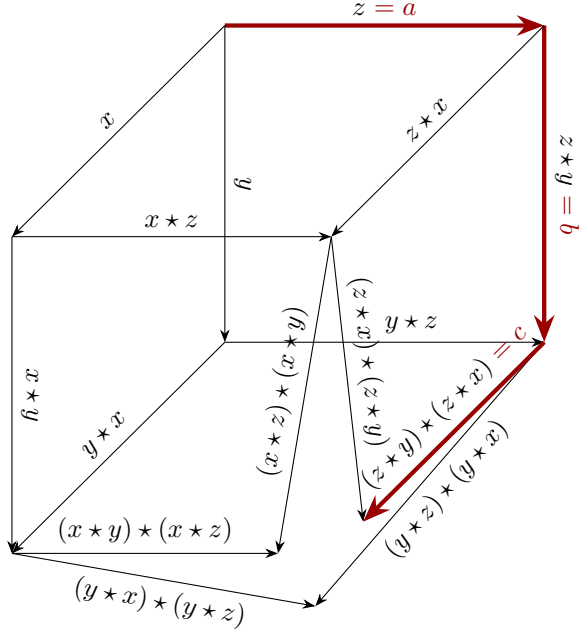
\begin{figure}\centering
	\begin{tikzpicture}[x=0.7cm, y=0.7cm]
		\node[opacity=0, text=gray, text opacity=0.5,minimum width=3cm, minimum height=1cm,transform shape,  yslant=1 ] at (4.3,1.6) {\huge\textbf{face 2}};
		\node[opacity=0, text=gray, text opacity=0.5,minimum width=3cm, minimum height=1cm,transform shape,  xslant=1 ] at (1,4) {\huge\textbf{face 3}};
		\node[opacity=0,text = gray, text opacity=0.5]() at (3,3) {\huge\textbf{face 1}};
		\draw[-Stealth, color=Maroon, ultra thick] (0,6) to node[sloped,above] {\textcolor{black}{$z$} \textcolor{Maroon}{$=a$}} (6,6);
		\draw[-Stealth] (0,6) to node[sloped,above] {$x$} (-4,2);
		\draw[-Stealth] (0,6) to node[sloped,above] {$y$} (0,0);
		\draw[-Stealth, color=Maroon, ultra thick] (6,6) to node[sloped,above] {\textcolor{black}{$z\star y$} \textcolor{Maroon}{$=b$}} (6,0);
		\draw[-Stealth] (6,6) to node[sloped,above] {$z\star x$} (2,2);
		\draw[-Stealth] (-4,2) to node[sloped,above] {$x\star z$} (2,2);
		\draw[-Stealth] (-4,2) to node[sloped,above] {$x\star y$} (-4,-4);
		\draw[-Stealth] (0,0) to node[sloped,above] {$y\star x$} (-4,-4);
		\draw[-Stealth] (0,0) to node[sloped,above] {$\quad \quad y\star z$} (6,0);
		\draw[-Stealth] (-4,-4) to node[sloped,above] {$(x\star y)\star(x\star z)$} (1,-4);
		\draw[-Stealth] (2,2) to node[sloped,above] {$(x\star z)\star(x\star y)$} (1,-4);
		\draw[-Stealth] (2,2) to node[sloped,above] {$(z\star x)\star(z\star y)\;\;\;\;$} (2.6,-3.4);
		\draw[-Stealth] (-4,-4) to node[sloped,below] {$(y\star x)\star(y\star z)$} (1.7,-5);
		\draw[-Stealth] (6,0) to node[sloped,below] {$(y\star z)\star(y\star x)$} (1.7,-5);
		\draw[-Stealth, color=Maroon, ultra thick] (6,0) to node[sloped,above] {\textcolor{black}{\;\;\;$(z\star y)\star (z\star x)$} \textcolor{Maroon}{$=c$} } (2.6,-3.4);
	\end{tikzpicture}
	\caption{Cube relation for $\star$.}\label{fig:cube-for-star}
\end{figure}
The definition of $\star$ yields 
\begin{align}
	\label{face1}\tag{face 1}&& &y = a\rightharpoonup b, &&y\star z = a\leftharpoonup b;\\
	\label{face2}\tag{face 2}&&&z\star x = b\rightharpoonup c,&& (z\star x)\star (z\star y) = b\leftharpoonup c;\\
	\label{face3}\tag{face 3}&&&x= a\rightharpoonup (z\star x) = a\rightharpoonup(b\rightharpoonup c), && x\star z = a\leftharpoonup(b\rightharpoonup c).
\end{align}
Now, by \eqref{yb1}, and by \eqref{face3}, we get $$x = a\rightharpoonup(b\rightharpoonup c) = (a\rightharpoonup b )\rightharpoonup \left((a\leftharpoonup b)\rightharpoonup c\right).$$	
By definition of $\star$, the relations $x = (a\rightharpoonup b )\rightharpoonup \left((a\leftharpoonup b)\rightharpoonup c\right)$ and $y = a\rightharpoonup b$ yield 
$$y\star x = (a\leftharpoonup b)\rightharpoonup c.$$
By \eqref{eqn:claim}, one has
$$x\star y = y\leftharpoonup (y\star x) = (a\rightharpoonup b )\leftharpoonup \left((a\leftharpoonup b)\rightharpoonup c\right).$$
Moreover,
$$x\star y = (a\rightharpoonup b )\leftharpoonup \left((a\leftharpoonup b)\rightharpoonup c\right) =  \left(a\leftharpoonup (b\rightharpoonup c)\right)\rightharpoonup(b\leftharpoonup c)$$
follows from \eqref{yb2}.

We now turn to the face bordered by $x\star z$ and $x\star y$. The equations $x\star z = a\leftharpoonup(b\rightharpoonup c)$ and $x\star y = (a\leftharpoonup(b\rightharpoonup c))\rightharpoonup(b\leftharpoonup c)$ yield, by definition of $\star$,
$$(x\star z)\star (x\star y) = b\leftharpoonup c.$$
In particular, we obtained 
$$(x\star z)\star (x\star y) = b\leftharpoonup c = (z\star x)\star (z\star y).$$
This, by generality of $x,y,z$, is enough to conclude.

All the other properties for $(\Aa, \star)$ to be a weak RC-system are easily verified. 

If $x\star y = x\star y'$ holds, then we have $(x\rightharpoonup\blank)^{-1}(y) = (x\rightharpoonup\blank)^{-1}(y')$, whence $y=y'$: this proves the left-non-degeneracy.
\end{proof}
\begin{prop}
\label{prop:braided_quivers_are_w_co-RCs}
Let $\sigma$ be a YBM on a quiver $\Aa$, where we write  $\sigma(x,y) = (x\rightharpoonup y, x\leftharpoonup y)$ as before. Suppose that $\sigma$ is right-non-degenerate and involutive. Set $x\bullet y = (\blank\leftharpoonup x)^{-1}(y)$, where the inverse is well-defined because of the right-non-degeneracy condition. Then, $x\bullet y$ is defined whenever $\target(x) = \target(y)$, and $(\Aa,\bullet)$ is a left-non-degenerate weak co-RC-system.
\end{prop}
\begin{proof}
On the opposite quiver $\bar{\Aa}$, the morphism $\bar{\sigma}\colon \bar{\Aa}^{\otimes 2}\to \bar{\Aa}^{\otimes 2}$ defined by 
\[  \bar{x}|\bar{z}\mapsto \overline{z\leftharpoonup x}|\overline{z\rightharpoonup x} \]
(where $\bar{\cdot}$ denotes the reverse of the arrows in $\Aa$) is clearly a YBM; and $\overline{x\bullet y} = \bar{x}\,\bar{\star}\, \bar{y}$ where the operation $\bar{\star}$ is defined as in Proposition \ref{prop:braided_quivers_are_wRCs}, but with respect to the YBM $\bar{\sigma}$. One has that $(\Aa,\bullet)$ is a left-non-degenerate weak co-RC-system if and only if $(\bar{\Aa},\bar{\star})$ is a left-non-degenerate weak RC-system. The conclusion, then, follows from Proposition \ref{prop:braided_quivers_are_wRCs}. 
\end{proof}
In force of the previous propositions, the proof of the following result becomes a straightforward verification:
\begin{prop}\label{prop:weakRLCsystem}

Let $\sigma$ be an involutive non-degenerate YBM on a quiver $\Aa$.
If we define $x\;\tilde{\star}\; y = y\bullet x$ as in Proposition $\ref{prop:braided_quivers_are_w_co-RCs}$,
then $(\Aa,\star,\tilde{\star})$ is a weak RLC-system.
\end{prop}

\subsection{YBMs from a class of presented categories}
\label{sec:converse}
We now prove that categories with a suitable presentation produce involutive non-degenerate quiver-theoretic YBMs. This is analogous to Chouraqui's result \cite[Theorem 1 (\textit{ii})]{chouraqui2010garside}, which allows to retrieve a YBM from every Garside group with a suitable presentation. However, unlike \cite[Theorem 1 (\textit{ii})]{chouraqui2010garside}, we do not assume the existence of a Garside structure: this will follow \textit{a posteriori} from Theorem \ref{teo:Garside-structure}
(see also \cite[Proposition  XIII.2.34]{dehornoy2015foundations}).

Suppose that a category $\Cc$ has a presentation $\Cc = \langle \Aa\mid R\rangle^+$, with $\Obj(\Aa) = \Lambda$. We assume 
relations to be ordered ($q\sim r$ and $r\sim q$ are distinct relations),
and for a set of relations $R$, we let $R^{op}$ be the set $\{ r\sim q\mid q\sim r\in R\}$.
The set of redundant relations in $R$ is defined as $R\cap R^{op}$.
We assume the following conditions:
\begin{enumerate}
\item[\textit{i}.] Every relation in $R$ has the form $a|v\sim b|w$ (quadratic relations with $a, b, v, w\in\Aa$). If $x\sim y,\, x\sim y'\in R$ then $y = y'$; if $x\sim y,\, x'\sim y\in R$ then $x = x'$; and if $x\sim y,\, y'\sim x \in R$, then $y=y'$. 
\item[\textit{ii}.] For all $a,b\in \Aa$, with $\source(a) = \source(b)$ and $a\neq b$, there exists a unique relation $a|v\sim b|w$ or $b|w\sim a|v$ in $R$.
\item[\textit{ii}$'$.] For all $a,b\in \Aa$, with $\target(a) = \target(b)$ and $a\neq b$, there exists a unique relation $v|a\sim w|b$ or $w|b\sim v|a$ in $R$.
\item[\textit{iii}.] For all $a\in \Aa$, there exists a unique $z_a\in \Aa(\target(a),\Lambda)$ such that: 
\begin{itemize}
	\item[\textit{iii}a.]if $v\in\Aa(\target(a),\Lambda)\smallsetminus\{ z_a\}$, then there  exist $b\in\Aa(\source(a),\Lambda)\smallsetminus\{ a\}$ and $w\in\Aa$ satisfying $(a|v\sim b|w)\in R\cup R^{op}$;
	\item[\textit{iii}b.]if $(a|z_a\sim b|w)\in R\cup R^{op}$ for some $b, w\in\Aa$, then $b=a$ and $w=z_a$.
\end{itemize}
\item[\textit{iii}$'$.] For all $a\in \Aa$, there exists a unique $z^a\in \Aa(\Lambda, \source(a))$ such that: 
\begin{itemize}
	\item[\textit{iii$'$}a.] if $v\in\Aa(\Lambda, \source(a))\smallsetminus\{ z^a\}$, then there exist $b\in\Aa(\Lambda, \target(a))\smallsetminus\{ a\}$ and $w\in\Aa$ satisfying $(v|a\sim w|b)\in R\cup R^{op}$;
	\item[\textit{iii$'$}b.] if $(z^a|a\sim w|b)\in R\cup R^{op}$ for some $b, w\in\Aa$, then $b=a$ and $w=z^a$.
\end{itemize}
\end{enumerate}
\begin{lem}In the above hypotheses, the following conditions hold.
\begin{enumerate}
	\item[\textit{iv}.]If $(a|v\sim a|w)\in R$ for some $a, v, w\in\Aa$, then $v=w =z_a$. 
	\item[\textit{iv}$'$.]If $(v|a\sim w|a)\in R$ for some $a, v, w\in\Aa$, then $v=w = z^a$.\end{enumerate}
	\end{lem}
	\begin{proof}
We only prove \textit{iv}, since the proof of \textit{iv}$'$ is dual. If $(a|v \sim a|w)\in R$ and $v\neq z_a$, then by \textit{iii}a there exist $b\neq a$ and $w'$ such that $a|v\sim b|w'$; but by \textit{i} the path $a|v$ can only appear in one relation, thus we would have $a|w = b|w'$, contradicting $b\neq a$. Hence $v = z_a$. 
It follows from \textit{iii}b and $(a|z_a\sim a|w)\in R$ that $w=z_a$, which implies \textit{iv}.
\end{proof}
\begin{rem}
Observe that the conditions \textit{i}--\textit{iii$'$} do not imply that relations of the form $a|z_a\sim a|z_a$ exist for all $a$.
\end{rem}

The condition \textit{ii} implies, for all $a\in \Aa$, the existence of a map \[ a\star \blank\colon \Aa(\source(a),\Lambda)\smallsetminus\set{a}\to \Aa(\target(a),\Lambda),\quad b\mapsto a\star b, \]
such that $(a|a\star b\sim b|w)\in R$ for some $w$. We moreover assume:
\begin{enumerate}
\item[\textit{v}.] The operation $\star$ satisfies, for all pairwise distinct $a,b,c\in \Aa$, $\source(a) = \source(b) = \source(c)$, the RC-law
\[ (a\star b)\star (a\star c) = (b\star a)\star (b\star c). \]
\end{enumerate}
Because of the conditions \textit{i} and \textit{iii}, the map $a\star \blank$ defined on $ \Aa(\source(a),\Lambda)\smallsetminus\set{a}$ can be extended to a bijection $a\star '\blank\colon \Aa(\source(a),\Lambda)\to \Aa(\target(a),\Lambda)$, $b\mapsto a\star'b$, defined by
\[ a\star' b = \begin{cases}
a\star b\text{ if }a\neq b,\\
z_a\text{ otherwise.}
\end{cases} \]

Dually, using \textit{i}, \textit{ii}$'$ and \textit{iii}$'$, every $a\in\Aa$ yields a map $a\bullet \blank\colon \Aa(\Lambda,\target(a))\smallsetminus\set{a}\to \Aa(\Lambda,\source(a))$
satisfying $(a\bullet b)|a\sim w|b$ for $b\in\Aa(\Lambda, \target(a))\smallsetminus\{ a\}$, which can be extended to a bijection $a\bullet '\blank \colon\Aa(\Lambda,\target(a))\to \Aa(\Lambda,\source(a))$ defined by
\[ a\bullet ' b= \begin{cases} a\bullet b\text{ if }a\neq b,\\
z^a\text{ otherwise.}\end{cases} \]

The rest of this section is devoted to proving the following result, and to giving some examples of its application. 
\begin{thm}\label{thm:converse}
Let $\Cc = \langle \Aa\mid R\rangle^+$ be a category presentation satisfying the above 
conditions \textit{i}--\textit{v}, and let $\star'$ be defined as above. Then, the map
$\sigma\colon \Aa\otimes\Aa\to \Aa\otimes\Aa$, $\sigma(a|b)= (a\rightharpoonup b)|(a\leftharpoonup b)$, defined by
\[ a\rightharpoonup b = c\text{ if and only if }b = a\star' c,\quad a\leftharpoonup b = (a\rightharpoonup b)\star' a  \]
is an involutive non-degenerate YBM on $\Aa$, whose structure category is $\Cc$.
\end{thm}

In particular, Theorem \ref{thm:converse} will imply that such a category $\Cc$ is perfect Garside, with Garside family as in Propositions \ref{prop:Deltas} and \ref{prop:Deltas1}.
\subsection{Proof of Theorem \ref{thm:converse}}
Before proving the theorem, we need some preliminary results. The first one is a converse to Proposition \ref{prop:braided_quivers_are_wRCs}.
\begin{prop}\label{prop:lnd_weakRCs_is_braiding}
Let $(\Aa, \star')$ be a left-non-degenerate weak RC-system, and let
\[ a\rightharpoonup b = c\text{ if and only if }b = a\star' c,\quad a\leftharpoonup b = (a\rightharpoonup b)\star' a. \]
Then, $\sigma\colon a|b\mapsto (a\rightharpoonup b)|(a\leftharpoonup b)$ is a YBM on $\Aa$.
\end{prop}
\begin{proof}
We check \eqref{yb1}, \eqref{yb2}, and \eqref{yb3}, for a path $a|b|c\in \Aa\otimes \Aa\otimes \Aa$, by direct computation. Let $s= (a\rightharpoonup b)\rightharpoonup((a\leftharpoonup b)\rightharpoonup c)$ and $t= a\rightharpoonup(b\rightharpoonup c)$ be the left- and right-hand side of \eqref{yb1}, respectively. Then, $ (a\leftharpoonup b)\rightharpoonup c = (a\rightharpoonup b)\star's$, thus $((a\rightharpoonup b)\star' a)\rightharpoonup c = (a\rightharpoonup b)\star' s$, whence
\[ c= ((a\rightharpoonup b)\star' a)\star' ((a\rightharpoonup b)\star' s), \]
which by the RC-law is equal to $(a\star'(a\rightharpoonup b))\star'  (a\star' s) = b\star'(a\star's)$. Therefore, $s = a\rightharpoonup (b\rightharpoonup c) =t$ as desired.

Let now $s'= (a\leftharpoonup b)\leftharpoonup c$ and $t'= (a\leftharpoonup (b\rightharpoonup c))\leftharpoonup (b\leftharpoonup c)$ be the left- and right-hand side of \eqref{yb3}, respectively. Then
\begin{align*}
	t' &= \big((a\leftharpoonup (b\rightharpoonup c))\rightharpoonup (b\leftharpoonup c)\big)\star' (a\leftharpoonup (b\rightharpoonup c))\\
	&= \Big(\big( (a\rightharpoonup (b\rightharpoonup c))\star' a \big)\rightharpoonup ((b\rightharpoonup c)\star' b)\Big)\star' \Big((a\rightharpoonup (b\rightharpoonup c))\star' a\Big).
\end{align*}
From the RC-law, one has
\begin{align*}
	(b\rightharpoonup c)\star' b&= \Big( a\star' (a\rightharpoonup (b\rightharpoonup c)) \Big)\star'(a\star' (a\rightharpoonup b))\\
	&= \big((a\rightharpoonup (b\rightharpoonup c))\star' a\big)\star'\big( (a\rightharpoonup (b\rightharpoonup c))\star' (a\rightharpoonup b)\big),
\end{align*}
whence
\[ \big((a\rightharpoonup (b\rightharpoonup c))\star'a\big)\rightharpoonup ((b\rightharpoonup c)\star'b)  = (a\rightharpoonup (b\rightharpoonup c))\star'(a\rightharpoonup b), \]
and consequently
\begin{align*}t'&= \Big(\big( (a\rightharpoonup (b\rightharpoonup c))\star' a \big)\rightharpoonup ((b\rightharpoonup c)\star' b)\Big)\star' \Big((a\rightharpoonup (b\rightharpoonup c))\star' a\Big) \\
	&= \Big((a\rightharpoonup (b\rightharpoonup c))\star'(a\rightharpoonup b)\Big)\star' \Big((a\rightharpoonup (b\rightharpoonup c))\star'a\Big)\\
	&= \big( (a\rightharpoonup b)\star'(a\rightharpoonup(b\rightharpoonup c))\big)\star' ((a\rightharpoonup b)\star'a).\end{align*}
From the proof of \eqref{yb1}, we had
\begin{equation*}((a\rightharpoonup b)\star'a)\rightharpoonup c = (a\rightharpoonup b)\star' s = (a\rightharpoonup b)\star'(a\rightharpoonup (b\rightharpoonup c)), \end{equation*}
which we can plug into the previous expression of $t'$, thus getting
\begin{align*}
	t'&= \Big(((a\rightharpoonup b)\star'a)\rightharpoonup c\Big)\star'((a\rightharpoonup b)\star'a)\\
	&= ((a\rightharpoonup b)\star'a)\leftharpoonup c\\
	&= (a\leftharpoonup b)\leftharpoonup c\\
	&= s',
\end{align*}
and this concludes the proof of \eqref{yb3}. 

Finally, we let $s''= (a\rightharpoonup b)\leftharpoonup ((a\leftharpoonup b)\rightharpoonup c)$ and $t'' =(a\leftharpoonup (b\rightharpoonup c))\rightharpoonup (b\leftharpoonup c)$ be the left- and right-hand side of \eqref{yb2}, respectively.
Then, by using \eqref{yb1},
\begin{align*}
	s'' &= \big((a\rightharpoonup b)\rightharpoonup ((a\leftharpoonup b)\rightharpoonup c)\big)\star'(a\rightharpoonup b)\\
	&= (a\rightharpoonup(b\rightharpoonup c))\star'(a\rightharpoonup b),\\
	t''&= \big((a\rightharpoonup(b\rightharpoonup c))\star'a\big)\rightharpoonup ((b\rightharpoonup c)\star'b),\\
\end{align*}
and from the RC-law
\begin{align*}&\big((a\rightharpoonup (b\rightharpoonup c))\star' a\big)\star' \big((a\rightharpoonup (b\rightharpoonup c))\star'(a\rightharpoonup b)\big)\\
	&= \big(a\star' (a\rightharpoonup (b\rightharpoonup c))\big)\star'(a\star'(a\rightharpoonup b))\\
	&= (b\rightharpoonup c)\star'b.\end{align*}
Thus we obtain
\begin{align*}
	t''&= ((a\rightharpoonup (b\rightharpoonup c))\star'a)\rightharpoonup ((b\rightharpoonup c)\star'b)\\
	&= (a\rightharpoonup (b\rightharpoonup c))\star' (a\rightharpoonup b)\\
	&= s'',
\end{align*}
as desired.
\end{proof}
We now prove the first part of Theorem \ref{thm:converse}: namely, that $\sigma$ is an involutive left-non-degenerate YBM. 

It is clear that $\sigma$ would be left-non-degenerate, because $a\star'\blank$ is bijective for all $a\in\Aa$. Involutivity is also easy (using \eqref{i1} and \eqref{i2}): $(a\rightharpoonup b)\rightharpoonup (a\leftharpoonup b)$ is by definition the unique $c$ such that $(a\rightharpoonup b)\star' c = a\leftharpoonup b = (a\rightharpoonup b)\star'a$, whence $c=a$. From \eqref{i1}, one also obtains \eqref{i2}: because $(a\rightharpoonup b)\rightharpoonup(a\leftharpoonup b)=a$, we get $(a\rightharpoonup b)\leftharpoonup(a\leftharpoonup b)=((a\rightharpoonup b)\rightharpoonup(a\leftharpoonup b))\star'(a\rightharpoonup b)=a\star'(a\rightharpoonup b)=b$.

We now need to prove the YBE for $\sigma$: we do so by proving that $(\Aa,\star')$ is a left-non-degenerate weak RC-system, thus $\sigma$ is a YBM by Proposition \ref{prop:lnd_weakRCs_is_braiding}. We first need a couple of lemmas.
\begin{lem}\label{lem:xxxx}
In the hypotheses of Theorem \ref{thm:converse}, one has $(a\star' b)\bullet' (b\star' a) = a$ for all $a,b\in\Aa$, $\source(a) = \source(b)$.
\end{lem}
\begin{proof}
We first suppose $a = b$. Because of the conditions \textit{iii} and \textit{iii}$'$, 
there is no $b\in\Aa(\Lambda,\target(z_a))\smallsetminus\set{z_a}$ such that $z_a\bullet b = a$. Because $z_a\bullet'\blank\colon \Aa(\Lambda,\target(z_a))\to \Aa(\Lambda,\source(z_a))$ is bijective, this implies $z_a\bullet 'z_a = a$. Since $z_a = a\star' a$, we get the desired formula.

We now prove the formula when $a\neq b$. In this case, $(a\star' b)\bullet' (b\star' a) =(a\star b)\bullet' (b\star a) $. From the condition \textit{iv}$'$, one has $a\star b\neq b\star a$, and hence $(a\star b)\bullet' (b\star a)  = (a\star b)\bullet (b\star a) $, which equals $a$ by condition \textit{ii}$'$.
\end{proof}

The proof of Lemma \ref{lem:xxxx} never uses the condition \textit{v}. Since the set of axioms \textit{i}--\textit{iii}$'$ is self-dual, the following result is also true:
\begin{lem}[dual of Lemma \ref{lem:xxxx}]\label{lem:dualxxxx}
In the hypotheses of Theorem \ref{thm:converse}, one has $(a\bullet 'b)\star' (b\bullet' a)= a$ for all $a,b\in A$, $\target(a) = \target(b)$.
\end{lem}
\begin{cor}\label{cor:bab}
One has $b\bullet' (a\leftharpoonup b) = a$ for all $a,b\in \Aa$, $\target(a) = \source(b)$.
\end{cor}
\begin{proof}
Using Lemma \ref{lem:xxxx}, we compute:
\begin{align*}
	b\bullet' (a\leftharpoonup b) &= b\bullet' ((a\rightharpoonup b)\star 'a)\\
	&= (a\star' (a\rightharpoonup b)) \bullet' ((a\rightharpoonup b)\star 'a)\\
	&= a.
\end{align*}
\end{proof}
\begin{lem}\label{lem:z_star}
In the hypotheses of Theorem \ref{thm:converse}, $z_{a\star b} = z_{b\star a}$ implies $a\star b = b\star a$.
\end{lem}
\begin{proof}
One has
\begin{align*}
	a\star b\overset{(\dagger)}&{=} ((a\star b)\star ' (a\star b)) \bullet ' ((a\star b)\star ' (a\star b))\\
	&= z_{a\star b}\bullet ' z_{a\star b}\\
	&= z_{b\star a}\bullet 'z_{b\star a}\\
	&= ((b\star a)\star' (b\star a))\bullet' ((b\star a)\star'(b\star a))\\
	\overset{(\dagger)}&{=} b\star a,
\end{align*}
where we use Lemma \ref{lem:xxxx} in each equality marked with $(\dagger)$.
\end{proof}
\begin{lem}\label{lem:two_ways}
One has $(a\star b)\star z_a = z_{b\star a}$ for all $a\neq b$ with the same source.
\end{lem}
\begin{proof}
From the condition \textit{ii} one has $\Aa(\target(a\star b), \Lambda) = \Aa(\target(b\star a),\Lambda)$, hence this set can be described in two equivalent ways:
\begin{align*}
	\Aa(\target(a\star b),\Lambda) &= \Big\lbrace  (a\star b)\star (a\star c),\; (a\star b)\star z_a \; \Big|\; c\in \Aa(\source(a),\Lambda)\smallsetminus\set{a,b} \Big\rbrace\sqcup \set{z_{a\star b}}\\
	= \Aa(\target(b\star a),\Lambda)&= \Big\lbrace  (b\star a)\star (b\star c),\; (b\star a)\star z_b \; \Big|\; c\in \Aa(\source(a),\Lambda)\smallsetminus\set{a,b} \Big\rbrace\sqcup \set{z_{b\star a}}.
\end{align*}
Now, the condition \textit{v} implies $(b\star a)\star (b\star c) = (a\star b)\star (a\star c)$; thus, for the two above sets to be equal, one either has
\begin{equation}\label{eq:case1}
	(a\star b)\star z_a = (b\star a)\star z_b\text{ and }z_{a\star b} =z_{b\star a},
\end{equation}
or 
\begin{equation}\label{eq:case2}
	(a\star b)\star z_a = z_{b\star a}\text{ and }z_{a\star b} = (b\star a)\star z_{b}.
\end{equation}
The case \eqref{eq:case1} cannot occur: from Lemma \ref{lem:z_star}, if $z_{a\star b} = z_{b\star a}$ then $a\star b = b\star a$, but then the relation $(a|(a\star b)\sim b|(b\star a)) = (a|(a\star b)\sim b|(a\star b))$ lies in $R$, which by \textit{iv}$'$ would imply $a=b$, which is a contradiction. Therefore, one must have \eqref{eq:case2}, the desired conclusion.
\end{proof}
We now prove that $(\Aa,\star')$ is a left-non-degenerate weak RC-system. The RC-law
\[ (a\star' b)\star' (a\star' c) = (b\star' a)\star' (b\star' c),\]
for $a,b,c$ pairwise distinct, follows from the RC-law for $\star$. If $a=b$, then the RC-law holds trivially. If $a=c\neq b$, then
\[ (a\star'b)\star'(a\star'c) = (a\star b)\star'(a\star'a) = (a\star b)\star' z_a,\]
while, on the other hand,
\[ (b\star' a)\star' (b\star' c) = (b\star a)\star' (b\star a) = z_{b\star a}, \]
and these two are equal from Lemma \ref{lem:two_ways}.

Finally, in case $b=c\neq a$, one has again from Lemma \ref{lem:two_ways}
\begin{align*}
(a\star' b)\star '(a\star'c)&= (a\star b)\star'(a\star b)\\
&= z_{a\star b}\\
&= (b\star a)\star' z_b\\
&= (b\star a)\star' (b\star'b)\\
&= (b\star' a)\star'(b\star' c),
\end{align*}
as desired. Thus $(\Aa,\star')$ is a weak RC-system, and it is clearly left-non-degenerate. This concludes the proof that $\sigma$ is an involutive left-non-degenerate YBM on $\Aa$. 

We now observe that $\sigma$ is also right-non-degenerate. 
Indeed, suppose that $a,a',b\in \Aa$ ($\target(a)=\target(a')=\source(b)$) satisfy
$a\leftharpoonup b = a'\leftharpoonup b$. Then
\[ a = b\bullet' (a\leftharpoonup b) = b\bullet' (a'\leftharpoonup b) = a', \]
from Corollary \ref{cor:bab}. This proves that $\blank\leftharpoonup b\colon \Aa(\Lambda,\source(b))\to \Aa(\Lambda,\target(b))$ is injective. We now show the surjectivity. It follows from the definition of $\rightharpoonup$ and from Lemma \ref{lem:dualxxxx} that
$(b\bullet' v)\rightharpoonup b = v\bullet' b$ for any  $v\in \Aa(\Lambda,\target(b))$, and 
\[ (b\bullet' v)\leftharpoonup b = (v\bullet ' b)\star' (b\bullet' v)= v \]
as a consequence of Lemma \ref{lem:dualxxxx}. Thus
$v$ is in the image of $\blank\leftharpoonup b$. This concludes the proof that $\sigma$ is non-degenerate.

We finally prove that $\Cc$ is the structure category of $\sigma$. We denote by $\bar{R}$ the set of all relations of the form $a|v\sim a|v$ in $R$, and write
$R'=\{ a|b\sim\sigma(a|b)\mid a|b\in\Path_2(\Aa)\}$, which is the relation that defines $\Cc(\sigma)$.
We denote by $\bar{R}'$ the subset of $R'$ consisting of trivial relations of the form $a|b\sim a|b$. 
Observe that $\bar{R}'$ may in principle be larger than $\bar{R}$. If $R$ includes relations of the form $a|z_a \sim a|z_a$ for all $a\in\Aa$, then $\bar{R} = \bar{R}'$.

We can show $(R\smallsetminus\bar{R})\subset R'$. Indeed, if $(a|v\sim a|w)\in R$, then the condition \textit{iv} induces $v=w$, and thus $(a|v\sim a|w)\in \bar{R}$.
If $(a|v\sim b|w)\in R$ $(a\ne b)$, then $v=a\star b$ and $w=b\star a$. Because $\sigma(a|a\star b)=b|b\star a$, $(a|v\sim b|w)\in R'$.
Similarly, $(R'\smallsetminus\bar{R'})\subset R\cup R^{op}$.

It is easy to see that the congruence generated by $R\smallsetminus\bar{R}$ is the same as that generated by $R$ and that the congruence generated by $R'\smallsetminus\bar{R'}$ is the same as that generated by $R'$, because the relations in $\bar{R}$ and $\bar{R'}$ have no effect on the congruences at all
(see Remark \ref{rem:congruence}).
Furthermore, the congruence generated by $R\cup R^{op}$ is the same as that generated by $R$.
Since $(R\smallsetminus\bar{R})\subset R'$ and $(R'\smallsetminus\bar{R'})\subset R\cup R^{op}$, the congruence generated by $R$ is exactly that generated by $R'$, and consequently $\Cc=\Cc(\sigma)$; thus concluding the proof of Theorem \ref{thm:converse}.

\begin{rem}
Notice that $\sigma$ can equivalently be defined  by means of the operation $\bullet'$, as $\sigma((v\bullet' w)|v)= (w\bullet' v)|w$; and the two definitions of $\sigma$ coincide by Lemma \ref{lem:dualxxxx}. This proves that $(\Aa,\bullet')$ is a weak co-RC-system (see Proposition \ref{prop:braided_quivers_are_w_co-RCs}).
\end{rem}
\begin{rem}
For any involutive non-degenerate YBM $\sigma$ on a quiver $\Aa$, 
we consider the structure category $\Cc(\sigma)$ with its usual presentation
$\langle\Aa\mid R\rangle^+$.
The set of relations $R\smallsetminus(R\cap R^{op})$,
from which the redundant relations have been removed,
satisfies the conditions \textit{i}--\textit{iii}$'$ and \textit{v},
and the presented category $\langle\Aa\mid R\smallsetminus(R\cap R^{op})\rangle^+$ coincides with 
$\Cc(\sigma)$.
In addition, the involutive non-degenerate YBM on the quiver $\Aa$ produced by
$\Cc(\sigma)$ in this section is exactly $\sigma$.
\end{rem}

\section{On the Garsideness of the structure category of YBMs}\label{sec:Garside}
\noindent We investigate here the main focus of this paper. Namely, we establish a connection between involutive non-degenerate YBMs, and Garside groupoids. The proof of our result differs from the proof of Chouraqui's theorem, in numerous details. However, the outline of the proof is similar. We shall establish our result by passing through the structure categories of some weak RC-systems.

Our starting point is Proposition \ref{prop:braided_quivers_are_wRCs}, saying that an involutive left-non-degenerate YBM 
$\sigma$ on a quiver $\Aa$ also provides a left-non-degenerate weak RC-system $(\Aa,\star)$. This weak RC-system need not be unital, nor even have a unit family. Indeed, the existence of a unit family requires \textit{at least} that every vertex of the quiver has a loop; and this need not be true in general
(see examples in \S\ref{sec:examples}).

We would like to take the left-non-degenerate weak RC-system $(\Aa,\star)$ 
in Proposition \ref{prop:braided_quivers_are_wRCs}, and consider its completion $(\hat{\Aa},\hat{\star})$.
\begin{rem}\label{rem:Ahat-satisfies-add-cond}
Notice that $\hat{\Aa}$ satisfies the additional condition \eqref{eqn:add_cond}. Indeed, if $x,y$ lie in $\Aa =\hat{\Aa}\smallsetminus\mathcal{E}$ and $x\neq y$, then $x\;\hat{\star}\;y=x\star y$ is defined as $(x\rightharpoonup\blank)^{-1}(y)$ which is an element of $\Aa$, and hence lies in $\hat{\Aa}\smallsetminus\mathcal{E}$. 
\end{rem}
As we have seen in the construction of $Q^\sharp$ and $\hat{Q}$, the completion is not an extension: taking the completion $\hat{Q}$ modifies the operation $\star$, hence possibly modifies the YBM.

The actual scenario is even worse. If $(\Aa,\star)$ is a weak RC-system obtained by a YBM $\sigma$ on $\Aa$, the weak RC-system $(\hat{\Aa},\hat{\star})$ need not be associated with any quiver-theoretic YBM. In the following remark, we are going to see that in \textit{almost all} cases, there is no possible non-degenerate YBM $\hat{\sigma}\colon \hat{\Aa}\otimes\hat{\Aa}\to \hat{\Aa}\otimes \hat{\Aa}$ that can induce $\hat{\star}$.
\begin{rem}Suppose that the 
involutive left-non-degenerate
YBM $\sigma$ is given by $\sigma(x,y)$ $= (x\rightharpoonup y, x\leftharpoonup y)$, as before. We define $x\star y= (x\rightharpoonup\blank)^{-1}(y)$. Then we consider the completion $(\hat{\Aa},\hat{\star})$ and search for an operation $\hat{\rightharpoonup}$, defined on $\hat{\Aa}$, such that $x\;\hat{\star}\; y = (x\;\hat{\rightharpoonup}\;\blank)^{-1}(y)$.

What properties should this $\hat{\rightharpoonup}$ satisfy? The relations $x\;\hat{\star}\;\epsilon_{\source(x)} =\epsilon_{\target(x)}$ and $x\;\hat{\star}\; x=\epsilon_{\target(x)}$
yield, for $\hat{\rightharpoonup}$, the relations $x\,\hat{\rightharpoonup}\,\epsilon_{\target(x)} = \epsilon_{\source(x)}$ and $x\,\hat{\rightharpoonup}\, \epsilon_{\target(x)} = x$.	It is apparent that these two relations are, in most of cases, inconsistent. Since they must hold for all $x$, they imply $x = \epsilon_{\source(x)}$ for all $x$. In other words: $\Aa$ has no arrows at all, and $\hat{\Aa}$ consists of one loop $\epsilon_\lambda$ for each vertex $\lambda$.
\end{rem}
Although the weak RC-system $(\hat{\Aa},\hat{\star})$ is generally not induced by a YBM, this completion is useful to describe the structure category.
\begin{lem}\label{lemma:iso_struct_cat}
Let $\sigma$ be an involutive left-non-degenerate YBM on a quiver $\Aa$.
The structure category $\Cc(\hat{\Aa})$ of $\hat{\Aa}$ 
in Definition \ref{def:structure-category-of-wRCs} is isomorphic to the structure category $\Cc(\sigma)$ of $\sigma$ 
in Definition \ref{defin:struct-grp}.
\end{lem}
\begin{proof}
Recall that $\hat{\Aa}$ is a unital weak RC-system (Lemma \ref{lem:extension-is-weakRCs}). Denote by $\mathcal{E}$ the unit family. By Remark \ref{rem:Ahat-satisfies-add-cond} and Proposition \ref{prop:Garside-family}, $\Cc(\hat{\Aa})$ admits the presentation
\begin{align*}&\!\!\Cc(\hat{\Aa})\\ &= \Big\langle \hat{\Aa}\smallsetminus\mathcal{E}\;\Big|\;  x|(x\,\hat{\star}\, y) \sim y|(y\,\hat{\star}\, x)\text{ for all }x\neq y \text{ in } \hat{\Aa}\smallsetminus \mathcal{E}\text{ such that } x\,\hat{\star}\, y \text{ is defined}\Big\rangle^+\\
	&=\Big\langle \Aa \;\Big|\; x|(x\star y) \sim y|(y\star x)\text{ for all }x\neq y\text{ in } \Aa\text{ such that }x\star y\text{ is defined}\Big\rangle^+\\
	&=\Big\langle \Aa \;\Big|\; x|(x\star y)\sim y|(y\star x)\text{ for all }x, y\in\Aa\text{ such that }x\star y\text{ is defined}\Big\rangle^+. \end{align*}
Notice that if $x=y$, the relation $x| (x\star y)\sim y|(y\star x)$ is trivial. Because $\star$ is left-non-degenerate,
$y\star x=x\leftharpoonup z$ if $z=x\star y$ $(\Leftrightarrow y=x\rightharpoonup z)$, and consequently,
\[
\Cc(\hat{\Aa})=\Big\langle \Aa \;\Big|\; x|z\sim (x\rightharpoonup z)|(x\leftharpoonup z)\text{ for all }x, z\in\Aa\text{ such that }x|z\text{ is defined}\Big\rangle^+. 
\]
This is exactly the definition of $\Cc(\sigma)$.
\end{proof}
Along the proof of Lemma \ref{lemma:iso_struct_cat},
we have also obtained that $\Cc(\sigma)\cong \Cc(\hat{\Aa})$ satisfies the additional condition (\ref{eqn:add_cond}) of Proposition \ref{prop:Garside-family}. Therefore, we obtain the following result
by merging Proposition \ref{prop:Garside-family} and Remark \ref{rem:Ahat-satisfies-add-cond}.
\begin{thm}\label{teo:Garside-structure}
Let $\sigma$ be an involutive left-non-degenerate YBM on a quiver $\Aa$.
Let $j\colon \Aa\to \Cc(\sigma)$ denote the obvious map sending $s\in \Aa$ to $s\in \Path(\Aa)$ regarded as a path of length one, and then sending $s$ to its class modulo the relations $x|y\sim (x\rightharpoonup y)|(x\leftharpoonup y)$.

Then, the map $j$ is an embedding of $\Aa$ into $\Cc(\sigma)$. Moreover, the category $\Cc(\sigma)$
\begin{enumerate}
	\item[\textit{i.}] satisfies a quadratic isoperimetric inequality with respect to the presentation 
	$$\Cc(\sigma) =  \Big\langle \Aa \;\Big|\; x|(x\star y) \sim y|(y\star x)\text{ for all }x\neq y\in \Aa\text{ such that }x\star y\text{ is defined}\Big\rangle^+\text{;}$$ 
	\item[\textit{ii.}] has no nontrivial invertible elements;
	\item[\textit{iii.}] is left-cancellative;
	\item[\textit{iv.}] admits unique conditional right-lcms, and the complementation is given by the operation $\star$;
	\item[\textit{v.}] is Noetherian;
	\item[\textit{vi.}] has a family of atoms given by the elements of $\Aa = \hat{\Aa}\smallsetminus\mathcal{E}$;
	\item[\textit{vii.}] has a Garside family $E$ given by the closure of $\Aa$ under right-lcms; this is the smallest Garside family for $\Cc(\sigma)$ which includes $\Aa$ and $\mathbf{1}_{\Cc(\sigma)}$.
\end{enumerate}
\end{thm}
Moreover, $\Cc(\sigma)$ admits right-lcms, because of the proof of Proposition \ref{prop:Garside-family}.
\subsection{When the structure category is Ore}
\noindent Theorem \ref{teo:Garside-structure} establishes the existence of a Garside structure for $\Cc(\sigma)$. Our next purpose is understanding how this structure reflects onto the structure groupoid $\Gg(\sigma)$. Given a category $\Cc$, we recall from Proposition \ref{prop:envgroupoid} that the \emph{enveloping groupoid} $\Env(\Cc)$ is the category
\begin{align*}\Env(\Cc)=\Big\langle \Cc\cup \bar{\Cc} \;\Big| \; x|y\sim xy\text{ for all composable }x,y\in\Cc\smallsetminus\mathbf{1}_\Cc;\Big.\\ \Big.x \bar{x}\sim\varepsilon_{\source(x)},\, \bar{x}x\sim\varepsilon_{\target(x)},\; \mathbf{1}_\lambda\sim \varepsilon_\lambda\Big\rangle^+,\end{align*}
where $\bar{\Cc}$ here denotes as usual the opposite category of $\Cc$. The relations $x \bar{x}\sim\varepsilon_{\source(x)}$ and $ \bar{x}x\sim\varepsilon_{\target(x)}$ imply that the equivalence  class of $\bar{x}$ is the unique inverse of the equivalence class of $x$, hence in $\Env(\Cc)$ we shall harmlessly make confusion between the notation $\bar{x}$ and the notation $x^{-1}$.

Recall from Definition \ref{defin:Ore} the notion of \textit{left-Ore category}. From Proposition \ref{prop:Ore}, we can embed left-Ore categories into their enveloping groupoid. 

If a category $\Cc$ is left-Ore and admits left-lcms, and $\Ss$ is a Garside family in $\Cc$, then the enveloping groupoid $\Env(\Cc)$ inherits most of the Garside structure of $\Cc$. Namely, a \textit{symmetric $\Ss$-normal decomposition} is defined on $\Cc(\sigma)\Cc(\sigma)^{-1}\subseteq \Env(\Cc)$ \cite[Proposition III.2.20]{dehornoy2015foundations}.
\begin{prop}\label{prop:is-Ore}
The structure category $\Cc=\Cc(\sigma)$ of an involutive non-degenerate YBM is left-Ore, and it admits left-lcms. As a consequence, $\Cc(\sigma)$ is embedded into $\Env(\Cc)$, and every element of $\Cc\Cc^{-1}\subseteq \Env(\Cc)$ has a symmetric normal decomposition mutuated from the Garside normal form of $\Cc$.
\end{prop}
\begin{proof}
By Theorem \ref{teo:Garside-structure}, $\Cc=\Cc(\sigma)$ is left-cancellative. 

Using the map $\bar{\sigma}$ on $\bar{\Aa}$ and the weak RC-system $(\bar{\Aa},\bar{\star})$, as in the proof of Proposition \ref{prop:braided_quivers_are_w_co-RCs}, we obtain immediately that $\Cc(\bar{\sigma})$ admits right-lcms, as was mentioned after Theorem \ref{teo:Garside-structure}. Because $\Cc(\sigma)$ is contravariantly isomorphic to $\Cc(\bar{\sigma})$, we conclude that $\Cc(\sigma)$ admits left-lcms.

\end{proof}
\subsection{Garside families for involutive non-degenerate YBMs}
Let $\sigma$ be an involutive non-degenerate YBM on a quiver $\Aa$. By Proposition \ref{prop:is-Ore}, the
structure category $\Cc(\sigma)$ is left-Ore and it admits left-lcms. Recall from \S\ref{subsection:YBMandRC} that $(\Aa, \star, \tilde{\star})$ is a weak RLC-system, with $x\star y=(x\rightharpoonup\blank)^{-1}(y)$ 
$(\source(x)=\source(y))$ and
$x\,\tilde{\star}\,y=(\blank\leftharpoonup y)^{-1}(x)$
$(\target(x)=\target(y))$.

We now give a concrete way to describe the Garside family $E$ of 
$\Cc(\sigma)$ in Theorem
\ref{teo:Garside-structure}.
This description shows that the Garside family $E$ is \textit{perfect} (see Definition \ref{def:perfect}). 

Recall from \cite[Definition III.2.28]{dehornoy2015foundations} that, given a category $\Cc$ and a generating subfamily $\Ss\subseteq \Cc$, a \emph{left-lcm witness} $\tilde{\theta}$ on the subfamily
$\Ss^\sharp = \Ss\Cc^\times \cup \Cc^\times$ (see \S\ref{subsection:Garside})
is a partially defined map from $\Ss^\sharp\times \Ss^\sharp$
to $\Path(\Ss^\sharp)$ satisfying: (1) for all $s, t\in \Ss^\sharp$,
the elements $\tilde{\theta}(s, t)$ and $\tilde{\theta}(t, s)$ exist, if and only if 
$s$ and $t$ admits a left-lcm;
and (2) in this case, $\tilde{\theta}(s, t)t=\tilde{\theta}(t, s)s$ is a left-lcm of $s$ and $t$.
We moreover assume that $\tilde{\theta}(s,s)$ is always defined; and the properties of the left-lcm witness clearly imply $\tilde{\theta}(s, s)=\mathbf{1}_{\source(s)}$ for all $s\in\Ss^\sharp$.
A left-lcm witness $\tilde{\theta}$ on
$\Ss^\sharp$ is called \emph{short} if $\tilde{\theta}(s, t)$ belongs to $\Ss^\sharp$ or is empty for all $s, t$.

\begin{defin}[{\cite[Definition III.3.6]{dehornoy2015foundations}}]\label{def:perfect}
Let $\Cc$ be a left-Ore category that admits left-lcms.
A Garside family $\Ss$ in $\Cc$ is \emph{perfect}, if there exists a short left-lcm witness $\tilde{\theta}$ on
$\Ss^\sharp$ such that $\tilde{\theta}(s,t)t$ lies in $\Ss^\sharp$ for all $s,t$ with the same target.
\end{defin}   

Because the structure category $\Cc(\sigma)$ has no nontrivial invertible elements
(Theorem \ref{teo:Garside-structure}), we have $\Ss^\sharp=\Ss\cup \mathbf{1}_\Aa$ for all subfamilies $\Ss\subset\Cc(\sigma)$. Here, $\mathbf{1}_\Aa=\{\mathbf{1}_\lambda \mid \lambda\in \Obj(\Aa)\}\subset\Cc(\sigma)$.

Let $\lambda\in\Lambda = \Obj(\Cc(\sigma))$ and let $I\neq\emptyset$ be a finite subset of $\Aa(\lambda,\Lambda)$.
There exists the right-lcm $\Delta_I$ of the set $I$, which can be computed explicitly
by means of the \emph{RC-calculus} introduced by Dehornoy \cite{DehornoyRCcalculus}; see also \cite[\S XIII.2.2]{dehornoy2015foundations}. 
For $I = \set{x_1,\dots, x_n}$
with $x_i\ne x_j$ $(i\ne j)$, we define as in \cite{DehornoyRCcalculus}
\[ \Omega_1(x_1) = x_1, \quad \Omega_i(x_1,\dots, x_i) = \Omega_{i-1}(x_1,\dots, x_{i-1})\star \Omega_{i-1}(x_1,\dots, x_{i-2}, x_i)\ (2\leq i\leq n). \]

Observe that these $\Omega_i(x_1,\dots, x_i)$ for $i=1, \ldots n$
are well defined, because the source of $\Omega_i(x_1, \ldots, x_i)$
is the same as that of $\Omega_i(x_1, \ldots, x_{i-1}, x_{i+1})$ 
for any $i=1, \ldots, n-1$.
They satisfy,
for every permutation $\pi$ in the symmetric group $\mathfrak{S}_{i-1}$ $(i\geq 2)$,
\begin{equation}\label{eq:omegapermutation}
\Omega_i(x_{\pi(1)}, \ldots, x_{\pi(i-1)}, x_i)=\Omega_i(x_1, \ldots, x_i),
\end{equation}
and
$\Delta_I=\Delta_n(x_1, \ldots, x_n)=
[\Omega_1(x_1) | \Omega_2(x_1, x_2)|\dots |\Omega_n(x_1, \dots, x_n)]\in\Cc(\sigma)$
is the 
the right-lcm of the finite set $I$, with the same argument as in \cite[Lemma 3.3]{DehornoyRCcalculus}; see also Figure \ref{fig:Delta}. The order of the arrows $x_1, \dots, x_n$ is irrelevant, since $\Delta_n$ is a symmetric function of them (cf.\@ \cite[Lemma 2.6]{DehornoyRCcalculus}). \begin{rem}
The rule
\[ \Delta_2(x,y)\backslash_R [z] = [x|x\star y]\backslash_R [z] = [(x\star y)\star (x\star z)]  \]
is a compatibility between the product and the weak RC-system structure in $\Cc(\sigma)$, for pairwise distinct $x,y,z\in\Aa$ (cf.\@ \cite[Proposition II.2.15]{dehornoy2015foundations}). Moreover, this allows us to rewrite $[\Omega_i(x_1,\dots, x_i)] = \Delta_{i-1}(x_1,\dots, x_{i-1})\backslash_R [x_i]$ when $x_j\neq x_k$ for all $j\neq k$.
\end{rem} 

Because of Theorem \ref{teo:Garside-structure},
the Garside family $E$ is given by the closure of $\Aa$ under right-lcms. For all finite nonempty subsets $I$ and $J$ of $\Aa(\lambda, \Lambda)$, one has that $\Delta_{I\cup J}$ is the right-lcm of $\Delta_I$ and $\Delta_J$, thus
the following is immediate.
\begin{prop}\label{prop:Deltas}
The Garside family $E$ can be described as $E=\{ \Delta_I\in\Cc(\sigma)\mid  \lambda\in \Lambda,\;I\subset\Aa(\lambda,\Lambda),\; 1\le |I|<+\infty\} \cup \mathbf{1}_\Aa$.
\end{prop}
\begin{figure}[t]
\centering
\begin{tikzpicture}[x=3cm, y=3cm]
	\draw[-Stealth] (-0.5,0) to node[above] {$x_2$} (0,0);
	\draw[-Stealth] (0,0.5) to node[above]{$\Omega_2(x_1, x_2)$} (0.5,0.5);
	\draw[-Stealth] (0,-0.5) to (0.5,-0.5);
	\draw[-Stealth] (0.5,0) to (1,0);
	\draw[-Stealth] (0,0) to (0.5,0.5);
	\draw[-Stealth] (0,0) to (0.5,-0.5);
	\draw[-Stealth] (0.5,0.5) to node[right]{$\Omega_3(x_1, x_2, x_3)$} (1,0);
	\draw[-Stealth] (0.5,-0.5) to (1,0);
	\draw[-Stealth] (-0.5,0) to node[left] {$x_1(=\Omega_1(x_1))\;$} (0,0.5);
	\draw[-Stealth] (-0.5,0) to node[left] {$x_3$} (0,-0.5);
	\draw[-Stealth] (0,0.5) to node[right] {} (0.5,0);
	\draw[-Stealth] (0,-0.5) to node[right] {} (0.5,0);
	\node() at(-0.6,0) {$\lambda$};
\end{tikzpicture}
\caption{The definition of $\Delta_3(x_1, x_2, x_3)$ $(\source(x_1)=\source(x_2)=\source(x_3))$.}
\label{fig:Delta}
\end{figure}

We observe that $\Delta_I$ is also the left-lcm of a suitable set $\set{\tilde{x}_1,\dots, \tilde{x}_n}$, that we are going to construct. Indeed, for any finite set $J = \set{y_1,\dots, y_n}\subset\Aa(\Lambda,\mu)$
with $y_i\ne y_j$ $(i\ne j)$, we define
\[ \widetilde{\Omega}_1(y_1) = y_1, \quad \widetilde{\Omega}_j(y_1,\dots, y_j) = \widetilde{\Omega}_{j-1}(y_1, y_3, \dots, y_j)\,\tilde{\star}\, \widetilde{\Omega}_{j-1}(y_2,\dots, y_j)\ (2\leq j\leq n), \]
and
$\widetilde{\Delta}_J=\widetilde{\Delta}_n(y_1, \ldots, y_n)=
[\widetilde{\Omega}_n(y_1, \dots, y_n) | \widetilde{\Omega}_{n-1}(y_2, \ldots, y_n)|\dots |\widetilde{\Omega}_1(y_n)]\in\Cc(\sigma)$.
By a dual argument, $\widetilde{\Delta}_J$ is the left-lcm of $J$.

If we set
$\tilde{x}_i=\Omega_n(x_1, \ldots, \hat{x}_i, \ldots, x_n, x_i)$ $(i=1, \ldots, n)$, then the $\tilde{x}_i$'s for $i=1, \ldots, n$ are pairwise distinct, and $\target(\tilde{x}_1)=\cdots =\target(\tilde{x}_n)$.
Indeed, for $i\ne j$,
\[
\tilde{x}_i=\Omega_n(x_1, \ldots, \hat{x}_i, \ldots, \hat{x}_j, \ldots, x_n, x_j, x_i)
\]
from \eqref{eq:omegapermutation}, and
consequently
\begin{align*}
\target(\tilde{x}_i)=&\target(\Omega_n(x_1, \ldots, \hat{x}_i, \ldots, \hat{x}_j, \ldots, x_n, x_j, x_i))
\\
=&\target(\Omega_{n-1}(x_1, \ldots, \hat{x}_i, \ldots, \hat{x}_j, \ldots, x_n, x_j)\star\Omega_{n-1}(x_1, \ldots, \hat{x}_i, \ldots, \hat{x}_j, \ldots, x_n, x_i))
\\
=&
\target(\Omega_{n-1}(x_1, \ldots, \hat{x}_i, \ldots, \hat{x}_j, \ldots, x_n, x_i)\star\Omega_{n-1}(x_1, \ldots, \hat{x}_i, \ldots, \hat{x}_j, \ldots, x_n, x_j))
\\
=&\target(\tilde{x}_j),
\end{align*}
because $(\Aa, \star)$ is a weak RC-system.

Analogously to the proofs of \cite[Lemma 3.3]{DehornoyRCcalculus}
and \cite[Lemma XIII.2.27]{dehornoy2015foundations}, one has \begin{equation}\label{eq:omega}\Omega_i(x_1, \ldots, x_i)=\widetilde{\Omega}_{n+1-i}(\tilde{x}_i, \ldots, \tilde{x}_n).\end{equation} For each $n$ we can prove it by induction on $i=1, \ldots,n$,
using the fact that $(A, \star, \tilde{\star})$ is a weak RLC-system.
If $i=n$, then \eqref{eq:omega} holds trivially.
For $i< n$, we write $s=\Omega_i(x_1, \ldots, x_i)$,
$s'=\Omega_i(x_1, \ldots, x_{i-1}, x_{i+1})$,
$t=\Omega_{i+1}(x_1, \ldots, x_i, x_{i+1})$,
and $t'=\Omega_{i+1}(x_1, \ldots, x_{i-1}, x_{i+1}, x_i)$, and assume
\[ t=\widetilde{\Omega}_{n-i}(\tilde{x}_{i+1}, \ldots, \tilde{x}_n),\quad 
t'=\widetilde{\Omega}_{n-i}(\tilde{x}_i, \tilde{x}_{i+2}, \ldots, \tilde{x}_n),\]
which are the inductive hypotheses. Then $s\star s'=t$ and $s'\star s=t'$ by the definition of $\Omega_i$. It follows
from Definition \ref{definition:RLCsystem} that $t'\,\tilde{\star}\,t=s$,
and consequently
$\Omega_i(x_1, \ldots, x_i)=s=t'\,\tilde{\star}\,t=\widetilde{\Omega}_{n+1-i}(\tilde{x}_i,\ldots, \tilde{x}_n)$, by the definition of $\widetilde{\Omega}_j$.
This completes the proof of \eqref{eq:omega}.

Hence,
$\Delta_n(x_1, \ldots, x_n)=\widetilde{\Delta}_n(\tilde{x}_1, \ldots, \tilde{x}_n)$. Therefore, $\Delta_I$ $(I=\{x_1, \ldots, x_n\})$
is the left-lcm of $\{\tilde{x}_1, \ldots, \tilde{x}_n\}$.

On the other hand, if we set  $\tilde{y}_j=\widetilde{\Omega}_n(y_j, y_1, \ldots, \hat{y}_j, \ldots, y_n)$
$(j=1, \ldots, n)$, then the $\tilde{y}_j$'s for $j=1, \ldots, n$ are pairwise distinct, $\source(\tilde{y}_1)=\cdots =\source(\tilde{y}_n)$,
and 
\begin{equation}\label{eqn:tildeomega}
\widetilde{\Omega}_j(y_j, \ldots, y_1)=\Omega_{n+1-j}(\tilde{y}_n, \ldots, \tilde{y}_j),\quad 
\widetilde{\Delta}_n(y_1, \ldots, y_n)=\Delta_n(\tilde{y}_1, \ldots, \tilde{y}_n).
\end{equation}
The left-lcm of any finite nonempty set $ J\subset\Aa(\Lambda, \mu)$ 
is thus an element of the Garside family $E$. This proves the following.
\begin{prop}\label{prop:Deltas1}
One has $E=\{ \widetilde{\Delta}_I\in\Cc(\sigma)\mid \mu\in \Lambda,\; I\subset \Aa(\Lambda,\mu),\; 1\le|I|<+\infty\}\cup \mathbf{1}_\Aa$.
\end{prop}

Now we show that the Garside family $E$ is perfect.
Let $f, g\in E\smallsetminus\mathbf{1}_\Aa$ with the same target $\mu\in \Lambda$.
There exist finite subsets $I, J\subset\Aa(\Lambda, \mu)$ 
satisfying $f=\widetilde{\Delta}_I$ and $g=\widetilde{\Delta}_J$.
Then $\widetilde{\Delta}_{I\cup J}$ is the left-lcm of $f=\widetilde{\Delta}_I$ and $g=\widetilde{\Delta}_J$.
Hence, sending the pair $(f,g)$ to the element
$\widetilde{\Delta}_J\backslash_L\widetilde{\Delta}_{I\cup J}$ yields a left-lcm witness on $E$.

Let $J=\{s_{m+1}, \ldots, s_n\}$
and 
$I\cup J=\{ s_1, \ldots, s_m, s_{m+1}, \ldots, s_n\}$
$(m\leq n)$, where the $s_i$'s are all distinct.
Since
\begin{align*}
\widetilde{\Delta}_J&=\widetilde{\Delta}_{n-m}(s_{m+1}, \ldots, s_n)\\
&=[\widetilde{\Omega}_{n-m}(s_{m+1}, \dots, s_n) | \widetilde{\Omega}_{n-m-1}(s_{m+2}, \ldots, s_n)|\dots |\widetilde{\Omega}_1(s_n)],\\
\widetilde{\Delta}_{I\cup J}&=
\widetilde{\Delta}_n(s_1, \ldots, s_n)\\
&=[\widetilde{\Omega}_{n}(s_1, \dots, s_n) | \widetilde{\Omega}_{n-1}(s_2, \ldots, s_n)|\dots |\widetilde{\Omega}_1(s_n)]\\
&=[\widetilde{\Omega}_{n}(s_1, \dots, s_n) |\dots | \widetilde{\Omega}_{n-m+1}(s_m, \ldots, s_n)]\widetilde{\Delta}_J,
\end{align*}
we have $
\widetilde{\Delta}_J\backslash_L\widetilde{\Delta}_{I\cup J}
=[\widetilde{\Omega}_{n}(s_1, \dots, s_n) |\dots | \widetilde{\Omega}_{n-m+1}(s_m, \ldots, s_n)]$.
Because of (\ref{eqn:tildeomega}),
\[
\widetilde{\Delta}_J\backslash_L\widetilde{\Delta}_{I\cup J}
=[\Omega_1(\tilde{s}_1) |\Omega_2(\tilde{s}_1, \tilde{s}_2)|\dots | \Omega_m(\tilde{s}_1, \ldots, \tilde{s}_m)]=\Delta_m(\tilde{s}_1, \ldots, \tilde{s}_m).
\]
It follows from Proposition \ref{prop:Deltas}
that this is an element of the Garside family $E$,
and $\widetilde{\Delta}_J\backslash_L\widetilde{\Delta}_{I\cup J}$
gives 
a short left-lcm witness of $f$ and $g$ as a result.
By Proposition \ref{prop:Deltas1}, $(\widetilde{\Delta}_J\backslash_L\widetilde{\Delta}_{I\cup J})\widetilde{\Delta}_J
=\widetilde{\Delta}_{I\cup J}$ is an element of $E$.
We have proven the following. 
\begin{cor}
The Garside family $E$ is perfect. 
\end{cor}

Finally we give a sufficient condition for the Garside family $E$ to
satisfy $E\ne\Cc(\sigma)$.
If $x, y\in\Aa$ 
satisfy $\source(x)=\source(y)$, then
$\source(x\star y)=\target(x)$,
and $x\star y\in\Aa(\target(x),\Lambda)$ as a result.
Because the map $x\star \blank: \Aa(\source(x),\Lambda)\to\Aa(\target(x),\Lambda)$
is bijective and $\Aa\ne\emptyset$,
there exists an element of $\Cc(\sigma)$ of any (finite) length.

Since the length of $\Delta_I$ is $|I|$ by its definition,
an immediate consequence is the following.
\begin{cor}\label{cor:bounded}
Suppose that there exists a finite number $n$ such that
$|\Aa(\lambda,\Lambda)|\leq n$ for all $\lambda\in\Lambda$. Then, the length of any element of the Garside family $E$ is bounded by $n$. In particular, $E\subsetneq \Cc(\sigma)$.
\end{cor}

\section{Examples of solutions and their structure categories}  \label{sec:examples}
We now see some examples of braided quivers, obtained by applying Theorem \ref{thm:converse}. Our examples will be \textit{Schurian} quivers, i.e., quivers $\Aa$ such that $|\Aa(\lambda,\mu)|\le 1$ for all pairs of vertices $(\lambda,\mu)$. For a Schurian quiver, we adopt the notation $[\lambda,\mu]$ to signify the unique (if any) arrow $\lambda\to \mu$.
\begin{ex}\label{ex:pres_0}
We consider the Schurian quiver $\Aa$ over $8$ vertices $\Lambda= \{ 1,\dots, 8 \}$, with $24$ arrows: \begin{align*} &[1, 2], [2, 1], [2, 3], [3, 2], [3, 4], [4, 3], [4, 1], [1,
	4], [5, 6], [6, 5], [6, 7], [7, 6],\\ & [7, 8], [8, 7], [8, 5], [5, 8],
	[1, 5], [5, 1], [4, 8], [8, 4], [2, 6], [6, 2], [3, 7], [7, 3];\end{align*} and relations 
\begin{align*}
	&[1, 2][2, 3]\sim [1, 4][4, 3], && [1, 2][2, 6]\sim [1, 5][5, 6], &&
	[1, 4][4, 8]\sim [1, 5][5, 8],\\ & [2, 1][1, 5]\sim [2, 6][6, 5], &&
	[2, 1][1, 4]\sim [2,
	3][3, 4],&& [2, 3][3, 7]\sim [2, 6][6, 7],\\ & [3, 2][2, 1]\sim [3, 4][4, 1],&& [3,
	2][2, 6]\sim [3, 7][7, 6], && [3, 4][4, 8]\sim [3, 7][7, 8],\\ & [4, 1][1, 2]\sim [4,
	3][3, 2],&& [4, 1][1, 5]\sim [4, 8][8, 5],&& [4, 3][3, 7]\sim [4, 8][8, 7],\\
	&[5,
	1][1, 2]\sim [5, 6][6, 2], && [5, 1][1, 4]\sim [5, 8][8, 4],&& [5, 6][6, 7]\sim [5,
	8][8, 7], \\ & [6, 2][2, 1]\sim [6, 5][5, 1], &&
	[6, 2][2, 3]\sim [6, 7][7, 3], && [6,
	5][5, 8]\sim [6, 7][7, 8],\\ & [7, 3][3, 2]\sim [7, 6][6, 2],&& [7, 3][3, 4]\sim [7,
	8][8, 4],&&
	[7, 6][6, 5]\sim [7, 8][8, 5],\\ & [8, 4][4, 1]\sim [8, 5][5, 1],&& [8,
	4][4, 3]\sim [8, 7][7, 3],&& [8, 5][5, 6]\sim [8, 7][7, 6].
\end{align*}

The shape of $\Aa$ is depicted in Figure \ref{fig:ex.pres.0}. All the hypotheses of Theorem \ref{thm:converse} are easily verified for this presentation.
\end{ex}
\begin{figure}[t]
\centering
\begin{tikzpicture}[x=2cm, y=2cm]
	\draw[-Stealth] (0,0) to[bend left=10] (1,0);
	\draw[-Stealth] (1,0) to[bend left=10] (0,0);
	\draw[-Stealth] (0,0) to[bend left=10] (0,1);
	\draw[-Stealth] (0,1) to[bend left=10] (0,0);
	\draw[-Stealth] (0,0) to[bend left=10] (0.3,0.2);
	\draw[-Stealth] (0.3,0.2) to[bend left=10] (0,0);
	\draw[-Stealth] (0.3,0.2) to[bend left=10] (1.3,0.2);
	\draw[-Stealth] (1.3,0.2) to[bend left=10] (0.3,0.2);
	\draw[-Stealth] (0.3,0.2) to[bend left=10] (0.3,1.2);
	\draw[-Stealth] (0.3,1.2) to[bend left=10] (0.3,0.2);
	\draw[-Stealth] (0,1) to[bend left=10] (0.3,1.2);
	\draw[-Stealth] (0.3,1.2) to[bend left=10] (0,1);
	\draw[-Stealth] (1,0) to[bend left=10] (1.3,0.2);
	\draw[-Stealth] (1.3,0.2) to[bend left=10] (1,0);
	\draw[-Stealth] (0,1) to[bend left=10] (1,1);
	\draw[-Stealth] (1,1) to[bend left=10] (0,1);
	\draw[-Stealth] (1,0) to[bend left=10] (1,1);
	\draw[-Stealth] (1,1) to[bend left=10] (1,0);
	\draw[-Stealth] (1,1) to[bend left=10] (1.3,1.2);
	\draw[-Stealth] (1.3,1.2) to[bend left=10] (1,1);
	\draw[-Stealth] (1.3,1.2) to[bend left=10] (1.3,0.2);
	\draw[-Stealth] (1.3,0.2) to[bend left=10] (1.3,1.2);
	\draw[-Stealth] (1.3,1.2) to[bend left=10] (0.3,1.2);
	\draw[-Stealth] (0.3,1.2) to[bend left=10] (1.3,1.2);
\end{tikzpicture}\caption{The quiver in Example \ref{ex:pres_0}.}\label{fig:ex.pres.0}
\end{figure}
Let $\Aa$ be a quiver over $\Lambda$, such that $|\Aa(\lambda,\Lambda)|\le 2$ for all $\lambda\in\Lambda$. In this case, condition \textit{v} from the hypotheses of Theorem \ref{thm:converse} is automatically satisfied. We present two instances of this situation.
\begin{ex}\label{ex:pres_1}
We consider the quiver $\Aa$ with vertices $\Lambda=\{ 1, 2, 3 \}$; arrows \[  [1, 2], [2, 1], [2, 3], [3, 2], [3, 1], [1, 3], \]
see Figure \ref{fig:ex.pres.1}. The category generated by $\Aa$ with relations $[1, 2][2, 1]\sim [1, 3][3, 1],$ $ [2, 3][3,
2]\sim [2, 1][1, 2],$ $[3, 1][1, 3]\sim[3, 2][2, 3]$ satisfies the hypotheses of Theorem \ref{thm:converse}.
\end{ex}
\begin{figure}[t]
\centering
\begin{tikzpicture}[x=2cm, y=2cm]
	\draw[-Stealth] (0,0) to[bend left=10] (1,0);
	\draw[-Stealth] (1,0) to[bend left=10] (0,0);
	\draw[-Stealth] (0,0) to[bend left=10] (0.5,0.86);
	\draw[-Stealth] (0.5,0.86) to[bend left=10] (0,0);
	\draw[-Stealth] (1,0) to[bend left=10] (0.5,0.86);
	\draw[-Stealth] (0.5,0.86) to[bend left=10] (1,0);
\end{tikzpicture}\caption{The quiver in Example \ref{ex:pres_1}.}\label{fig:ex.pres.1}
\end{figure}
\begin{ex}\label{ex:pres_2}
We consider the quiver $\Aa$ with vertices $\Lambda=\{ 1, 2, 3, 4 \}$, and arrows $\{ [1, 2], [2, 1], [2, 3], [3,
2], [3, 4], [4, 3], [4, 1], [1, 4] \}$; see Figure \ref{fig:ex.pres.2}. The category generated by $\Aa$ with relations $[1, 2][2, 3]\sim [1,
4][4, 3],$ $ [2, 3][3, 4]\sim [2, 1][1, 4],$ $ [3, 4][4, 1]\sim[3, 2][2, 1],$ $ [4,
1][1, 2]\sim[4, 3][3, 2]$ satisfies the hypotheses of Theorem \ref{thm:converse}.
\end{ex}
\begin{figure}[t]
\centering
\begin{tikzpicture}[x=2cm, y=2cm]
	\draw[-Stealth] (0,0) to[bend left=10] (1,0);
	\draw[-Stealth] (1,0) to[bend left=10] (0,0);
	\draw[-Stealth] (0,0) to[bend left=10] (0,1);
	\draw[-Stealth] (0,1) to[bend left=10] (0,0);
	\draw[-Stealth] (1,1) to[bend left=10] (1,0);
	\draw[-Stealth] (1,0) to[bend left=10] (1,1);
	\draw[-Stealth] (1,1) to[bend left=10] (0,1);
	\draw[-Stealth] (0,1) to[bend left=10] (1,1);
\end{tikzpicture}\caption{The quiver in Example \ref{ex:pres_2}.}\label{fig:ex.pres.2}
\end{figure}
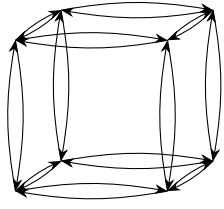

\section{Solutions of principal homogeneous type}\label{sec:PH-type} 
\noindent We now apply our theory to a special class of 
(quiver-theoretic) YBMs, the solutions of \emph{principal homogeneous type} (\emph{PH type}, for short). Involutive non-degenerate YBMs of PH type will produce structure groupoids, which in turn yield examples of Garside groupoids by Theorem \ref{teo:Garside-structure} and Propositions \ref{prop:is-Ore}, \ref{lemma:Gsigma_is_Env}. At the end of this section, we shall consider explicit examples of YBMs of PH type: our examples will not only be solutions, but also \emph{braidings} on groupoids.
\subsection{Principal homogeneous groupoids}\label{subsec:PH-groupoids}  We begin by introducing solutions of \emph{principal homogeneous type}: these are solutions defined on a complete groupoid of degree 1, which is equivalent to a \emph{groupoid of pairs}; see e.g.\@ \cite[Example 1.11]{introGroupoids}. We moreover recall the notion of braided groupoid, and prove that the datum of a group is equivalent to the datum  of a braided groupoid of pairs with a distinguished vertex.
\begin{defin}\label{def:weak-morph}
A \textit{weak morphism of quivers} $f\colon Q\to R$, where $Q$ and $R$ have possibly different sets of vertices, is a datum of a map of sets $f^1\colon Q\to R$ and a map of sets $f^0 \colon \Obj(Q)\to \Obj(R)$, satisfying $\source_R(f^1(x)) = f^0(\source_Q(x))$, $\target_R(f^1(x)) = f^0(\target_Q(x))$ for any $x\in Q$.
Here,  $\Obj(Q)$ means the set of all vertices of the quiver $Q$.

A \emph{weak morphism} 
$f: \Gg\to\Hh$ \emph{of groupoids}, between two groupoids $\Gg, \Hh$ with possibly different sets of vertices, is a weak morphism $(f^1, f^0)$ between the underlying quivers, satisfying moreover $f^1(a\cdot_{\Gg} b) = f^1(a)\cdot_{\Hh} f^1 (b)$ for all $a,b \in \Gg$.
\end{defin}
We denote by $\Cat{Quiv}$ the category of all quivers, regardless of their sets of vertices, equipped with weak morphisms.
\begin{rem}
Notice that a morphism in $\Cat{Quiv}_\Lambda$ is a weak morphism $f = (f^1, f^0)$ with $f^0 = \id[\Lambda]$.
\end{rem}
\begin{defin}
Let $\Lambda$ be a nonempty set, $n$ any cardinal (possibly infinite). A \textit{complete quiver of degree $n$} over $\Lambda$ is a quiver $Q$ over $\Lambda$ such that, for all $a,b\in \Lambda$ (not necessarily distinct), there exist exactly $n$ elements $q\in Q$ with $\source (q)=a$ and $\target(q)=b$. 

A complete quiver of degree $1$ will be called here a \emph{principal homogeneous} (\emph{PH}) \emph{groupoid}, following the nomenclature of \cite{matsumotoshimizu}; and the
category of principal homogeneous groupoids, endowed with weak morphisms, is here denoted by $\mathsf{PHG}$. Here, for all $a,b\in \Lambda$, there exists a unique arrow with source $a$ and target $b$, which we denote by $[a,b]$. For a sequence $a_1,\dots, a_n$ of vertices, the notation $[a_1,\dots, a_n]$ denotes the path $[a_1, a_2] \,|\, [a_2, a_3]\,|\, \dots \,|\, [a_{n-1}, a_n]$. The empty path on $a$ is denoted by $[a]$ (notice that $[a]$ is not an edge of any groupoid of pairs, and $[a]\neq [a,a]$).
\end{defin}
\begin{rem}
Let $\Gg$ be in $\mathsf{PHG}$. Then $\Gg$ is indeed a groupoid, with a \emph{unique} groupoid operation: namely, the operation $\cdot$ defined by $[a,b]\cdot [b,c]= [a,c]$, having units given by the loops $[a,a]$. The multiplication $\cdot$ will also be denoted by $m \colon \Gg\otimes \Gg\to \Gg$, whenever this notation comes most in handy. We denote the inverse of $x\in \Gg$ by $x^{-1}$.
\end{rem}
\begin{defin}[cf.\@ \cite{matsumotoshimizu, shibukawa2007invariance}]\label{PH-type} A (quiver-theoretic)
YBM $\sigma\colon \Gg\otimes \Gg\to \Gg\otimes \Gg$ on a PH groupoid $\Gg$ will be called a solution \emph{of principal homogeneous (PH) type}.
\end{defin}
\begin{rem}\label{rem:blow-up}
Given a set $X$, its \textit{groupoid of pairs} $\hat{X}$ is defined as the PH groupoid on the set of vertices $X$, with the set of arrows given by $X\times X$, and the source and target maps are the projections on the first and the second factor respectively; see e.g.\@ \cite[Example 1.11]{introGroupoids}. Given a map of sets $f\colon X\to Y$, we define $\hat{f} = (\hat{f}^1, \hat{f}^0)\colon \hat{X}\to \hat{Y}$ by $\hat{f}^0 = f\colon X\to Y$, and $\hat{f}^1([x,y])= [f(x), f(y)]$. It is clear that $\hat{f}$ is a weak morphism of quivers. This defines a functor $\hat{(\blank)}\colon \mathsf{Set}\to \mathsf{Quiv}$. 

On the other hand, we can consider a quiver $Q$ and take its set of vertices $\Obj(Q)$. This defines a functor $\Obj\colon \mathsf{Quiv} \to \mathsf{Set}$, where the image of the morphism $g = (g^1, g^0)$: $Q \to R$ under the functor $\Obj$ is defined as the map $g^0$.

Then, $\mathsf{Set}$ is equivalent to the category $\mathsf{PHG}$, via the two functors $\hat{(\blank)}|^{\mathsf{PHG}}$ and $\Obj|_{\mathsf{PHG}}$. Indeed, a morphism $g$ between two PH groupoids is uniquely described by $g^0$, thus it is easy to check that $\Obj|_{\mathsf{PHG}}\circ \hat{(\blank)}|^{\mathsf{PHG}} = \id[\mathsf{Set}]$, while $\hat{(\blank)}|^{\mathsf{PHG}}\circ\Obj|_{\mathsf{PHG}}$ is canonically isomorphic to $\id[\mathsf{PHG}]$.
\end{rem}
We recall the following definition from Andruskiewitsch \cite{andruskiewitsch2005quiver}. Here, for a groupoid $\Gg$ and a map $f$ defined on the tensor product $\Gg^{\otimes 2}$ (defined as the tensor product of the underlying quivers), 
the well-established notation
$f_{i\;i+1}$ $(i=1, 2)$ means the map on $\Gg^{\otimes 3}$ defined by applying $f$ on the $i$-th and 
($i+1$)-st component, and applying the identity on the other tensor component.
\begin{defin}\label{braided_groupoid}
A \emph{braided groupoid} is the datum of a groupoid $\Gg$, with vertices $\Lambda= \Obj(\Gg)$, a multiplication given by a morphism $m\colon \Gg\otimes \Gg\to \Gg$ of quivers over $\Lambda$, and a family of units $\set{\mathbf{1}_\lambda}_{\lambda\in\Lambda}$; and of an isomorphism $\sigma\colon \Gg\otimes \Gg\to \Gg\otimes \Gg$ of quivers over $\Lambda$, called a \emph{braiding}, written as $\sigma(x,y) = (x\rightharpoonup y, x\leftharpoonup y)$,  satisfying the following properties for all $x,y,z$ such that the path $x|y|z$ is defined:
\begin{align}
	&\label{bg1}\tag{BG1}\sigma(x, \mathbf{1}_{\target(x)})=(\mathbf{1}_{\source(x)}, x);\\
	&\label{bg2}\tag{BG2}\sigma(\mathbf{1}_{\source(x)}, x)= (x, \mathbf{1}_{\target(x)});\\
	&\label{bg3}\tag{BG3}\sigma \circ m_{23} = m_{12}\circ \sigma_{23}\circ \sigma_{12},\\\nonumber &\text{i.e.\@ }x\rightharpoonup yz = (x\rightharpoonup y)((x\leftharpoonup y)\rightharpoonup z)\text{ and }x\leftharpoonup(yz)=(x\leftharpoonup y)\leftharpoonup z;\\
	&\label{bg4}\tag{BG4} \sigma \circ m_{12} = m_{23}\circ \sigma_{12}\circ \sigma_{23},\\&\nonumber \text{i.e.\@ }xy\leftharpoonup z = (x\leftharpoonup(y\rightharpoonup z))(y\leftharpoonup z)\text{ and }(xy)\rightharpoonup z=x\rightharpoonup(y\rightharpoonup z);\\
	&\label{braided-commutative}\tag{BG5}m\circ \sigma = m.
\end{align}
\end{defin}
\begin{rem}
A morphism $\sigma\colon \Gg\otimes\Gg\to\Gg\otimes\Gg$ of quivers over $\Lambda$ satisfying \eqref{bg1}--\eqref{braided-commutative} is necessarily invertible. Indeed (with the same proof as in \cite[Theorem 1]{LYZ}), let $x|y = \sigma(u|v) = (u\rightharpoonup v)|(u\leftharpoonup v)$. The arrows $u\rightharpoonup (vv^{-1})$ and $u\leftharpoonup (vv^{-1})$ are well-defined, since $vv^{-1} = \mathbf{1}_{\target(u)}$, and from \eqref{bg1} we get $u\leftharpoonup (vv^{-1}) = u$. Thus $y\leftharpoonup v^{-1} = (u\leftharpoonup v)\leftharpoonup v^{-1} = u$, whence
\[ (y\rightharpoonup v^{-1}) u =  (y\rightharpoonup v^{-1})(y\leftharpoonup v^{-1})  \overset{\eqref{braided-commutative}}{=} yv^{-1} \overset{\eqref{braided-commutative}}{=} x^{-1} u, \]
which, by cancelling $u$ on both sides, implies $v^{-1} = y^{-1}\rightharpoonup x^{-1}$. This allows us to retrieve $v$ from $x$ and $y$. At this point, using \eqref{braided-commutative}, it is also immediate to retrieve $u$ as $u= xyv^{-1}= xy (y^{-1}\rightharpoonup x^{-1})$. This proves that $\sigma$ is invertible. We invite the reader to keep track of the sources and targets, and check that all the above operations are well-defined.
\end{rem}
Because of \eqref{bg1} and \eqref{bg3}, $\leftharpoonup$ is a right action $\Gg\otimes \Gg\to \Gg$; and because of \eqref{bg2} and \eqref{bg4},
$\rightharpoonup$ is a left action $\Gg\otimes \Gg\to \Gg$, in the sense of \cite{andruskiewitsch2005quiver}. 

Condition \eqref{braided-commutative} is called the \emph{braided-commutativity} of $m$ with respect to $\sigma$. A graphical interpretation of  \eqref{bg3} and \eqref{bg4} is given in Figure \ref{fig:hex_prism}.
\begin{figure}[t]
\centering
\begin{subfigure}{0.4\linewidth}\centering
	\begin{tikzpicture}[x=1.5cm, y=1cm]
		\draw[-Stealth] (-1.5,4) to node[above]{$x\rightharpoonup y$} (1.5,4);
		\draw[-Stealth] (-1.5,4) to node[sloped,below]{$\;(x\rightharpoonup y) ((x\leftharpoonup y)\rightharpoonup z) $} (0,3);
		\draw[-Stealth] (-1.5,4) to node[sloped,above]{$x$} (-1.5,1);
		\draw[-Stealth] (-1.5,1) to node[sloped, above]{$y\quad \quad\qquad \qquad$} (1.5,1);
		\draw[-Stealth] (1.5,4) to node[sloped,above]{$(x\leftharpoonup y)\rightharpoonup z\quad\;\; $} (0,3);
		\draw[-Stealth] (0,3) to node[sloped,below]{$\!\!x\leftharpoonup y z$} (0,0);
		\draw[-Stealth] (-1.5,1) to node[sloped,below]{$y z$} (0,0);
		\draw[-Stealth] (1.5,1) to node[sloped,above]{$z$} (0,0);
		\draw[-Stealth] (1.5,4) to node[sloped, below]{$x\leftharpoonup y$} (1.5,1);
		\node[] at (0,-0.65) {};
	\end{tikzpicture}\caption{The axiom \ref{bg3}}\end{subfigure}\begin{subfigure}{0.4\linewidth}\centering\begin{tikzpicture}[x=1.5cm, y=1cm]
		\draw[-Stealth] (-1.5,4) to node[above]{$x$} (1.5,4);
		\draw[-Stealth] (-1.5,4) to node[sloped,below]{$x y $} (0,3);
		\draw[-Stealth] (-1.5,4) to node[sloped,above]{$x y\rightharpoonup z$} (-1.5,1);
		\draw[-Stealth] (-1.5,1) to node[sloped, above]{$ x\leftharpoonup (y\rightharpoonup z)\qquad\qquad\qquad $} (1.5,1);
		\draw[-Stealth] (1.5,4) to node[sloped,below]{$ y $} (0,3);
		\draw[-Stealth] (0,3) to node[sloped,above]{$ z$} (0,0);
		\draw[-Stealth] (-1.5,1) to node[sloped,below]{$(x\leftharpoonup (y\rightharpoonup z)) (y\leftharpoonup z) $} (0,0);
		\draw[-Stealth] (1.5,1) to node[sloped,above]{$y\leftharpoonup z$} (0,0);
		\draw[-Stealth] (1.5,4) to node[sloped, above]{$y\rightharpoonup z$} (1.5,1);
	\end{tikzpicture}\caption{ The axiom \ref{bg4}}\end{subfigure}
\caption{The axioms \ref{bg3} and \ref{bg4} as the closure of a prism.}\label{fig:hex_prism}
\end{figure}
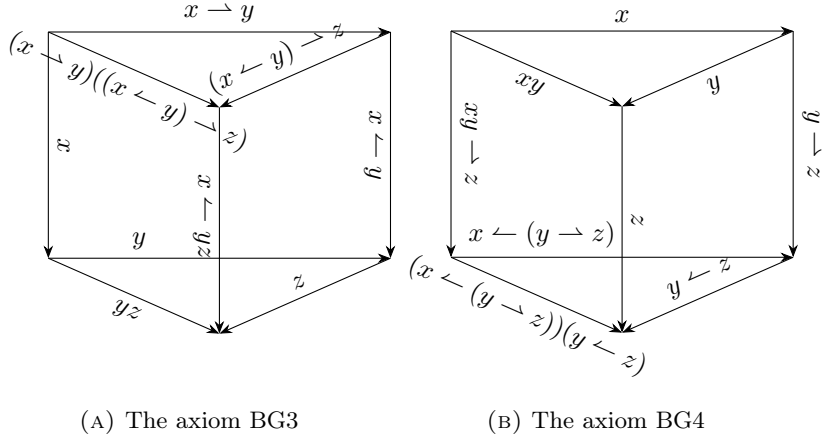

It is well known that braidings $\sigma$ on groupoids are solutions to the YBE \cite{andruskiewitsch2005quiver}. The solution corresponding to a braided groupoid is left- and right-non-degenerate; for example, if $x\rightharpoonup y=z$, then $y=x^{-1}\rightharpoonup z$,
because $\rightharpoonup$ is a left action.

In the principal homogeneous case, an arrow of $\Gg$ is uniquely determined by its source and its target. Therefore, $\sigma\colon \Gg\otimes\Gg\to\Gg\otimes\Gg$ is uniquely determined by a ternary operation $\langle\blank,
\blank,\blank\rangle\colon \Lambda\times\Lambda\times\Lambda\to\Lambda$, by imposing
\begin{equation}\label{eq:heap}\sigma[a,b,c] = [a,\langle a,b,c\rangle, c].\end{equation}
\begin{prop}[{\cite[Theorem 3.2 and Proposition 7.1]{shibukawa2007invariance}}]\label{prop:ternary-op-and-YBM}A map $\sigma$ defined as in \eqref{eq:heap} is a YBM on $\Gg$ if and only if the ternary operation satisfies
\begin{align*}
	\langle a, \langle a,b,c\rangle , \langle \langle a,b,c\rangle, c, d\rangle  \rangle = &\;\langle a, b, \langle b,c,d\rangle\rangle,\\
	\langle  \langle a, b, \langle b,c,d\rangle\rangle, \langle b,c,d\rangle, d\rangle =& \;\langle \langle a,b,c\rangle, c, d\rangle,
\end{align*}
for all $a,b,c\in\Lambda = \Obj(\Gg)$. The map $\sigma$ is involutive if and only if 
\begin{equation}\label{eq:ternary-op-involutivity}\langle a,\langle a,b,c\rangle, c\rangle = b \end{equation}
holds for all $a,b,c\in\Lambda$. It is non-degenerate if and only if the equation $\langle a,b,c\rangle = b'$ can be solved uniquely in the variable $a$ for all $b,b',c \in\Lambda$; and be solved uniquely in the variable $c$ for all $a,b,b'\in\Lambda$.
\end{prop}

The following definition appeared in Pr\"ufer \cite{prufer1924theorie}, then (without the assumption of abelianity) in Baer \cite{baer1929einfuhrung}, and was generalised by Wagner \cite{wagner1953}.
\begin{defin}
A \emph{heap}\footnote{The original German name was \textit{(die) Schar}, plur.\@ \textit{Scharen}, which translates better as ``crowd'', ``herd'', or ``flock''. In Russian (e.g.\@ in the works of Wagner \cite{wagner1953}), it was translated as груда. For the English name currently in use, we refer to Hollings and Lawson \cite{hollings2017wagner}, and Brzezi\'nski \cite{BRZEZINSKITrussesParagons}. Notably, among the other English translations proposed or used, there was also \textit{principal homogeneous space} (see \cite{hollings2017wagner}), which foreshadows our correspondence in Theorem \ref{ternary_operation_and_braiding}.} is the datum of a set $\Lambda$, and a ternary operation $\langle\blank,\blank,\blank\rangle$ on $\Lambda$ satisfying, for all $a,b,c,d\in \Lambda$:
\begin{align}
	&\label{maltsev_1}\tag{M1} \langle a,b,b\rangle = a,\\
	&\label{maltsev_2}\tag{M2} \langle a,a,b\rangle = b,\\
	&\label{associativity}\tag{A} \langle a,b,\langle c,d,e\rangle\rangle = \langle \langle a,b,c\rangle, d,e\rangle.
\end{align}
A heap is called \emph{abelian} if $\langle a,b,c\rangle = \langle c,b,a\rangle$ for all $a,b,c\in\Lambda$. 
\end{defin}
Conditions \eqref{maltsev_1} and \eqref{maltsev_2} are called \emph{Mal'tsev conditions}, while \eqref{associativity} is called the \emph{associativity condition}. The reader interested in how these conditions arose, and were integrated in the theory of \emph{Mal'tsev categories}, may refer to \cite{BournGranJacqmin, carboni1991Maltsev, MaltsevOriginal} and references therein. A \emph{pointed heap}, i.e.\@ a heap with a distinguished element, is the same thing as a group; see Baer \cite{baer1929einfuhrung}, or \cite[Lemma 2.1]{BRZEZINSKITrussesParagons} for a contemporary formulation. The idea of using heaps as an ``affine version'' of groups was applied in several places; see e.g.\@ 
\cite{baer1929einfuhrung,breaz2024heaps,BRZEZINSKITrussesParagons,brzezinski2024affgebras,prufer1924theorie,wagner1953}.

\begin{rem}\label{rem:weakheap}
The associativity condition implies
\begin{align}
	\label{associativity_1}\tag{A1}&\langle a,b,d\rangle = \langle \langle a,b,c\rangle, c, d\rangle,\\
	\label{associativity_2}\tag{A2}&\langle a,c,d\rangle = \langle a, b, \langle b,c,d\rangle\rangle.
\end{align}
This follows immediately from \eqref{associativity} and the Mal'tsev conditions. Conversely, \eqref{associativity_1} and \eqref{associativity_2} together with the Mal'tsev conditions imply \eqref{associativity} by an immediate computation; see e.g.\@ \cite[Proposition 7]{BournGranJacqmin}.
\end{rem}
\begin{lem}\label{ternary_operation_bijective}
Let $(\Gg,\sigma)$ be a braided PH groupoid, associated with the ternary operation $\langle\blank,\blank,\blank\rangle$ on $\Lambda = \Obj(\Gg)$. Then:
\begin{enumerate}
	\item[\textit{i.}] The map $\langle a,b,\blank\rangle$ is invertible for all $a,b\in\Lambda$, and the inverse is $\langle b,a,\blank\rangle$.
	\item[\textit{ii.}] The map $\langle \blank, b,c\rangle$ is invertible for all $b,c\in\Lambda$, and the inverse is $\langle \blank, c,b\rangle$.
	\item[\textit{iii.}] If $\sigma$ is moreover involutive, then $\langle\blank,b,c\rangle$ has also the inverse $\langle b,c,\blank \rangle$; thus $\langle c,b,\blank\rangle = \langle \blank, b,c\rangle$ holds.
\end{enumerate}
\end{lem}
\begin{proof}
Let $x = [a,b], y=[b,c]$. One has
\begin{equation}\label{eq:yields-associativity}\begin{split}
		\sigma(x^{-1}\;|\;x\rightharpoonup y) &= x^{-1}\rightharpoonup (x\rightharpoonup y) \;|\; x^{-1}\leftharpoonup (x\rightharpoonup y)\\
		&= y\;|\;x^{-1}\leftharpoonup (x\rightharpoonup y)\\
		&\hspace{-1.3pt}\overset{(\dagger)}{=} y\;|\;(x\leftharpoonup y)^{-1},
\end{split}\end{equation}
where the equality marked with $(\dagger)$ follows from 
\[
(x^{-1}\leftharpoonup (x\rightharpoonup y)) (x\leftharpoonup y) \overset{\eqref{bg4}}{=} (x^{-1} x)\leftharpoonup y = 1\leftharpoonup y \overset{\eqref{bg2}}= 1.
\]
Since \eqref{eq:yields-associativity} is an equation on paths, it implies the equality of the middle vertices: thus we get $c = \langle b,a,\langle a,b,c\rangle\rangle$, which proves \textit{i}. The proof of \textit{ii} is analogous. If $\sigma$ is involutive, \eqref{eq:yields-associativity} moreover implies $\sigma(y\;|\;(x\leftharpoonup y)^{-1}) = x^{-1}\;|\; x\rightharpoonup y$, whence $a = \langle b,c,\langle a,b,c\rangle\rangle$, as desired. 
On the other hand, $y=(y\rightharpoonup(xy)^{-1})\rightharpoonup(y\leftharpoonup(xy)^{-1})$ by \eqref{i1}. Because $\rightharpoonup$ is a left action, we get $(y\rightharpoonup(xy)^{-1})^{-1}\rightharpoonup y=y\leftharpoonup(xy)^{-1}$, and $\langle\langle b, c, a\rangle, b, c\rangle=a$ as a result.
Therefore, one has $\langle b,c,\blank\rangle = \langle \blank,b,c\rangle^{-1}$, but also $\langle b,c,\blank\rangle = \langle c,b,\blank\rangle^{-1}$, whence $\langle c,b,\blank\rangle = \langle \blank, b,c\rangle$. This concludes the proof of \textit{iii}.
\end{proof}

From the previous lemma and Proposition \ref{prop:ternary-op-and-YBM}, the braiding of every braided PH groupoid is always non-degenerate.
\begin{prop}[{see e.g.\@ \cite[Lemma 2.1]{BRZEZINSKITrussesParagons}}]\label{lem:group_iff_pointed_heap}
Let $\Lambda$ be a set. The following data are equivalent:
\begin{enumerate}
	\item[\textit{i.}] A group operation $*$ on $\Lambda$.
	\item[\textit{ii.}] A heap structure on $\Lambda$, and a distinguished element $u\in\Lambda$.\end{enumerate}
Via this correspondence, abelian groups correspond to ternary operations satisfying \eqref{eq:ternary-op-involutivity}.\end{prop}
The correspondence is as follows. A group operation $*$ on $\Lambda$ defines a heap structure on $\Lambda$ by
$\langle a, b, c\rangle=a*b^{-*}*c$ $(a, b, c\in\Lambda)$. Here, $b^{-*}$ is the inverse of $b$ with respect to the group operation $*$. As the distinguished element $u$, we take the unit of the group $\Lambda$. Conversely, a heap structure on $\Lambda$ with distinguished element $u\in\Lambda$ can define a group structure $*$ on $\Lambda$ whose unit is $u$ by
$a*b=\langle a, u, b\rangle$ $(a, b\in\Lambda)$.

The following is the main result in this section:
\begin{thm}\label{ternary_operation_and_braiding} Let $\Lambda$ be a set, and we denote by $\Gg = \hat{\Lambda}$ the corresponding groupoid of pairs (Remark \ref{rem:blow-up}). The following data are equivalent:
\begin{enumerate}
	\item[\textit{i.}] A heap structure on $\Lambda$.\item[\textit{ii.}] A braiding $\sigma$ on $\Gg$.
\end{enumerate}
Via this correspondence, ternary operations satisfying \eqref{eq:ternary-op-involutivity} correspond to involutive braidings.
\end{thm}
\begin{proof}\hspace{-4pt}\footnote{The equivalence emerged after a conversation that the first author (DF) had with Marino Gran, who suggested that the Mal'tsev and associativity conditions could produce solutions to the YBE.}
We define a braiding $\sigma$ on $\Gg$ by \eqref{eq:heap} by means of the heap structure on $\Lambda$ and vice versa. Observe that \eqref{maltsev_1} corresponds to \eqref{bg1}, and \eqref{maltsev_2} corresponds to \eqref{bg2}. Immediate computations show that \eqref{associativity_1} corresponds to \eqref{bg3}, and \eqref{associativity_2} corresponds to \eqref{bg4}. It remains to observe that \eqref{braided-commutative} is always satisfied: $m \sigma([a,b,c]) = m([a,b,c])$ because both terms are forced to be the unique arrow with source $a$ and target $c$.

We know from Proposition \ref{prop:ternary-op-and-YBM} that $\sigma$ is involutive if and only if the ternary operation satisfies \eqref{eq:ternary-op-involutivity}.
\end{proof}
\begin{cor}\label{ternary_operation_and_left_quasigroup} Let $\Lambda$ be a set, $\Gg = \hat{\Lambda}$ the corresponding groupoid of pairs. The following data are equivalent:
\begin{enumerate}
	\item[\textit{i.}] A group operation $*$ on $\Lambda$.
	\item[\textit{ii.}] A pointed heap structure on $\Lambda$: i.e., the datum of a heap structure, and of a distinguished element $u\in\Lambda$.
	\item[\textit{iii.}] A braiding $\sigma$ on $\Gg$, and a distinguished vertex $u\in \Lambda$.
\end{enumerate}
Via this correspondence, abelian groups correspond to involutive braidings, and thus to ternary operations satisfying \eqref{eq:ternary-op-involutivity}.
\end{cor}
Let $\Cat{Gp}$ be the category of groups. Let $\Cat{BrPHG}$ be the category of braided PH groupoids $(\Gg,\sigma)$, with morphisms given by the weak morphisms of groupoids $f = (f^1, f^0)\colon\Gg\to\Hh$ satisfying $(f^1\times f^1)\sigma_\Gg = \sigma_\Hh (f^1\times f^1)$; and let $\Cat{BrPHG}^\bullet$ be the category of braided PH groupoids with a distinguished vertex $u$, and morphisms given by the morphisms $f=(f^1, f^0)\colon \Gg\to \Hh$ in $\Cat{BrPHG}$ satisfying $f^0(u_\Gg) = u_\Hh$. Let $\Cat{Hp}$ be the category of heaps $(H,\xi)$, where $H$ is a set and $\xi\colon H\times H\times H\to H$ is a ternary operation; with morphisms given by the maps $f\colon H\to K$ satisfying $f \xi_H= \xi_K (f\times f\times f)$. Finally, let $\Cat{Hp}^*$ be the category of pointed heaps, with puncture-preserving heap morphisms. We now prove that the correspondence of Corollary \ref{ternary_operation_and_left_quasigroup} is functorial, and thus provides isomorphisms of categories.
\begin{prop}\label{prop:heap_category_equiv}
Let $f\colon \Lambda\to \Lambda'$ be a map of sets. We write two given group structures on $\Lambda$ and $\Lambda'$ as $(\Lambda, *, u)$ and $(\Lambda', *', u')$ respectively, and suppose the ternary operations $\langle\blank,\blank,\blank\rangle$, $\langle\blank,\blank,\blank\rangle'$ and the braiding $\sigma, \sigma'$ are defined as in Corollary \ref{ternary_operation_and_left_quasigroup}. The following conditions are equivalent:
\begin{enumerate}
	\item[\textit{i.}] The map $f\colon \Lambda\to \Lambda'$ is a group homomorphism.
	\item[\textit{ii.}] The map $f\colon \Lambda\to \Lambda'$ is a morphism of pointed heaps.
	\item[\textit{iii.}] The map $\hat{f}\colon \hat{\Lambda}\to \hat{\Lambda}'$ defined as in Remark \ref{rem:blow-up} is a weak morphism of braided groupoids, with $\hat{f}^0(u) = u'$.
\end{enumerate}
In particular, the correspondence of Corollary \ref{ternary_operation_and_left_quasigroup} is functorial. Therefore, the categories $\Cat{Hp}^*$ and $\Cat{Gp}$ are isomorphic, and both are equivalent to the category $\Cat{BrPHG}^\bullet$.
\end{prop}
\begin{proof}
Let $f$ be a group homomorphism; then clearly $f\langle a,b,c\rangle = f(a*b^{-*}*c) = f(a)*f(b)^{-*}*f(c)=\langle f(a), f(b), f(c) \rangle'$, and $f(u) = u'$, which proves \textit{i}$\Rightarrow$\textit{ii}. 

As for \textit{ii}$\Rightarrow$\textit{iii}, let $f$ be a morphism of pointed heaps. We already know that $\hat{f} = (\hat{f}^1, \hat{f}^0)$ is a weak morphism of groupoids, and that $f = \hat{f}^0$ satisfies $f(u) = u'$, thus we only need to check that $\hat{f}^1$ intertwines the two braidings $\sigma$ and $\sigma'$. One has
\begin{align*}
	(\hat{f}^1\otimes \hat{f}^1)\sigma([a,b,c])&= (\hat{f}^1\otimes \hat{f}^1)([a, \langle a,b,c\rangle, c])\\
	&= [f(a), f(\langle a,b,c\rangle), f(c)]\\
	&= [f(a), \langle f(a),f(b),f(c)\rangle', f(c)]\\
	&= \sigma' ([f(a), f(b), f(c)])\\
	&= \sigma' (\hat{f}^1\otimes \hat{f}^1)([a,b,c]),
\end{align*}
as desired. Finally, assuming \textit{iii}, one has 
\begin{align*}
	f(a*b)&= \langle f(a), f(u), f(b)\rangle'\\
	&= \langle f(a), u', f(b)\rangle'\\
	&= f(a)*' f(b),
\end{align*}
which proves \textit{iii}$\Rightarrow$\textit{i}. This concludes the proof that the correspondence of Corollary \ref{ternary_operation_and_left_quasigroup} is functorial. 

It is known (see \cite{BRZEZINSKITrussesParagons}) that $\mathsf{Hp}^*$ and $\mathsf{Gp}$ are isomorphic. Let $F\colon \mathsf{Hp}^*\to \mathsf{BrPHG}^\bullet$ be the functor sending a pointed heap $(\Lambda, u)$ to the braided groupoid of pairs $\hat{\Lambda}$ with distinguished vertex $u$; and let $G\colon \mathsf{BrPHG}^\bullet\to \mathsf{Hp}^*$ be the functor sending a braided groupoid of pairs $\Gg$ over $\Lambda$ with distinguished vertex $u$, to the associated pointed heap $(\Lambda,u)$. Similarly to Remark \ref{rem:blow-up}, it is immediate to see that $G \circ F = \id[\mathsf{Hp}^*]$, while $F\circ G$ is isomorphic to $\id[\mathsf{BrPHG}^\bullet]$. Indeed, for all braided PH groupoid $\Gg$ over $\Lambda$, with distinguished vertex $u\in\Lambda$, the pair of maps $(f^1, f^0)$, defined by $f^1\colon \Gg\ni x\mapsto (\source(x), \target(x))\in \hat{\Lambda}$ and $f^0  = \id[\Lambda]$, makes a natural isomorphism in $\mathsf{BrPHG}^\bullet$ between $\Gg$ and the braided groupoid of pairs $\hat{\Lambda}$ with distinguished vertex $u$. Therefore, $F$ and $G$ yield an equivalence of categories.
\end{proof}
\begin{rem}
In Proposition \ref{prop:heap_category_equiv}, the equivalence \textit{i}$\Leftrightarrow$\textit{ii} is well-known; see e.g.\@ \cite{baer1929einfuhrung}, or \cite[Lemma 2.1]{BRZEZINSKITrussesParagons}.\end{rem}
\begin{rem}
Notice that the category $\Cat{BrPHG}^\bullet$ of braided PH groupoids with a distinguished vertex is \emph{not} the category of braided \emph{pointed} PH groupoids: the latter is instead the category of braided PH groupoids with a distinguished \emph{arrow}.
\end{rem}
\subsection{Examples of structure groupoids in the principal homogeneous case} 
In this section, we present concrete examples of structure groupoids $\Gg(\sigma)$, when $\sigma$ is an involutive YBM of PH type.

Let $A$ be an abelian group, and let $\Aa = \hat{A}$ denote, as above, the groupoid of pairs on the set of vertices $A$. The ternary operation $\langle a,b,c\rangle = a-b+c$ is an abelian heap structure, and the map $\sigma$ sending $[a,b,c]$ to $[a, a-b+c, c]$ is an involutive braiding (and in particular a YBM) on $\Aa$, by Corollary \ref{ternary_operation_and_left_quasigroup}. By Lemma \ref{ternary_operation_bijective}, one has that $\sigma$ is also non-degenerate.

The cardinality of $\Aa(a,A)$ is constant for all $a\in A$, and equals the cardinality of $A$. Thus by Corollary \ref{cor:bounded}, the Garside family $E$ of $\Cc(\sigma)$ has bounded length, where the upper bound is $|A|$. In particular, $E\subsetneq \Cc(\sigma)$.
\begin{ex}\label{ex:Z3}
Let $A = \Z/3\Z$. The map $\sigma$ defined above
acts as follows:
\begin{align*}
	&[a,b,b]\mapsto [a,a,b],\; [a,a,b]\mapsto [a,b,b] \text{ for all }a,b\in A,\\
	&[a,b,a]\mapsto [a,2a-b,a]\text{ for all }a,b\in A,\\
	&[a,b,c]\mapsto [a,b,c]\text{ if }a,b,c\text{ are all distinct.}
\end{align*}
Therefore, the structure groupoid is generated by a complete quiver of degree $1$ on three vertices, with the relations
\begin{align*}
	& [a,b,b]\sim[a,a,b] \text{ for all }a,b\in A,
	\qquad [a,b,a]\sim [a,2a-b,a]\text{ if }a\neq b.
\end{align*}
Notice that $2a-b\neq b$ if and only if $a\neq b$, thus the second class of relations contains no redundancies. 

We now denote by $[[a_1,\dots, a_n]]$ the equivalence class in $\Cc(\sigma)$ of a path $[a_1, \dots, a_n]$. The complementation $\star$ is given by 
\begin{align*}
	&[[a,b]]\star [[a,c]] = [[b, b-a+c]] \text{ when }b\neq c,
\end{align*}
and it is easy to see by direct computation that $b-a+c = a$ whenever $a,b,c$ are distinct. From Proposition \ref{prop:Deltas},
$E=\{ \mathbf{1}_a, [[a, a]], [[a, b]], [[a, a, b]], [[a, b, a]], [[a, a, b, a]]\mid a, b\in\Aa \text{ distinct}\}$. Observe that the Garside family $E$ of $\Cc(\sigma)$ is the union of the sets $\mathrm{Div}_L(\Delta_a)$ of left-divisors of $\Delta_a$, for $a\in A$; and if $a,b,c$ are three distinct vertices, by the previous considerations one has $\Delta_a = [[a,b,b,a]] = [[a,a,b,a]] = [[a,a,c,a]] =[[a,c,c,a]]=[[a, b, a, a]]=[[a, c, a, a]]$. In particular, $\Delta_a$ is a loop for all $a$.
\end{ex}
\begin{ex}
Let $A = (\Z/2\Z)^n$. An element of $A$ will be denoted by a row vector of $0$'s and $1$'s. For $a = (a_1,\dots, a_n), b = (b_1,\dots, b_n)\in A$, we define $\delta_{a,b}= (\delta_{a_1, b_1},\dots, \delta_{a_n, b_n})$, where $\delta_{a_i, b_i}$ is Kronecker's delta symbol, and $\mathbf{1}=(1, 1, \ldots, 1)$.

It is easy to see that
$$\sigma([a,b,c]) = [a, b+\mathbf{1}+\delta_{a,c}, c].$$
Indeed, $\sigma([a,b,c])_i = [a_i, a_i-b_i+c_i, c_i] = [a_i,a_i+b_i+c_i, c_i]$. If $a_i = c_i$ then $a_i+b_i+c_i= b_i+2a_i = b_i$, otherwise $a_i+c_i = 1$ and $a_i+b_i+c_i = b_i+1$.

Thus, the structure groupoid of $\sigma$ is generated by a complete quiver of degree $1$ on $2^n$ vertices, modulo the relations $[a,b,c]\sim [a, b+\mathbf{1}+\delta_{a,c}, c]$. Using the same notation as in Example \ref{ex:Z3} for the equivalence class of a path, the complementation $\star$ is
$$[[a,b]]\star [[a,c]] = [[b, a+\mathbf{1}+\delta_{b,c}]],$$
for $b\neq c$. Indeed, $[[a,b]]\star [[a,c]]$ equals $[[b,d]]$ for a unique path $[b,d]$ such that $[a, a+b+d, d] = [a,c,d]$; thus we need $a+b+d = c$. This implies $d = c+a+b$, but this in turn is $a+\mathbf{1}+\delta_{b,c}$, as we proved above.
\end{ex}
\bibliographystyle{acm}
\bibliography{refs}
\end{document}